\documentclass[a4paper]{amsart}
\usepackage{color,stmaryrd}
\usepackage{hyperref}
\usepackage{amssymb}
\usepackage{amscd}
\usepackage{amsthm}
\usepackage{amsmath}
\usepackage{latexsym}
\usepackage[all]{xy}
\usepackage{wasysym}
\usepackage{mathrsfs}
\usepackage{mathtools}
\usepackage[normalem]{ulem}

\theoremstyle{plain}
\newtheorem{theorem}{Theorem}
\newtheorem{corollary}[theorem]{Corollary}
\newtheorem{lemma}[theorem]{Lemma}
\newtheorem{proposition}[theorem]{Proposition}

\theoremstyle{definition}
\newtheorem{definition}[theorem]{Definition}
\newtheorem{example}[theorem]{Example}

\theoremstyle{remark}

\newtheorem{remark}[theorem]{Remark}

\numberwithin{theorem}{section}

\numberwithin{equation}{theorem}

\makeatletter
\renewcommand{\p@enumii}{}
\makeatother

\newcommand{\Loc}{\mbox{\rm{Loc\,}}}

\newcommand{\Spec}[1]{\mbox{\rm{Spec}}\,#1}

\newcommand{\Add}{\mbox{\rm{Add\,}}}
\newcommand{\Prod}{\mbox{\rm{Prod\,}}}

\newcommand{\Gen}{\mbox{\rm{Gen\,}}}
\newcommand{\gen}{\mbox{\rm{gen\,}}}
\newcommand{\Cogen}{\mbox{\rm{Cogen\,}}}

\newcommand{\End}{\mbox{\rm{End\,}}}

\newcommand{\im}{\mbox{\rm{Im\,}}}
\newcommand{\ProjA}{\ensuremath\mbox{\rm{Proj}-$A$}}

\newcommand{\AInj}{\ensuremath\mbox{$A$-\rm{Inj}}}
\newcommand{\projA}{\ensuremath\mbox{\rm{proj}-$A$}}

\newcommand{\hocolim}{\mbox{\rm{hocolim}}}

\newcommand{\Hom}[3]{\mbox{\rm{Hom}}_{#1}(#2,#3)}
\newcommand{\rhom}[3]{\mbox{\rm{\bf R}Hom}_{#1}(#2,#3)}
\newcommand{\lten}[1]{\otimes^{\bf L}_#1}
\newcommand{\Ext}[4]{\mbox{\rm{Ext}}^{#1}_{#2}(#3,#4)}
\newcommand{\Tor}[4]{\mbox{\rm{Tor}}_{#1}^{#2}(#3,#4)}

\newcommand{\rfmod}[1]{\mbox{\rm{mod}-\!}{#1}}
\newcommand{\rmod}[1]{\mbox{\rm{Mod}-\!}{#1}}
\newcommand{\lfmod}[1]{{#1}\mbox{-\rm{mod}}}
\newcommand{\lmod}[1]{{#1}\mbox{-\rm{Mod}}}

\newcommand{\ModA}{\rmod{A}}
\newcommand{\AMod}{\lmod{A}}
\newcommand{\modA}{\rfmod{A}}
\newcommand{\ModB}{\rmod{B}}

\newcommand{\Modk}{\rmod{k}}

\newcommand{\DModA}{\DD({\rm Mod}\text{-}A)}
\newcommand{\DModB}{\DD({\rm Mod}\text{-}B)}
\newcommand{\DAMod}{\DD(A\text{-}{\rm Mod})}
\newcommand{\DModk}{\DD({\rm Mod}\text{-}k)}

\newcommand{\KModA}{\K({\rm Mod}\text{-}A)}
\newcommand{\KAMod}{\K(A\text{-}{\rm Mod})}
\newcommand{\KModk}{\K({\rm Mod}\text{-}k)}

\newcommand{\Supp}[1]{\mbox{\rm{Supp}}\,#1}
\newcommand{\supp}[1]{\mbox{\rm{supp}}\,#1}
\newcommand{\Ker}[1]{\mbox{\rm{Ker}}\,#1}
\newcommand{\Coker}{\mbox{\rm{Coker}}}
\newcommand{\Cone}{\mbox{\rm{Cone}}}

\newcommand{\p}{\mathbf{p}}
\newcommand{\q}{\mathbf{q}}
\newcommand{\tube}{\mathbf{t}}

\newcommand{\A}{\mathcal{A}}

\newcommand{\Acal}{\ensuremath{\mathcal{A}}}

\newcommand{\Fcal}{\ensuremath{\mathcal{F}}}

\newcommand{\Dcal}{\ensuremath{\mathcal{D}}}
\newcommand{\Scal}{\ensuremath{\mathcal{S}}}
\newcommand{\Ucal}{\ensuremath{\mathcal{U}}}
\newcommand{\Vcal}{\ensuremath{\mathcal{V}}}
\newcommand{\Wcal}{\ensuremath{\mathcal{W}}}
\newcommand{\Tcal}{\ensuremath{\mathcal{T}}}
\newcommand{\Kcal}{\ensuremath{\mathcal{K}}}
\newcommand{\Mcal}{\ensuremath{\mathcal{M}}}
\newcommand{\Lcal}{\ensuremath{\mathcal{L}}}
\newcommand{\Ecal}{\ensuremath{\mathcal{E}}}
\newcommand{\Ccal}{\ensuremath{\mathcal{C}}}

\newcommand{\Bcal}{\ensuremath{\mathcal{B}}}
\newcommand{\Ycal}{\ensuremath{\mathcal{Y}}}
\newcommand{\Xcal}{\ensuremath{\mathcal{X}}}

\newcommand{\Z}{\mathbb{Z}}

\newcommand{\V}{\mathrm{V}}

\newcommand{\DD}{\mathrm{D}}
\newcommand{\LL}{\mathbf{L}}

\newcommand{\K}{\mathrm{K}}

\newcommand{\Aop}{A^{\mathrm{op}}}

\begin{document}

\title{Parametrizing torsion pairs in derived categories}
\author{Lidia Angeleri H\"ugel}
\address[Lidia Angeleri H\"ugel]{Dipartimento di Informatica - Settore di Matematica, Universit\`a degli Studi di Verona, Strada le Grazie 15 - Ca' Vignal, I-37134 Verona, Italy} 
\email{lidia.angeleri@univr.it}
\author{Michal Hrbek}
\address[Michal Hrbek]{Institute of Mathematics of the Czech Academy of Sciences, \v{Z}itn\'{a} 25, 115 67 Prague, Czech Republic}
\email{hrbek@math.cas.cz}

\thanks{Lidia Angeleri H\"ugel was partially supported by Istituto Nazionale di Alta Matematica INdAM-GNSAGA. Michal Hrbek was supported by the Czech Academy of Sciences Programme for research and mobility support of starting researchers, project MSM100191801.}

\begin{abstract}
We investigate parametrizations of compactly generated t-structures, or more generally, t-structures with a definable coaisle, in the unbounded derived category $\DModA$ of a ring $A$. To this end, we provide a construction of t-structures from chains in the lattice of ring epimorphisms starting in $A$, which is a natural extension of the construction of compactly generated t-structures from  chains of  subsets of the Zariski spectrum known for the commutative noetherian case.  We also provide constructions  of silting and cosilting objects in $\DModA$. This leads us to classification results over  some classes of commutative rings and over finite dimensional hereditary algebras.
\end{abstract}
\maketitle

\section{Introduction}
Since the seminal work of Gabriel, Hopkins, and Neeman, it is well known that over a commutative noetherian ring $A$ many important structures in the category of modules $\ModA$ and its derived category $\DModA$ are controlled by subsets of the Zariski spectrum. For example, the hereditary torsion pairs in $\ModA$ are parametrized by the specialization-closed subsets of $\Spec(A)$, that is, by unions of Zariski-closed subsets. Similarly, it was shown by Alonso, Jerem\'ias and Saor\'in   in \cite{AJSa} that the compactly generated t-structures in  $\DModA$ are parametrized by descending chains of specialization-closed subsets of $\Spec(A)$.

The aim of this paper is to interpret these results from the viewpoint  of silting theory and to establish similar results over further classes of rings, notably rings of weak global dimension at most one.

Silting theory is a useful tool to study decompositions and localizations of categories both at  abelian and triangulated level. Indeed, the silting  objects in the derived category $\DModA$ of a ring $A$ correspond bijectively to certain TTF triples $(\Ucal, \Vcal, \Wcal)$ consisting of a t-structure with an adjacent co-t-structure. Dual results hold for cosilting objects, implying for example that  every  compactly generated TTF triple which is non-degenerate corresponds to a pure-injective cosilting object. There are also abelian versions of these results. 
Indeed, (co)silting modules, which are by definition the zero cohomologies  of  (co)silting complexes of length two, correspond bijectively to certain  torsion pairs in the module category $\ModA$. 

In previous work \cite{AH}, we have already
seen that, over a commutative noetherian ring $A$, (co)silting modules  are in bijection with  hereditary torsion pairs.  The classification result from \cite{AJSa} mentioned above then shows that every
 compactly generated t-structure in $\DModA$
  encodes a sequence of nested cosilting torsion pairs in $\ModA$.
 In Theorem~\ref{commnoeth}, we determine the conditions ensuring that this sequence gives rise to a cosilting complex in $\DModA$, obtaining a   parametrization of the pure-injective cosilting objects over $A$ in terms of  chains of subsets of $\Spec(A)$.
 
 An essential ingredient for this classification is a result from \cite{HN} stating that all pure-injective cosilting objects over  commutative noetherian rings are of cofinite type, i.e.~they correspond to compactly generated TTF triples. In Theorem \ref{hereditaryutc} it turns out that the same holds true over hereditary rings.

 Inspired by these findings, we proceed to investigate possible parametrizations of cosilting objects over further classes of rings. The idea is to replace chains of subsets of the prime spectrum of $A$ by chains of ring epimorphisms starting in $A$.
 
  Ring epimorphisms  with nice homological properties starting in a given ring $A$ form a complete lattice which is known to be related with silting modules. In \cite{AMV2} it is shown that the homological ring epimorphisms starting in a hereditary ring $A$ are in bijection with minimal silting modules, that is, silting modules satisfying a certain minimality condition. 
 Here we discuss the  connection with cosilting modules. Given an arbitrary ring $A$, we provide a general construction of a cosilting $A$-module from a ring epimorphism $\lambda: A\to B$ which satisfies a certain homological condition (Theorem~\ref{copreco-mincosilting}).  As a consequence, we prove that the homological ring epimorphisms  starting in a ring of weak global dimension at most one, or in a commutative noetherian ring,  are in bijection with a class of cosilting modules   which we call minimal (Corollaries~\ref{weak}, and \ref{flatcomm}). If the ring $A$ is hereditary, then minimal silting and cosilting modules  correspond to each other under a silting-cosilting duality (Corollary~\ref{hered}).
 
 We then  turn to  a  chain $\cdots\lambda_n\le\lambda_{n+1}\cdots$  inside the lattice of ring epimorphisms starting in a given ring $A$. If all $\lambda_n$ satisfy our homological condition, we obtain  a chain of  cosilting classes  which gives rise to a t-structure $(\Ucal,\Vcal)$ and a TTF triple $(\Ucal,\Vcal,\Wcal)$ in the derived category of $A$ (Proposition~\ref{construction}). We show that this construction is a natural extension of the construction of compactly generated t-structures from  chains of  subsets of the Zariski spectrum for the commutative noetherian case.  

The coaisle $\Vcal$ obtained from our construction is a definable subcategory of the derived category, that is, it is determined by a set of morphisms between compact objects. Conversely, when the ring $A$ has weak global dimension at most one, every t-structure $(\Ucal,\Vcal)$  with definable coaisle $\Vcal$ encodes a sequence of nested cosilting classes, and we see that it arises  from a chain of ring epimorphisms according to our construction if and only if all cosilting classes involved are minimal (Theorem~\ref{minimalTTF}). 
We also provide a dual construction and 
determine the conditions ensuring that the TTF triples we obtain are induced by a cosilting or a silting object, respectively. Again, if $A$ is hereditary, such silting and cosilting objects  will be related to each other by  a silting-cosilting duality. 

 Finally, we apply our investigations to specific classes of rings. We provide classification results over commutative rings of weak global dimension  at most one and over semihereditary rings (Subsection~\ref{commweak}).  When   $A$ is a  finite dimensional hereditary algebra over a field, we observe that the compact silting complexes correspond to  finite chains of finite dimensional homological ring epimorphisms $0_A\le \lambda_n\le\ldots\le\lambda_m\le {\rm id}_A$ ({Theorem~\ref{fdher}}).
Then we focus on the case  when $A$ is the path algebra of the Kronecker quiver $\xy\xymatrixcolsep{2pc}\xymatrix{ \bullet \ar@<0.5ex>[r]  \ar@<-0.5ex>[r] & \bullet } \endxy$. In Theorem~\ref{kroneckertstr}
 we give a complete classification of all compactly generated t-structures. We show that the chains of homological ring epimorphisms $\cdots\lambda_n\le\lambda_{n+1}\cdots$ with meet $0_A:A\to 0$ and join ${\rm id}_A$ correspond to silting and cosilting objects, and we give a parametrization of  all pure-injective cosilting complexes and their dual silting complexes (Theorems~\ref{kronecosilt} and~\ref{kronesilt}). Similar results are obtained for the ring $\mathbb Z$ and, more generally, for commutative noetherian rings of Krull dimension at most one (Theorem~\ref{T:dimone}, and Examples~\ref{commnoethcoherent} and~\ref{EX:cosiltgroups}).

 The paper is organized as follows. Section 2 contains some preliminaries. In particular, we review the notions of definability and purity in derived categories and  investigate the role of duality in this context. In Section 3 we discuss  compactly generated TTF triples in derived categories. We establish a duality   between compactly generated TTF triples in $\DModA$ and $\DAMod$ which restricts  to a silting-cosilting duality. Then we focus on the special cases when $A$ is  commutative noetherian or  hereditary. Section 4 is devoted to the connection between cosilting modules and ring epimorphisms. In Section 5 we deal with chains of ring epimorphisms and develop our construction of TTF triples. The classification results mentioned above are established in Section 6.

\section{Preliminaries}

\subsection{Notation}
Throughout this paper, let $A$ be a (unital) ring, $\ModA$ the category of right $A$-modules, and $\modA$ its subcategory of finitely presented modules. All subcategories are supposed to be full and strict. We denote by $\ProjA$ and $\projA$ the class of all projective and of all finitely generated projective right $A$-modules, respectively. For any additive subcategory $\Ccal$ of $\ModA$ we let $\K(\Ccal)$ (resp. $\K^b(\Ccal)$) denote the homotopy category of all complexes (resp. bounded complexes) with coordinates in $\Ccal$. Furthermore, we write $\DModA$ for the unbounded derived category of $\ModA$, and $\DD^c(\ModA)=\K^b(\projA)$ for the subcategory of compact objects in $\DModA$. We refer the reader to \cite{Sp} for the definitions and basic facts about K-projective, K-injective, and K-flat resolutions of complexes.

Given a module $M\in\ModA$,  we denote by $\Add{M}$ the class of all modules which are isomorphic to direct summands of direct sums of copies of $M$, and by $\Gen{M}$ the class of all $M$-generated modules, i.~e.~all epimorphic images of modules in $\Add M$. $\Cogen{M}$ and $\Prod M$ are defined dually.
 
 Given  a subcategory $\Ccal$ of $\rmod A$
  and a set of non-negative integers $I$ (which is usually expressed by symbols such as  $\geq n$, $\leq n$, or just $n$, with the obvious associated meaning), 
 we denote by
 $$\Ccal^{\perp_I}=\{X\in\ModA \mid \Ext{i}{R}{\Ccal}{X}=0 \text{ for all } i\in I\}$$
  $${}^{\perp_I}\Ccal=\{X\in\ModA \mid \Ext{i}{R}{X}{\Ccal}=0\text{ for all } i\in I\}.$$
 If $\Ccal$ consists of a single module $M$, we just write $M^{\perp_I}$, 
 ${}^{\perp_I}M$, etc.

We use a similar notation in the derived category $\DModA$.
Given a class of objects $\Xcal$  in $\DModA$ and a set of integers $I$,
 we denote
$${}^{\perp_I}\Xcal:=\{Y\in \DModA \mid \Hom{\DModA}{Y}{X[i]}=0 \text{ for all } X\in\Xcal \text{ and } i\in I\}$$ $${\Xcal}^{\perp_I}:=\{Y\in \DModA \mid \Hom{\DModA}{X}{Y[i]}=0 \text{ for all } X\in\Xcal \text{ and }  i\in I\}.$$

\bigskip

\subsection{Duality}
  We consider the two following dualities on $\DModA$. The first is the functor
$$(-)^* := \rhom{A}{-}{A}: \DModA \rightarrow \DAMod.$$
Recall that $(-)^*$ restricts to a contravariant equivalence between the categories of  finitely generated projective right and left $A$-modules, respectively. By d\'{e}vissage, this further extends to a contravariant equivalence $(-)^*: \DD^c(\ModA) \cong \DD^c(\AMod)$. 

For the second duality we fix a   commutative ring $k$  such that  $A$ is a $k$-algebra, and    an injective cogenerator $W$ in $\rmod \,k$.  For example, one can choose $k=\mathbb Z$ and  $W= \mathbb{Q}/\mathbb{Z}$.
 We denote by   $(-)^+ = \Hom{k}{-}{W}$  the duality functors between $\ModA$ and $\AMod$ and we use the same notation on the derived level: $$(-)^+ := \rhom{k}{-}{W}: \DModA \rightarrow \DAMod.$$
As $W$ is an injective $k$-module, the functor $(-)^+$ is represented by the ordinary Hom-functor $\Hom{k}{-}{W}$. By abusing the notation, we will  use the same notation for the functors defined on the left hand side: 
$$(-)^* := \rhom{\Aop}{-}{A}: \DAMod \rightarrow \DModA.$$
$$(-)^+ := \rhom{k}{-}{W}: \DAMod \rightarrow \DModA.$$

Let $A$ and $B$ be two $k$-algebras. The category of all $A - B$-bimodules is equivalent to the category of all right modules over the ring $B \otimes_k \Aop$. In this way, we define the derived category of all $A - B$-bimodules as {$\DD(A - B) = \DD(\rmod{(B \otimes_k \Aop)})$}. Let $X \in \DModA$, $Y \in \DD(A - B)$, and $Z \in \DModB$. Then we have the adjunction isomorphism in $\rmod{\, k}$:
$$\Hom{\DModB}{X \otimes_A^\LL Y}{Z} \cong \Hom{\DModA}{X}{\rhom{B}{Y}{Z}},$$
as well as its ``enriched'' version in $\DModk$:
\begin{equation}\label{enriched}
\rhom{B}{X \otimes_A^\LL Y}{Z} \cong \rhom{A}{X}{\rhom{B}{Y}{Z}}.
\end{equation}
The following formulas will be useful in the sequel:
\begin{lemma}\label{L:dualityformulas}
	\begin{enumerate}
		\item[(i)] For any $X \in \DModA$ and any $n \in \mathbb{Z}$ we have $H^n(X^+) \cong H^{-n}(X)^+$, and $H^n(X) = 0$ if and only if $H^{-n}(X^+)=0$.
		\item[(ii)] For any $X \in \DModA$ and $Y \in \DAMod$ we have {natural} isomorphisms (in $\DModk$):
			$$\rhom{\Aop}{Y}{X^+} \cong (X \otimes_A^\LL Y)^+ \cong \rhom{A}{X}{Y^+}.$$
		\item[(iii)] For any compact object $S \in \DD^c(\ModA)$ and any complex $X \in \DModA$ we have a {natural} isomorphism
			$$\rhom{A}{S}{X} \cong X \otimes_A^{\LL} S^*.$$
		\item[(iv)] For any compact object $S \in \DD^c(\ModA)$ and any $X \in \DModA$ we have a {natural} isomorphism
			$$\rhom{A}{S}{X}^+ \cong (S \otimes_A^\LL X^+).$$
	\end{enumerate}
\end{lemma}
	\begin{proof} 
		(i) The isomorphism of cohomology modules follows directly from $(-)^+: \ModA \rightarrow \AMod$ being an exact contravariant functor. The second claim follows from $(-)^+$ being a faithful functor, ensured by the assumption that $W$ is a  cogenerator.

		(ii) This follows by applying the enriched derived Hom-$\otimes$ adjunction from (\ref{enriched}) twice - once directly for $Y \in \DD(A-k)$, yielding
	$$(X \otimes_A^\LL Y)^+ \cong \rhom{A}{X}{Y^+},$$
			and once using the left module version for $X \in \DD(k-A)$:
	$$\rhom{\Aop}{Y}{X^+} \cong (X \otimes_A^\LL Y)^+.$$
		
		(iii) Follows from \cite[Proposition 20.11]{AF} and d\' evissage.

		(iv) Using (ii) and (iii) we have
			$\rhom{A}{S}{X}^+ \cong (X \otimes_A^\LL S^*)^+ \cong \rhom{\Aop}{S^*}{X^+} \cong (S \otimes_A^\LL X^+).$ 
	\end{proof}

	\subsection{Definable subcategories}
Next, we turn to a concept introduced in \cite{Kr3}. A subcategory $\Vcal$ of $\DModA$ is said to be \emph{definable} if there is a set $\Phi$ of maps between compact objects of $\DModA$ such that 
$$\Vcal = \{X \in \DModA \mid \Hom{\DModA}{f}{X} \text{ is surjective for each $f \in \Phi$}\}.$$ 
This notion is the derived analogue of the notion of a definable subcategory in $\ModA$. Recall that    any definable subcategory $\Dcal$ of $\ModA$ has  a dual definable subcategory $\Dcal^\vee$ in $\AMod$ which is uniquely determined by the property that a right $A$-module $M$ lies in $\Dcal$ if and only if $M^+$ lies in $\Dcal^\vee$. We are now going to prove an analogous result on derived level.

To this end, we  need some alternative  descriptions of definability.

\begin{lemma}
	Let $\Vcal$ be a subcategory of $\DModA$. The following conditions are equivalent:
	\begin{enumerate}
		\item[(i)] $\Vcal$ is definable;		
		\item[(ii)] there is a set $\Phi$ of morphisms between compact objects of $\DModA$ such that 

$\Vcal = \{X \in \DModA \mid \Hom{\DModA}{f}{X} \text{ is injective for all $f \in \Phi$}\}$;
		\item[(iii)] there is a set $\Phi$ of morphisms between compact objects of $\DModA$ such that 

$\Vcal = \{X \in \DModA \mid \Hom{\DModA}{f}{X} \text{ is zero for all $f \in \Phi$}\}$.
	\end{enumerate}
\end{lemma}
\begin{proof}
	The statement follows by a simple argument using the long exact sequence obtained by applying $\Hom{\DModA}{-}{X}$ onto a triangle in $\DD^c(\ModA)$, see also \cite[Lemma 3.1]{BH}.
\end{proof}

Some comments on condition (iii) are in order. First of all, recall that a map $f: X \rightarrow Y$ in an additive category is \emph{zero} if it is the zero element of the abelian group $\Hom{}{X}{Y}$, or equivalently, if it is the unique map between $X$ and $Y$ which factors through the zero object.  The advantage of the description in (iii) is  that the condition on a map being zero is preserved and reflected by the duality functor $(-)^+$, unlike the two other conditions which are dual to each other.
\begin{lemma}\label{L:dualdefinable2}
	 Let $\Phi$ be a set of morphisms between objects in $\DD^c(\ModA)$, and let $\Phi^* = \{f^* \mid f \in \Phi\}$ be the set of their duals in $\DD^c(\AMod)$. Let 
$$\Vcal = \{X \in \DModA \mid \Hom{\DModA}{f}{X} \text{ is zero for all $f \in \Phi$}\}$$ 
and 
$$\Vcal^{\vee} = \{X \in \DAMod \mid \Hom{\DAMod}{f}{X} \text{ is zero for all $f \in \Phi^*$}\}$$ 
be the corresponding definable categories of $\DModA$ and $\DAMod$, respectively. 

Then the following properties hold: 
	\begin{enumerate}
		\item[(i)] for any $X \in \DModA$, we have $X \in \Vcal$ if and only if $X^+ \in \Vcal^\vee$;
		\item[(ii)] for any $Y \in \DAMod$ we have $Y \in \Vcal^\vee$ if and only if $Y^+ \in \Vcal$.
	\end{enumerate}
\end{lemma}
\begin{proof}
	For any $f \in \Phi$ and any $X \in \DModA$, we have
	$$\Hom{\DModA}{f}{X}  \text{ is zero} \Leftrightarrow H^0\,\rhom{A}{f}{X} \text{ is zero} \Leftrightarrow$$ $$\Leftrightarrow H^0 \,\rhom{A}{f}{X}^+ \text{ is zero}.$$ 
	 
	We continue by computing as follows using the natural isomorphisms from Lemma~\ref{L:dualityformulas}:
	$$H^0 \,\rhom{A}{f}{X}^+ \cong H^0(X \otimes_A^\mathbf{L} f^*)^+ \cong$$ $$\cong H^0\, \rhom{\Aop}{f^*}{X^+} \cong \Hom{\DAMod}{f^*}{X^+}.$$
	In conclusion, the morphism $\Hom{\DModA}{f}{X}$ is zero if and only if $\Hom{\DAMod}{f^*}{X^+}$ is zero, which establishes (i).

	(ii) The claim follows by the same argument applied to $\Phi^*$, as $\Phi^{**} = \Phi$.
\end{proof}

	We will say that $\Vcal$ and $\Vcal^\vee$  as above are \emph{dual definable subcategories}.
\begin{remark}\label{R:dualdef} Given a definable subcategory $\Vcal$ of $\DModA$, its dual definable subcategory  is uniquely determined by the rule $\Vcal^\vee = \{X \in \DAMod \mid X^+ \in \Vcal\}$. Also, since $\Phi^{**} = \Phi$, we have that $(\Vcal^\vee)^\vee = \Vcal$, and in particular we have for any $X \in \Vcal$ that $X^{++} \in \Vcal$.
\end{remark}
	For a subcategory $\Ccal$ of $\DModA$ (or $\DAMod$), we set $\Ccal^* = \{X^* \mid X \in \Ccal\}$ and $\Ccal^+ = \{X^+ \mid X \in \Ccal\}$.
\begin{lemma}\label{L:dualdefinable}
	 Let $\Scal$ be a set of compact objects in $\DD^c(\ModA)$. Then $\Vcal = \Scal^{\perp_0}$ and $\Vcal' = (\Scal^*)^{\perp_0}$ are dual definable subcategories. 
\end{lemma}
\begin{proof}
	This is  a special case of Lemma~\ref{L:dualdefinable2} by noticing that $\Scal^{\perp_0}$ is equal to the definable subcategory  $\{X \in \DModA \mid \Hom{\DModA}{f}{X} \text{ is zero for all } f \in \Phi\}$ corresponding to the set of identity morphisms $\Phi = \{\;{\rm id}_S: S \rightarrow S \mid S \in \Scal\}$.
\end{proof}
\subsection{Purity}
We briefly recall some basic notions from the theory of purity in derived categories, for details and further references we refer the reader e.g.~to \cite{Bel,Kr1,L}. Along the way, we  show that several classical results on the relation of  purity with duality admit a natural generalization  to the derived setting.

A triangle $X \rightarrow Y \rightarrow Z \xrightarrow{h} X[1]$  in $\DModA$ is a \emph{pure triangle} if  it is taken to a short exact sequence  of abelian groups $0 \rightarrow \Hom{\DModA}{S}{X} \rightarrow \Hom{\DModA}{S}{Y} \rightarrow \Hom{\DModA}{S}{Z} \rightarrow 0$ by every  functor $\Hom{\DModA}{S}{-}$ given by a compact object $S \in \DD^c(\ModA)$. This is further equivalent to $h$ being a \emph{phantom} map in $\DModA$, that is, $\Hom{\DModA}{S}{h}$ is a zero map in $\Modk$ for any compact object $S \in \DModA$. 

If $X\to Y\to Z\to X[1]$ is a pure triangle, we say that $X$ is a \emph{pure subobject} and $Z$ is a \emph{pure quotient} of $Y$.  An object $X \in \DModA$ is \emph{pure-injective} if every pure triangle $X\to Y\to Z\to X[1]$ in $\DModA$ is a split triangle, or equivalently, if the functor $\Hom{\DModA}{-}{X}$ takes pure triangles in $\DModA$ to short exact sequences of abelian groups.

The following Lemma shows that the usual characterization of pure-exact sequences in module categories extends to the derived setting.
\begin{lemma}\label{L:puritychar}
	Let $X \rightarrow Y \rightarrow Z \xrightarrow{h} X[1]$ be a triangle in $\DModA$. Then the following conditions are equivalent:
\begin{enumerate}
	\item[(i)] the triangle $X \rightarrow Y \rightarrow Z \xrightarrow{h} X[1]$ is pure in $\DModA$,
	\item[(ii)] the triangle $X \otimes_A^\mathbf{L} C \rightarrow Y \otimes_A^\mathbf{L} C\rightarrow Z \otimes_A^\mathbf{L} C\xrightarrow{h \otimes_A^\mathbf{L} C} X[1] \otimes_A^\mathbf{L} C$ is pure in $\DModk$ for any object $C \in \DAMod$,
	\item[(iii)] the triangle $Z^+ \rightarrow Y^+ \rightarrow X^+ \xrightarrow{h^+[1]} Z^+[1]$ in $\DAMod$ is split.
\end{enumerate}
\end{lemma}
\begin{proof}
	$(i) \Rightarrow (ii):$ 
Let first $C$ be a compact object of $\DAMod$. By Lemma~\ref{L:dualityformulas}(iii), the triangle from condition (ii) is isomorphic to the triangle
$$\rhom{A}{C^*}{X} \rightarrow \rhom{A}{C^*}{Y} \rightarrow \rhom{A}{C^*}{Z} \rightarrow \rhom{A}{C^*}{X[1]}.$$
If $P$ is a compact object in $\DModk$, then we have an adjunction $\Hom{\DModk}{P}{\rhom{A}{C^*}{-}} \cong \Hom{\DModA}{P \otimes_k^\mathbf{L} C^*}{-}$, and $P \otimes_k^\mathbf{L} C^*$ is a compact object of $\DModA$. Therefore, the purity of the triangle above in $\DModk$ follows from (i). 
For a general object $C \in \DAMod$ we argue as follows. Let $F$ be a cochain complex quasi-isomorphic to $C$ such that $F$ is K-flat and consists of flat left $A$-modules, such a complex exists by \cite{Sp}. By \cite[Theorem 1.1]{CH}, $F$ can be written as a direct limit $F = \varinjlim_{i \in I} F_i$ of perfect complexes in the category of cochain complexes of left $A$-modules. If $P$ is a compact object in $\DModk$ and $U$ is an object in $\DModA$, we have natural isomorphisms
$$\Hom{\DModk}{P}{U \otimes_A F} \cong \Hom{\DModk}{P}{U \otimes_A \varinjlim_{i \in I}F_i} \cong $$
$$ \cong \Hom{\DModk}{P}{\varinjlim_{i \in I}(U \otimes_A F_i)} \cong \varinjlim_{i \in I}\Hom{\DModk}{P}{U \otimes_A F_i}.$$
For the last isomorphism, consult \cite[Proposition 5.4]{SSV}. Therefore, $\Hom{\DModk}{P}{-}$ sends the triangle from (ii) to a direct limit of exact sequences of abelian groups, and thus to an exact sequence as required.

	$(ii) \Rightarrow (iii):$ To establish the splitting of the triangle in condition (iii), it is enough to show that the map $h^+$ is zero. For any object $C \in \DAMod$, Lemma~\ref{L:dualityformulas}(ii) yields natural equivalences of morphisms $\Hom{\DAMod}{C}{h^+} \cong H^0(h \otimes_A^\mathbf{L} C)^+$ in $\Modk$. Using (ii), the map $h \otimes_A^\mathbf{L} C$ is a phantom map in $\DModk$, and therefore $H^0(h \otimes_A^\mathbf{L} C)$ is a zero map. Consequently, we obtain that the map $\Hom{\DAMod}{C}{h^+}$ is zero for any $C \in \DAMod$. Put differently, $h^+$ is sent to a zero map in the functor category $[\DAMod^{\text{op}},\Modk]$ via the Yoneda embedding $\DAMod \xhookrightarrow{} [\DAMod^{\text{op}},\Modk]$, and therefore $h^+$ must be zero in $\Hom{\DAMod}{X^+[1]}{Z^+}$ as well, as desired.
	
	$(iii) \Rightarrow (i):$ Let $S$ be a compact object in $\DModA$. By Lemma~\ref{L:dualityformulas}(iv), we have a natural equivalence of maps $\Hom{\DModA}{S}{h}^+ \cong H^0(S \otimes_A ^\mathbf{L} h^+)$ in $\Modk$. Since the triangle $Z^+ \rightarrow Y^+ \rightarrow X^+ \rightarrow Z^+[1]$ is split, so is the triangle $S \otimes_A^\mathbf{L} Z^+ \rightarrow S \otimes_A^\mathbf{L}Y^+ \rightarrow S \otimes_A^\mathbf{L}X^+ \xrightarrow{S \otimes_A^\mathbf{L}h^+[1]} S \otimes_A^\mathbf{L}Z^+[1]$, and therefore, using Lemma~\ref{L:dualityformulas}(iv), the map $\Hom{\DModA}{S}{h}^+ \cong H^0(S \otimes_A ^\mathbf{L} h^+)$ is a zero morphism. Consequently by duality, the map $\Hom{\DModA}{S}{h}$ is zero for any compact object $S \in \DModA$. Therefore, $h$ is a phantom map in $\DModA$, showing that the triangle $X \rightarrow Y \rightarrow Z \xrightarrow{h} X[1]$ is pure.
\end{proof}
	 It will be useful in the sequel to note that the usual evaluation map has a derived counterpart enjoying similar properties. Let $X$ be an object of $\DAMod$. Using the Hom-$\otimes$ adjunction twice, we get the following natural isomorphisms:
$$\Hom{\DAMod}{X}{X^{++}} \cong \Hom{\DModk}{X^+ \otimes_A^\LL X}{W} \cong \End_{\DModA}{(X^+)}.$$
We let $\epsilon_X \in \Hom{\DAMod}{X}{X^{++}}$ be the map which corresponds to the identity of the ring $\End_{\DModA}{(X^+)}$ under the isomorphism above and call $\epsilon_X: X \rightarrow X^{++}$ the \emph{evaluation morphism}.

	Now let $P$ be a K-projective quasi-isomorphic replacement of $X$ in $\DAMod$. For any acyclic complex $N$ of right $A$-modules, we have the adjunction isomorphism $\Hom{\KModA}{N}{\Hom{k}{P}{W}} \cong \Hom{\KModk}{N \otimes_A P}{W}$. Since $P$ is also $K$-flat, and $\Hom{k}{-}{W}$ is exact, we infer that $P^+$ is a K-injective complex of right $A$-modules. Then there is a commutative square of natural isomorphisms
$$
\begin{CD}
	\Hom{\DAMod}{X}{X^{++}} @> \cong >> \Hom{\KAMod}{P}{P^{++}} \\
		@V \cong VV @V \cong VV \\
	\End_{\DModA}(X^+) @> \cong >> \End_{\KModA}(P^+) \\
\end{CD}
$$
The evaluation map $\epsilon_X \in \Hom{\DAMod}{X}{X^{++}}$ corresponds to a map $\epsilon_P \in \Hom{\KAMod}{P}{P^{++}}$ which is mapped  to the identity in $\End_{\KModA}(P^+)$ under the vertical arrow. It follows that the homotopy class $\epsilon_P: P \rightarrow P^{++}$ can be represented by the standard evaluation map given by the rule $\epsilon_P^n(x)(f) = f(x)$ for each $x \in P^n$ and $f \in (P^n)^+$ in each coordinate $n \in \Z$.

\begin{lemma}\label{L:evaluation}
The evaluation map $\epsilon_X: X \rightarrow X^{++}$ realizes $X$ as a pure subobject of $X^{++}$ for any $X \in \DModA$.
\end{lemma}
\begin{proof}
	As in the discussion above, we can replace $\epsilon_X$ by a map $\epsilon_P: P \rightarrow P^{++}$, where $P$ is a quasi-isomorphic K-projective replacement of $X$ such that $\epsilon_P$ is the usual evaluation map of cochain complexes. By Lemma~\ref{L:puritychar}, it is enough to check that $\epsilon_P^+: P^{+++} \rightarrow P^+$ is a split epimorphism in $\DAMod$. But it is straightforward to check that the evaluation map $\epsilon_{P^+}: P^+ \rightarrow P^{+++}$ of cochain complexes provides the desired section of $\epsilon_P^+$.
\end{proof}

\begin{corollary}\label{dualsarepureinjective}
Let $C$ be an object in $\DModA$. Then the following conditions are equivalent:
\begin{enumerate}
	\item[(i)] $C$  is pure-injective in $\DModA$,
	\item[(ii)] the evaluation map $\epsilon_C: C \rightarrow C^{++}$ is a split monomorphism,
	\item[(iii)] $C$ is isomorphic to a direct summand of $D^+$ for some $D \in \DAMod$.
\end{enumerate}
\end{corollary}
\begin{proof}
	$(i) \Rightarrow (ii):$ 
	This follows by combining (i) with Lemma~\ref{L:evaluation}.

	$(ii) \Rightarrow (iii):$ 
	Obvious.

	$(iii) \Rightarrow  (i):$ 
	By passing to direct summands, it is sufficient to establish the implication in the case when $C = D^+$. 
	We show that $\Hom{\DModA}{-}{C}$ sends pure triangles in $\DModA$ to short exact sequences in $\Modk$.
	We apply Lemma~\ref{L:puritychar}: if 
\begin{equation}\label{pt} X \rightarrow Y \rightarrow Z \xrightarrow{h} X[1]\end{equation} is  a pure triangle in $\DModA$, then  the triangle $X \otimes_A^\mathbf{L} D \rightarrow Y \otimes_A^\mathbf{L} D\rightarrow Z \otimes_A^\mathbf{L} D\xrightarrow{h \otimes_A^\mathbf{L} D} X[1] \otimes_A^\mathbf{L} D$ is pure in $\DModk$, and the functor  $\Hom{k}{-}{W}$ takes it to a split triangle in $\DModk$. But by adjunction the latter is isomorphic to the  triangle obtained by applying the functor $\rhom{A}{-}{D^+}$ on the triangle (\ref{pt}). Passing to cohomology yields that $0 \rightarrow \Hom{\DModA}{Z}{D^+}\rightarrow \Hom{\DModA}{Y}{D^+}\rightarrow \Hom{\DModA}{X}{D^+} \rightarrow 0$ is exact, establishing (i).
\end{proof}

\subsection{Torsion pairs and TTF triples}\label{tpandttf}
  A pair $(\Ucal,\Vcal)$ of full additive subcategories of $\DModA$ is a \emph{torsion pair} provided that the following conditions hold:
\begin{enumerate}
	\item both $\Ucal$ and $\Vcal$ are closed under direct summands,
	\item $\Hom{\DModA}{\Ucal}{\Vcal} = 0$, and
	\item for any object $X \in \DModA$ there is a triangle
			$$U \rightarrow X \rightarrow V \rightarrow U[1]$$
			in $\DModA$ with $U \in \Ucal$ and $V \in \Vcal$.
\end{enumerate}
Then $\Ucal$ is called the \emph{aisle} and $\Vcal$ the \emph{coaisle} of the torsion pair. 

A torsion pair $(\Ucal,\Vcal)$ is a \emph{t-structure} (resp.~\emph{co-t-structure}) provided that $\Ucal[1] \subseteq \Ucal$ (resp.~$\Ucal[-1] \subseteq \Ucal$). 
When  $(\Ucal,\Vcal)$ is a {t-structure}, the triangle from condition $(3)$ is determined uniquely up to a unique isomorphism, and it is always isomorphic to a triangle of form
	$$\tau_{\Ucal}(X) \rightarrow X \rightarrow \tau_{\Vcal}(X) \rightarrow \tau_{\Ucal}(X)[1],$$
	where $\tau_{\Ucal}$ (resp.~$\tau_{\Vcal}$) is the right (resp.~left) adjoint to the inclusion $\Ucal \subseteq \DModA$ (resp. $\Vcal \subseteq \DModA$). 
	
	\begin{example}\label{tstrexamples} 
		(i) For each $n \in \Z$, consider the  following subcategories of $\DModA$:
			$$\DD^{\leq n} = \{X \in \DModA \mid H^k(X) = 0 \text{ for all } k > n\}\text{, and}$$
			$$\DD^{> n} = \{X \in \DModA \mid H^k(X) = 0 \text{ for all } k \leq n\}.$$
			In the text, we will freely use the alternative symbols $\DD^{<n} = \DD^{\leq n-1}$ and $\DD^{\geq n} = \DD^{>n-1}$. We omit a reference to the ground ring $A$ which should always be clear from the context. It is well-known that the pair $(\DD^{\leq n},\DD^{>n})$ forms a t-structure in $\DModA$. The functors $\tau_{\DD^{\leq n}}$ and $\tau_{\DD^{>n}}$ are represented by the soft truncations $\tau^{\leq n}$ and $\tau^{>n}$ of cochain complexes.
			
			(ii) The following construction goes back to \cite{HRS}. Let $(\Tcal,\Fcal)$ be a torsion pair in $\ModA$, that is, a pair of full subcategories of $\ModA$ such that $\Tcal = {}^{\perp_0} \Fcal$ and $\Fcal = \Tcal^{\perp_0}$. Then there is a t-structure $(\Ucal,\Vcal)$ in $\DModA$, where
			$$\Ucal = \{X \in \DD^{\leq 0} \mid H^0(X) \in \Tcal\}\text{, and}$$
			$$\Vcal = \{X \in \DD^{\geq 0} \mid H^0(X) \in \Fcal\},$$
			called the \emph{Happel-Reiten-Smal\o~} t-structure. This construction yields an injective map from the class of torsion pairs in $\ModA$ to the class of t-structures in $\DModA$.
			\end{example}
	
	A \emph{TTF (torsion-torsion-free) triple} 
	is a triple $(\Ucal,\Vcal,\Wcal)$ formed by two adjacent torsion pairs $(\Ucal,\Vcal)$ and $(\Vcal,\Wcal)$. It is called 
\emph{suspended} (respectively, \emph{cosuspended})  if $\Vcal[1] \subseteq \Vcal$ (respectively, $\Vcal[-1] \subseteq \Vcal$).

   In other words, a suspended TTF triple is a triple $(\Ucal,\Vcal,\Wcal)$ such that $(\Ucal,\Vcal)$ is a co-t-structure, and $(\Vcal,\Wcal)$ is a t-structure, while a cosuspended TTF triple is a triple $(\Ucal,\Vcal,\Wcal)$ such that $(\Ucal,\Vcal)$ is a t-structure, and $(\Vcal,\Wcal)$ is a co-t-structure. 
   
 A    t-structure $(\Ucal,\Vcal)$ is said to be  \emph{stable} if $\Ucal$ and $\Vcal$ are triangulated subcategories of $\DModA$, or equivalently, $\Ucal$ is
a  \emph{localizing} subcategory of $\DModA$, i.e.~a full triangulated subcategory which is closed under  coproducts. If $\Vcal$ is also closed under coproducts, then $\Ucal$ is said to be \emph{smashing}. By \cite[Corollary 2.4]{NS}, every smashing subcategory $\Ucal$ gives rise  to a  TTF triple  $(\Ucal,\Vcal,\Wcal)$ which is \emph{stable}, i.e.~suspended and cosuspended.

 We say  that a torsion pair  $(\Ucal,\Vcal)$ is {\it non-degenerate} 
 if  it satisfies $$\bigcap_{n \in \Z}\Ucal[n] = 0 = \bigcap_{n \in \Z}\Vcal[n].$$ A suspended TTF triple $(\Ucal,\Vcal,\Wcal)$  will be called \emph{non-degenerate} if so is the t-structure $(\Vcal,\Wcal)$, and it will be called \emph{intermediate} 
 if   there are integers $m \leq n$ such that $\DD^{\leq m}\subseteq \Vcal \subseteq \DD^{\leq n}$.  
 A cosuspended TTF triple $(\Ucal,\Vcal,\Wcal)$  will be called \emph{non-degenerate} if so is the t-structure $(\Ucal,\Vcal)$, and it will be called  \emph{cointermediate} if there are integers $m \leq n$ such that $\DD^{\leq m} \subseteq \Ucal \subseteq \DD^{\leq n}$.

Moreover, we say that a torsion pair $(\Ucal,\Vcal)$, or a   TTF triple  $(\Ucal,\Vcal,\Wcal)$, 	is
\begin{itemize}
\item  {\it generated by a set} of objects $\Scal$ of $\DModA$ if $\Vcal = \Scal^{\perp_0}$,
\item  {\it compactly generated} if it is generated by a set of compact objects of $\DModA$,
\item {\it homotopically smashing} if $\Vcal$ is closed under directed homotopy colimits (see \cite{SSV},  \cite[Appendix]{HN}).		
		\end{itemize}
	Note that any set of compact objects $\Scal$ in $\DD^c(\ModA)$ generates a TTF triple $({}^{\perp_0}(\Scal^{\perp_0}),\Scal^{\perp_0},(\Scal^{\perp_0})^{\perp_0})$ by \cite[Theorem 4.3]{AI} and \cite[Theorem 3.11]{PS}.	
Furthermore, it is shown in \cite{AJS} that every set of objects $\Scal$ in $\DModA$ gives rise to a stable t-structure $(\mathsf{Loc}(\Scal), \Scal^{\perp_{\mathbb Z}})$ which is 
generated by the objects of $\Scal$ and all their shifts. Here $\Loc(\Scal)$ denotes the smallest  localizing subcategory of $\DModA$ containing $\Scal$. 

Of course, every compactly generated t-structure has a definable coaisle, and by \cite[Theorem 3.11]{L} every  t-structure with a definable coaisle is homotopically smashing. 

Recall from \cite[Fundamental Correspondence]{Kr3} that any definable subcategory of $\DModA$ is uniquely determined by the (indecomposable) pure-injective objects it contains. In what follows, we show that any t-structure with a definable coaisle is also determined by pure-injectives as a torsion pair. A torsion pair $(\Ucal,\Vcal)$ is said to be
\begin{itemize}
\item \emph{cogenerated} by a subcategory $\Scal$ of $\DModA$ if $\Ucal = {}^{\perp_0}\Scal$.
\end{itemize}
\begin{proposition}\label{P:cogenPI}
	Let $A$ be a ring and $(\Ucal,\Vcal)$ a t-structure in $\DModA$ such that the coaisle $\Vcal$ is definable. Then $(\Ucal,\Vcal)$ is cogenerated by a set of pure-injective objects of $\DModA$.
\end{proposition}
\begin{proof}
	Let $\Bcal = \bigcap_{n \in \Z}\Vcal[n]$. By \cite[Corollary 6.6]{LV}, there is a stable TTF triple $(\Lcal,\Bcal,\Kcal)$ in $\DModA$. Furthermore, by \cite[Proposition 6.11]{LV}, there is a pure-injective object $C \in \DModA$ such that the t-structure $(\Ucal',\Vcal')$ in $\DModA$ 	defined by $\Ucal'={}^{\perp_{\leq 0}}C$ (which exists by \cite[Corollary 5.4]{LV}) satisfies $\Vcal'=\Vcal \cap \Kcal$ and  $\Vcal = \Bcal \star \Vcal'$, where $\Bcal \star \Vcal'$ denotes the subcategory consisting of all objects $X \in \DModA$ fitting into a triangle $B \rightarrow X \rightarrow V' \rightarrow B[1]$ with $B \in \Bcal$ and $V' \in \Vcal'$. Then $\Ucal =\Ucal'\cap\Lcal$. By \cite[\S 4, Theorem]{NS} or \cite[Proposition 2.5]{BS}, there is an object $B \in \DD(A - A)$ in the derived category of $A$-$A$-bimodules (in fact, $B$ comes from a suitable homological epimorphism $A \rightarrow B$ of dg algebras) such that $\Lcal = \Ker (- \otimes_A^\LL B)$. Then by adjunction we have $\Lcal = \Ker \rhom{A}{-}{B^+}$. In conclusion, we have $\Ucal = {}^{\perp_{0}}\Scal$, where $\Scal = \{B^+[n], C[m] \mid n \in \mathbb{Z}, m \leq 0\}$. Finally, we know that $C$ is pure-injective, and $B^+$ is a pure-injective object in $\DModA$ by Corollary~\ref{dualsarepureinjective}(iii).
\end{proof}

\subsection{Silting and cosilting TTF triples}
We say that an object $T \in \DModA$ is \emph{silting} if the pair $(T^{\perp_{>0}},T^{\perp_{\leq 0}})$ is a t-structure, which we call the \emph{silting t-structure induced by} $T$. Two silting objects $T, T' \in \DModA$ are \emph{equivalent} if they induce the same t-structure.

In view of the duality results which will be established in Subsection~\ref{dual}, it is convenient to  consider the dual notion of a  cosilting object in the unbounded derived category $\DAMod$ of \emph{left} $A$-modules over a ring $A$.  An object $C \in \DAMod$ is {\it cosilting} if the pair $({}^{\perp_{\leq 0}}C,{}^{\perp_{>0}}C)$ forms a t-structure, which we call the \emph{cosilting t-structure induced by} $C$. Two cosilting objects  are \emph{equivalent} if they induce the same t-structure. 

A silting object  is called a {\it bounded silting complex} if it belongs to $\K^b(\ProjA)$, and a cosilting object is called a {\it bounded cosilting complex} if it belongs to $\K^b(\AInj)$.  

Silting t-structures can be characterized as t-structures fitting into certain TTF triples. 

\begin{theorem}\emph{(\cite[Theorem 4.11]{objects}, \cite[Theorem 4.6]{AMV1})}\label{T:siltingTTF} Let  $(\Vcal,\Wcal)$ be a  t-structure in $\DModA$.

(1)	 $(\Vcal,\Wcal)$ is silting
	if and only if it extends to a TTF triple $(\Ucal,\Vcal,\Wcal)$ which is non-degenerate, suspended, and  generated by a set of objects of $\DModA$.

	(2)   $(\Vcal,\Wcal)$ is induced by a bounded silting complex if and only if it extends to an intermediate suspended TTF triple $(\Ucal,\Vcal,\Wcal)$.
\end{theorem}

It is proved in \cite[Theorem 3.6 and Proposition 3.10]{MV} that every intermediate suspended TTF triple in  $\DModA$ is compactly generated, and that every bounded
cosilting complex is pure-injective, which in particular means that the induced t-structure is homotopically smashing. More generally, by \cite[Theorem 4.6]{L}, a t-structure $(\Ucal,\Vcal)$ is induced by a pure-injective cosilting object if and only if it is non-degenerate and homotopically smashing, and in this case the coaisle $\Vcal$ is automatically definable in $\DAMod$. Combining this with a result from \cite{LV} one obtains the following characterization.

 \begin{theorem}\emph{(\cite[Proposition 5.7]{LV},\cite[Theorem 4.6]{L}, \cite[Theorem 6.13]{objects}, \cite[Proposition 3.10 and Theorem 3.13]{MV})}\label{T:cosiltingTTF} Let  $(\Ucal,\Vcal)$ be a  t-structure in $\DAMod$.
	
(1) If the coaisle $\Vcal$ is definable then the t-structure extends to a cosuspended TTF triple $(\Ucal,\Vcal,\Wcal)$.

(2)	 $(\Ucal,\Vcal)$ is induced by a pure-injective cosilting object
	if and only if there is a non-degenerate cosuspended TTF triple $(\Ucal,\Vcal,\Wcal)$ which is homotopically smashing.
		
(3)   $(\Ucal,\Vcal)$ is induced by a bounded cosilting complex
	if and only if there is a cointermediate cosuspended TTF triple $(\Ucal,\Vcal,\Wcal)$. In particular, $(\Ucal,\Vcal)$ is then homotopically smashing. 
	\end{theorem}

We now restrict to compactly generated silting and cosilting t-structures, for which we will establish a duality result in Subsection~\ref{dual}.
\begin{definition}\label{finite and cofinite type}
	We say that a silting object $T$ in $\DModA$ is of \emph{finite type} if the induced silting TTF triple is compactly generated. Similarly, we call a cosilting object $C$ in $\DAMod$ of \emph{cofinite type} provided that it induces a compactly generated TTF triple.
\end{definition}
By \cite[Theorem 3.6 and Example 3.12]{MV}, any bounded silting complex is of finite type, but bounded cosilting complexes need not be of cofinite type. However, it is shown in \cite{HCG,HN}  that 
 every  pure-injective cosilting object over a commutative noetherian ring is of cofinite type, and we are going to see in Theorem~\ref{hereditaryutc} that the same holds true over hereditary rings.

 As an immediate consequence of Theorems~\ref{T:siltingTTF} and ~\ref{T:cosiltingTTF}, we obtain the following  characterization of TTF triples induced by (co)silting objects of (co)finite type.
\begin{corollary}\label{L:finitecofinite}
	Let $A$ be a ring. Then:
\begin{enumerate}
	\item[(i)] A compactly generated TTF triple  is silting if and only if it is suspended and non-degenerate.
	\item[(ii)] A compactly generated TTF triple  is cosilting if and only if it is cosuspended and non-degenerate.
\end{enumerate}
\end{corollary}

\subsection{Silting  and cosilting modules}
We now focus on bounded silting or cosilting complexes of length two. The modules that occur as zero cohomologies of such complexes can be defined as follows. For details we refer to \cite{AMV1,BPop}.

\begin{definition}\label{def p silting}
An $A$-module $T$ is said to be
\begin{itemize}
\item \emph{silting} if it admits  a projective presentation $ P\stackrel{\sigma}{\longrightarrow} Q\to T\to 0$ such that  $\Gen{T}$ coincides with the class 
$$\mathcal{D}_\sigma=\{X\in \ModA \mid \Hom{A}{\sigma}{X}\ \text{is surjective} \};$$
\item \emph{tilting}  if $\Gen{T}=T^{\perp_1}$, or equivalently, $T$ is silting and the map $\sigma$ is injective.
\end{itemize}
 The torsion class $\Gen{T}$ generated by a silting (respectively, tilting) module $T$ is called a \emph{silting} (respectively, \emph{tilting}) \emph{class}. Two silting modules $T$ and $T'$ are said to be \emph{equivalent} if they generate the same silting class, which amounts to having the same additive closure $\Add{T}=\Add{T^\prime}$.
  
 \emph{Cosilting} and \emph{cotilting} modules and classes are defined dually in terms of the classes $\Cogen C$ and $$\Ccal_\omega=\{X\in \ModA \mid \Hom{A}{X}{\omega}\ \text{is surjective} \},$$ where $\omega$  is an injective copresentation of the module $C$.
  Two cosilting modules $C,C'$ are equivalent if they cogenerate the same cosilting class, which amounts to the equality $\Prod{C}=\Prod{C^\prime}$. 
\end{definition} 

 Here is the connection between silting modules, objects, and t-structures:
 if $T$ is a silting module in $\ModA$ with respect to a projective presentation $\sigma$, then $\sigma$ is a silting object (of finite type) in $\DModA$, and the t-structure induced by $\sigma$ is the Happel-Reiten-Smal\o~ t-structure (cf.~Example~\ref{tstrexamples}(ii)) arising from the torsion pair $(\Gen T,T^{\perp_0})$.  
 Similarly, if  $C$ is a cosilting module with respect to an injective copresentation $\omega$, then  $\omega$ is a cosilting object in $\DModA$, and the t-structure induced by $\omega$
arises from the torsion pair $({}^{\perp_0}C, \Cogen C)$.  We say that $C$, or the cosilting  class $\Cogen C$, is {\em of cofinite type} if so is the cosilting object $\omega$.

Silting and cosilting classes are \emph{definable} subcategories of $\ModA$, i.e.~they are closed under direct products,  direct limits, and pure submodules. In fact, the cosilting classes are precisely the definable torsion-free classes, cf.~\cite[Corollary 3.9]{abundance}.

 Given a definable subcategory $\Dcal$ of $\ModA$, we denote by $\Dcal^\vee$ its dual definable subcategory in $\AMod$ determined by the property that a right $A$ module $M$ lies in $\Dcal$ if and only if $M^+$ lies in $\Dcal^\vee$. 

  \begin{proposition}\cite[Proposition 3.5]{AH}\label{dually} 
Let $\sigma$ be a map between projective right $A$-modules.
If $\Dcal_\sigma$ is a silting class in $\ModA$, then $\Dcal_\sigma\,^\vee =\Ccal_{\sigma^+}$ is a cosilting class in $\AMod$.	 
Furthermore, if $T_A$ is a silting module with respect to $\sigma$, then ${}_AT^+$ is a cosilting module  with respect to   $\sigma^+$.
		\end{proposition}
		\begin{corollary}\cite[Corollary 3.6]{AH}\label{duality} The assignment $\Dcal\mapsto \Dcal^\vee$ defines a bijection between silting classes  in $\ModA$ and cosilting classes of cofinite type in $\lmod{A}$. \end{corollary}


\section{Compactly generated TTF triples}
 In this section, we develop some tools to study (co)silting objects of (co)finite type. First of all, in subsection~\ref{dual} we show that the silting-cosilting duality discussed above on the level of module categories extends to derived categories. In subsection~\ref{commnoether}, we  see that over a commutative noetherian ring this duality yields a bijection between (equivalence classes of) silting objects of finite type and pure-injective cosilting objects.  The classification of compactly generated t-structures from \cite{AJSa} then provides a  parametrization of these classes  by certain chains of specialization-closed subsets of the Zariski spectrum. An essential ingredient for these results is the fact that all pure-injective cosilting  objects over a commutative noetherian ring are of cofinite type, which is proved in \cite{HN}. In subsection~\ref{hereditary} we establish the same result for hereditary rings.
 
 \subsection{Silting-cosilting duality}\label{dual}
Our aim in this subsection is to prove a triangulated version of Corollary~\ref{duality}.
\begin{theorem}\label{T:TTFduality}(cf.~\cite[Theorem 3.11]{PS})
	 There is a 1-1 correspondence
		$$\left \{ \begin{tabular}{ccc} \text{Compactly generated} \\ \text{TTF triples} \\ \text{in $\DModA$} \end{tabular}\right \}  \stackrel{\Psi}{\longleftrightarrow}  \left \{ \begin{tabular}{ccc} \text{Compactly generated} \\ \text{TTF triples} \\ \text{in $\DAMod$} \end{tabular}\right \}.$$
	The correspondence $\Psi$ is given as follows: to the TTF triple in $\DModA$ generated by  a set of compact objects   $\Scal$ in $\DD^c(\ModA)$  we assign the TTF triple in $\DAMod$ generated by the set $\Scal^*$.
\end{theorem}
\begin{proof}
	The only thing we need to prove is that the assignment is well-defined, that is, if $\Scal_0$ and $\Scal_1$ are two subcategories of $\DD^c(\ModA)$ such that $\Scal_0^{\perp_0} = \Scal_1^{\perp_0}$, then also $(\Scal_0^*)^{\perp_0} = (\Scal_1^*)^{\perp_0}$. For any $Y \in \DAMod$ we have by Lemma~\ref{L:dualdefinable} that: 
	$$Y \in (\Scal_0^*)^{\perp_0} \Leftrightarrow Y^+ \in \Scal_0^{\perp_0} =\Scal_1^{\perp_0} \Leftrightarrow Y \in (\Scal_1^*)^{\perp_0},$$
	establishing the claim.
\end{proof}

\begin{lemma}\label{L:nondegengeneral}
	Let  $(\Ucal',\Vcal',\Wcal')$ be a compactly generated TTF triple in $\DModA$ and $(\Ucal,\Vcal,\Wcal)$ a compactly generated TTF triple in $\DAMod$ corresponding to each other via $\Psi$. Then $(\Ucal',\Vcal',\Wcal')$ is suspended if and only if $(\Ucal,\Vcal,\Wcal)$ is cosuspended, and if that is the case, the following holds:
	\begin{enumerate}
		\item[(i)]  $\bigcap_{n \in \Z}\Vcal'[n] = 0$ if and only if $\bigcap_{n \in \Z}\Vcal[n] = 0$.
		\item[(ii)] If $\bigcap_{n \in \Z}\Wcal'[n] = 0$, then $\bigcap_{n \in \Z}\Ucal[n] = 0$.
		\item[(iii)] $(\Ucal',\Vcal',\Wcal')$ is intermediate if and only if  $(\Ucal,\Vcal,\Wcal)$ is cointermediate.
	\end{enumerate}
\end{lemma}
\begin{proof}
  For any $Y \in \Vcal'$, we have $(Y[1])^+ \cong Y^+[-1]$, and for any $V \in \Vcal$ we have $(V[-1])^+ \cong V^+[1]$. By Lemma~\ref{L:dualdefinable} we infer that $\Vcal$ is closed under $[-1]$ if and only if $\Vcal'$ is closed under $[1]$. 

	$(i)$ It follows from Lemma~\ref{L:dualdefinable} that $X \in \Vcal'[n]$ if and only if $X^+ \in \Vcal[-n]$. Therefore, if we assume $\bigcap_{n \in \Z}\Vcal[n] = 0$, then for any $X \in \bigcap_{n \in \Z}\Vcal'[n]$ we have $X^+ = 0$, and therefore $X = 0$. The other implication is proved in the same way.

	$(ii)$ First note that $\bigcap_{n \in \Z}\Wcal'[n] = \Vcal'^{\perp_\Z}$, and $\bigcap_{n \in \Z}\Ucal[n] = {}^{\perp_\Z}\Vcal$. Suppose that $\Vcal'^{\perp_\Z}=0$, and pick an object $X \in \DAMod$ belonging to ${}^{\perp_\Z}\Vcal$, which means that $\rhom{\Aop}{X}{\Vcal} = 0$. By Lemma~\ref{L:dualdefinable}, we have $\Vcal'^+ \subseteq \Vcal$, and thus $\rhom{\Aop}{X}{\Vcal'^+} = 0$. By Lemma~\ref{L:dualityformulas}, this translates as $\rhom{A}{\Vcal'}{X^+} = 0$. By assumption it follows  $X^+ = 0$ in $\DModA$, and thus $X = 0$ in $\DAMod$, as desired.
		
		$(iii)$ 
	Using Lemma~\ref{L:dualdefinable} and Lemma~\ref{L:dualityformulas}(i) we infer that for all $n \in \mathbb{Z}$, the inclusion
		$\DD{}^{\geq n} \subseteq \Vcal$ implies that $ X^+ \in \Vcal$ for all $X$ in $\DD{}^{\leq -n}$, hence $\DD{}^{\leq -n} \subseteq \Vcal'$. 
	Similarly, if $m\in\mathbb{Z}$ and 
		$\Vcal \subseteq \DD{}^{\geq m}$, then for all $X$ in $\DModA$ the condition $X^+ \in \Vcal$ implies that $X \in \DD{}^{\leq -m}$, hence  $\Vcal' \subseteq \DD{}^{\leq -m}$. 
	The same argument with the r\^ oles of $\Vcal$ and $\Vcal'$ switched concludes the proof.
\end{proof}

We can now prove the desired triangulated version of the silting-cosilting duality from Corollary~\ref{duality}. Note that while we showed in Lemma~\ref{L:nondegengeneral} that the duality restricts perfectly well to (co)intermediate TTF triples, the preservation of the non-degeneracy condition is established only in one direction. This is why  the first map in the following Theorem is only shown to be an injection. However, in \S\ref{commnoether}, we will be able to remove this inadequacy in case $R$ is commutative noetherian.
\begin{theorem}\label{P:silttocosilt}
	The correspondence $\Psi$ induces an injective map
		$$\left \{ \begin{tabular}{ccc} \text{Silting objects of finite type} \\ \text{in $\DModA$, up to equivalence} \end{tabular}\right \} \, \xhookrightarrow{} \, \left \{ \begin{tabular}{ccc} \text{Cosilting objects of cofinite type} \\  \text{ in $\DAMod$, up to equivalence} \end{tabular}\right \}\qquad\qquad\quad$$
		which is given by the assignment $T \mapsto T^+$, and which 
	 restricts to a bijection
		$$\left \{ \begin{tabular}{ccc} \text{Bounded silting complexes} \\ \text{in $\DModA$, up to equivalence} \end{tabular}\right \}  \longleftrightarrow  \left \{ \begin{tabular}{ccc} \text{Bounded cosilting complexes} \\ \text{of cofinite type in $\DAMod$,} \\ \text{ up to equivalence} \end{tabular}\right \}.$$
\end{theorem}			
				
\begin{proof}
	Let $T \in \DModA$ be a silting object of finite type,  let $(\Ucal',\Vcal',\Wcal')$ be the induced compactly generated suspended non-degenerate TTF triple in $\DModA$, and let $(\Ucal,\Vcal,\Wcal)$ be its image under $\Psi$ in $\DAMod$. By Lemma~\ref{L:nondegengeneral}, we see that the compactly generated  TTF triple $(\Ucal,\Vcal,\Wcal)$ is cosuspended and non-degenerate, and thus it is cosilting by Corollary~\ref{L:finitecofinite}. Put $C=T^+$, and let us show that $C$ is a cosilting object inducing $(\Ucal,\Vcal,\Wcal)$.

	Since $T \in \Vcal'$, we have $C \in \Vcal$. For any $X \in \DAMod$, we have by Lemma~\ref{L:dualdefinable} and \ref{L:dualityformulas} that:
$$X \in {}^{\perp_{>0}}C \Leftrightarrow \rhom{\Aop}{X}{T^+} \in \DD^{\leq 0} \Leftrightarrow \rhom{A}{T}{X^+} \in \DD^{\leq 0} \Leftrightarrow$$ $$\Leftrightarrow X^+ \in \Vcal' \Leftrightarrow X \in \Vcal.$$
We showed that ${}^{\perp_{>0}}C = \Vcal$. It remains to check that $\Ucal = {}^{\perp_{\leq 0}} C$. If $X \in \DAMod$ belongs to $ \Ucal$, then $\Hom{\DAMod}{X}{\Vcal}=0$. Since 
$\Vcal$ contains $C$ and all its negative shifts, we infer that
$X$ lies in $^{\perp_{\leq 0}} C$. For the other inclusion, consider the approximation triangle with respect to the t-structure $(\Ucal,\Vcal)$:
	\begin{equation}\label{EE:eeh}\tau_\Ucal(X) \rightarrow X \rightarrow \tau_\Vcal(X) \rightarrow \tau_\Ucal(X)[1].\end{equation}
		Assume $X \in {}^{\perp_{\leq 0}} C$. By the previous consideration, we have $\tau_\Ucal(X) \in {}^{\perp_{\leq 0}} C$ and $\tau_\Vcal(X) \in {}^{\perp_{>0}}C$. By applying $\Hom{\DAMod}{-}{C}$ to (\ref{EE:eeh}), we easily see that $\tau_\Vcal(X) \in  {}^{\perp_{\mathbb{Z}}}C$. But since $T$ is a silting complex, it is a generator in $\DModA$, and it is easy to check that $C=T^+$ is then necessarily a cogenerator in $\DAMod$, implying that $\tau_\Vcal(X) = 0$, and therefore $X \in \Ucal$.		
		
Let us now show  the second statement.
	It is clear that if $T$ is a bounded silting complex in $\DModA$, then $T^+$ belongs to $\K^b(\AInj)$. Since any bounded silting object in $\DModA$ is of finite type,  the assignment $T \mapsto T^+$ 
	thus restricts as stated, and we only have to prove surjectivity. Let  $C$ be a bounded cosilting object of cofinite type in $\DAMod$, let $(\Ucal,\Vcal,\Wcal)$ be the induced cointermediate cosuspended  TTF triple, and $(\Ucal',\Vcal',\Wcal')$ its preimage under $\Psi$. Then $(\Ucal',\Vcal',\Wcal')$ is intermediate by Lemma~\ref{L:nondegengeneral}, and by Theorem~\ref{T:siltingTTF} it is induced by   a bounded silting complex $T \in \DModA$. Then $T^+$ is a bounded cosilting object of cofinite type inducing $(\Ucal,\Vcal,\Wcal)$, and the proof is complete.		
\end{proof}

\subsection{Over commutative  noetherian rings} \label{commnoether}
In this section, we focus on commutative noetherian rings and strengthen the statements of Theorem~\ref{P:silttocosilt}.  Our arguments will rely on some important classification results 
which we review below. Let us first briefly recall some terminology. Given a commutative noetherian ring $A$ and  an element $\p$  in  the prime spectrum  $\Spec(A)$, we denote by $\kappa(\p) = A_{\p}/\p A_{\p}$ the residue field of $A$ at $\p$. The  \emph{support} of a complex  of $A$-modules $X$ is defined as
$\supp X=\{\mathfrak p\in\Spec(A)\,\mid\, X\otimes^\mathbb{L}_Ak(\p)\not=0\},$ and the support $\supp\Xcal$ 
of a subcategory $\Xcal$ of $\DModA$ is the union of the supports of the objects of $\Xcal$.
Notice that for a finitely generated $A$-module $M$ this definition of support agrees with the \emph{classical support} $\Supp\, M=\{\p\in\Spec A\mid M\otimes_A A_\p\not=0\}$.

 By a well-known result due to Neeman and Hopkins,  the assignment of support yields a parametrization of  the localizing subcategories of $\DModA$ by subsets of $\Spec(A)$. Moreover, it was shown by Alonso, Jerem\'ias and Saor\'in that  the  compactly generated t-structures in $\DModA$  are parametrized by certain chains of subsets of  $\Spec(A)$.  

\begin{definition}   A subset $P$ of $\Spec(A)$ is said to be \emph{closed under specialization} if for all primes $\mathfrak{p}\subseteq \mathfrak{q}$, if $\mathfrak{p}$ lies in $P$, then so does $\mathfrak{q}$. 
A \emph{filtration by supports} of $\Spec
A$ is a  map $\Phi :\mathbb Z\longrightarrow\mathcal{P}(\Spec(A))$ such
that each $\Phi (n)$ is a subset of $\Spec(A)$ closed under
specialization and  $\Phi (n)\supseteq\Phi (n+1)$ for all
$n\in\mathbb Z$.
\end{definition}
Every filtration by supports
$\Phi$ gives rise to a  t-structure $(\mathcal{U}_\Phi
,\mathcal{V}_\Phi)$ whose aisle
$$\mathcal{U}_\Phi = \{X\in \DModA\text{: }\Supp\, H^n(X)\subseteq\Phi (n)\text{ for all }n\in\mathbb Z\}$$  coincides with the smallest suspended cocomplete subcategory of $\DModA$ which contains the set $\{A/\p[-n]\mid  n\in\mathbb Z, \, \p \in \Phi(n)\}$. Recall from \S\ref{tpandttf} that a subcategory of $\DModA$ is called localizing if it is a triangulated subcategory closed under coproducts.

\begin{theorem}\label{Neeman} Let $A$ be a commutative noetherian ring. 

(1) \cite[Theorem 2.8]{N} The assignment $\Lcal\mapsto\supp\Lcal$  defines   a 1-1 correspondence 
$$\left \{ \begin{tabular}{ccc} \text{localizing subcategories of $\DModA$}
\end{tabular}\right \}  
\longleftrightarrow  \left \{ \begin{tabular}{ccc} \text{subsets of $\Spec(A)$}
\end{tabular}\right \}.$$
The inverse map assigns to a subset 
 $P$ of $\Spec(A)$  the localizing subcategory $\Lcal_P = \Loc\{ \kappa(\p)\mid\p \in P\}$. 
 
 \smallskip
 
(2) \cite[Theorem 3.11]{AJSa} The assignment $\Phi\mapsto(\mathcal{U}_\Phi
,\mathcal{V}_\Phi)$ defines a 1-1 correspondence 
$$\left \{ \begin{tabular}{ccc} \text{compactly generated} \\ \text{t-structures in $\DModA$}
\end{tabular}\right \}  
\longleftrightarrow  \left \{ \begin{tabular}{ccc} \text{filtrations by supports of $\Spec(A)$}
\end{tabular}\right \}.$$
\end{theorem}

There is a further ingredient we will need. We have seen in Section 2 that every compactly generated t-structure  is homotopically smashing. Over a commutative noetherian ring the converse is also true.

\begin{theorem}\label{HN} \cite{HN} If $A$ is a commutative noetherian ring, every homotopically smashing t-structure in $\DModA$ is compactly generated. In particular, every pure-injective cosilting  object is  of cofinite type.
 \end{theorem}
 
We now start by proving the missing implication from Lemma~\ref{L:nondegengeneral}, which shows that the map $\Psi$ in Theorem~\ref{T:TTFduality} respects non-degeneracy.
 
 \begin{lemma}\label{L:nondegenspecial}
	Let $A$ be a commutative noetherian ring. Then, in the setting of Lemma~\ref{L:nondegengeneral},  we have 
		$\bigcap_{n \in \Z}\Ucal[n] = 0$ if and only if  $\bigcap_{n \in \Z}\Wcal'[n] = 0$. 
\end{lemma}
\begin{proof}
	We have to prove the only-if-part, or equivalently, we have to show that ${}^{\perp_\Z}\Vcal=0$ implies  $\Vcal'^{\perp_\Z}=0$.	 
Since ${}^{\perp_\Z}\Vcal=0$,
	for any prime ideal $\p \in \Spec{A}$ there is $V \in \Vcal$ such that $\rhom{A}{\kappa(\p)}{V}$ is not a zero object of $\DModA$. By \cite[Proposition 2.3]{HCG}, $\rhom{A}{\kappa(\p)}{V}$ belongs to $\Vcal$. Furthermore, the object $\rhom{A}{\kappa(\p)}{V}$ is quasi-isomorphic to a complex of vector spaces over $\kappa(\p)$ in $\DModA$, and therefore $\rhom{A}{\kappa(\p)}{V} \cong \bigoplus_{n \in \Z} H^n\rhom{A}{\kappa(\p)}{V}[-n]$. Thus
	there is $n \in \Z$ such that $H^n\rhom{A}{\kappa(\p)}{V}$ is a non-zero
	vector space over $\kappa(\p)$. Since $\Vcal$ is closed under direct summands, we conclude that for each $\p \in \Spec{A}$ there is $n \in \Z$ such that $\kappa(\p)[n] \in \Vcal$.

	Let $W$ be an injective cogenerator of $\ModA$. By Lemma~\ref{L:dualdefinable}, $\rhom{A}{\kappa(\p)[n]}{W} \cong \Hom{A}{\kappa(\p)}{W}[-n] \in \Vcal'$. Since the non-zero $A$-module $\Hom{A}{\kappa(\p)}{W}$ is again naturally a vector space over $\kappa(\p)$, we see that $\Vcal'$ contains $\kappa(\p)[-n]$. Let $\Lcal = {}^{\perp_\Z}(\Vcal'^{\perp_\Z})$. Then $\Lcal$ is a localizing subcategory of $\DModA$ and $\Vcal' \subseteq \Lcal$. In particular, $\Lcal$ contains $\kappa(\p)$ for all $\p \in \Spec(A)$, and the subset   of $\Spec(A)$ corresponding to $\Lcal$ under the bijection in Theorem~\ref{Neeman}(1) must be $P = \Spec(A)$. Thus $\Lcal = \DModA$ and $\Vcal'^{\perp_\Z} = 0$, as desired.
\end{proof}
 As a consequence, the injective map from Theorem~\ref{P:silttocosilt} is now bijective, and we   obtain a classification of silting and cosilting objects over commutative noetherian rings which extends \cite[Corollary 3.5]{ASa}.

\begin{theorem}\label{commnoeth}
	Let $A$ be a commutative noetherian ring. There is a bijective correspondence between		
	\begin{enumerate}
	\item[(i)] equivalence classes of silting objects of finite type,
	\item[(ii)] equivalence classes of pure-injective cosilting objects,
	\item[(iii)] 		filtrations by supports $\Phi$ of $\Spec(A)$ such that $$\bigcup_{n\in\mathbb Z} \Phi(n)=\Spec(A)\text{  and  }\bigcap_{n\in\mathbb Z} \Phi(n)=\emptyset,$$		\end{enumerate}
	which restricts to a bijective correspondence between		
	\begin{enumerate}
				\item[(i')] equivalence classes of bounded silting complexes,
			\item[(ii')] equivalence classes of bounded cosilting complexes,
	\item[(iii')] 	filtrations by supports $\Phi$ of $\Spec(A)$ such that	there are   integers $n\leq m$ with $$\Phi(n)=\Spec(A)\text{  and }\Phi(m)=\emptyset.$$ 	
	\end{enumerate}
\end{theorem}

\begin{proof}
The bijection between $(i)$ and $(ii)$ is the first part of Theorem~\ref{P:silttocosilt} in conjunction with Theorem~\ref{HN} and Lemma~\ref{L:nondegenspecial}.

For the bijection with $(iii)$, we have to show that  $(\Ucal_\Phi, \Vcal_\Phi)$ is non-degenerate if and only if  the filtration by supports  $\Phi$ satisfies the stated conditions.
 The  condition on the intersection of the $\Phi(n)$ follows immediately from the fact that $\bigcap_{n\in\mathbb Z}\Ucal_\Phi[n]$ consists of the objects $X\in\DModA$
 whose cohomologies are supported in $\bigcap_{n\in\mathbb Z} \Phi(n)$.
Moreover, if $\Scal$ is a set of compact objects generating $(\Ucal_\Phi, \Vcal_\Phi)$, then the condition
$\bigcap_{n\in\mathbb Z}\Vcal_\Phi[n]=0$ holds if and only if $\Scal^{\perp_{\mathbb Z}}=0$, which  amounts to $\Loc(\Scal)=\DModA$. 

We claim that  the  subset  $P$ of $\Spec(A)$ corresponding to $\Loc(\Scal)$ under Theorem~\ref{Neeman} is precisely $P=\bigcup_{n\in\mathbb Z} \Phi(n)$. Indeed, since $\Scal \subseteq \Ucal_\Phi$, and the support $\supp X$ of   a compact object $X$ is given by the classical supports $\bigcup_{n\in\Z}\Supp H^n(X)$ of its cohomologies, we have $\supp \Scal \subseteq \bigcup_{n\in\mathbb Z} \Phi(n)$. The localizing subcategory  corresponding to $\bigcup_{n\in\mathbb Z} \Phi(n)$ must therefore contain $\Loc(\Scal)$, hence $P \subseteq \bigcup_{n\in\mathbb Z} \Phi(n)$. On the other hand, the fact that $\Loc(\Scal)$ contains $\Ucal_\Phi$
and thus also  the set $\{A/\p[-n]\mid  n\in\mathbb Z,\, \p \in \Phi(n)\}$
 yields the other inclusion.  We conclude using Theorem~\ref{Neeman} that
the condition $\Loc(\Scal)=\DModA$
is equivalent to $\bigcup_{n\in\mathbb Z} \Phi(n)=\Spec(A)$.

For the second part,  we combine 
  Theorem~\ref{T:cosiltingTTF} with Theorem~\ref{HN} to see  that every bounded cosilting complex in $\DModA$ is of cofinite type. The bijection between $(i')$ and $(ii')$ then follows from the second part of Theorem~\ref{P:silttocosilt}. 
 Furthermore, the existence of   integers $n\leq m$ such that $\mathsf D^{\leq n}\subseteq \Ucal_{\Phi}\subseteq \mathsf D^{\leq m}$ means precisely that the  cohomologies of objects in $\Ucal_{\Phi}$ are arbitrary in degrees  $\le n$ and vanish in degrees $>m$. In other words, $\Phi(i)=\Spec A$ for all $i\le n$ and $\Phi(i)=\emptyset$ for all $i> m$. This proves the equivalence of $(ii')$ and $(iii')$.
\end{proof}

\subsection{Over hereditary rings}\label{hereditary} 
We know from Theorem~\ref{HN} that  homotopically smashing t-structures over  commutative noetherian rings are compactly generated. The main result of this section establishes the same result over hereditary rings.

	In order to be consistent with later sections, it is convenient to switch to the unbounded derived category $\DAMod$ of \emph{left} $A$-modules over a ring $A$. We assume that $A$ is a left hereditary ring. Then the structure of the derived category simplifies considerably. Indeed, in this case, for any object $X \in \DAMod$ we have isomorphisms $X \cong \bigoplus_{n \in \Z}H^n(X)[-n] \cong \prod_{n \in \Z}H^n(X)[-n]$ (see e.g.~\cite[\S 1.6]{K2}). As a consequence, for all $X,Y \in \DAMod$ we have
	$$\Hom{\DAMod}{X}{Y} \cong \Hom{\DAMod}{\bigoplus_{n \in \Z}H^n(X)[-n]}{\prod_{n \in \Z}H^n(Y)[-n]} \cong$$ 
	$$\cong \prod_{n \in \Z}(\Hom{\Aop}{H^n(X)}{H^n(Y)} \oplus \Ext{1}{\Aop}{H^n(X)}{H^{n-1}(Y)}).$$
	Let $(\Ucal,\Vcal)$ be a t-structure. We fix the notation $$\Ucal_n = \{H^n(X) \mid X \in \Ucal\}\text{ and }\Vcal_n = \{H^n(X) \mid X \in \Vcal\}.$$ Then the left heredity of $A$ implies that $$\Ucal = \{X \in \DAMod \mid H^n(X) \in \Ucal_n \text{ for all } n \in \Z\} \text{ and}$$ $$\Vcal = \{X \in \DAMod \mid H^n(X) \in \Vcal_n \text{ for all } n \in \Z\},$$ as well as the formulas
	\begin{equation}\label{unvn}\Ucal_n = {}^{\perp_0}\Vcal_n \cap {}^{\perp_1}\Vcal_{n-1} \text{ and } \Vcal_n = \Ucal_n^{\perp_0} \cap \Ucal_{n+1}^{\perp_1}.\end{equation}
	\begin{lemma}\label{approxsummand}
		Let $A$ be a ring, $(\Ucal,\Vcal)$ a t-structure in $\DAMod$, and $X \in \DAMod$. Suppose that $C$ is a direct summand of $\tau_{\Vcal}(X)$ such that $\Hom{\DAMod}{X}{C} = 0$. Then $C = 0$.
	\end{lemma}
	\begin{proof}
		Consider the approximation triangle
		$$\tau_{\Ucal}(X) \rightarrow X \xrightarrow{g} \tau_{\Vcal}(X) \rightarrow \tau_{\Ucal}(X)[1].$$
		Since $\tau_{\Vcal}(X) = C \oplus C'$, and $\Hom{\DAMod}{X}{C} = 0$, a general argument in triangulated categories shows that $C[-1]$ is a direct summand of $\tau_{\Ucal}(X)$. Since $\tau_{\Ucal}(X) \in \Ucal$, and $C[-1] \in \Vcal$, this forces $C = 0$.
	\end{proof}
	\begin{remark}\label{SSVexplanation}
		For the proof of the next theorem, we will need to invoke a deep theorem \cite[Theorem A]{SSV} which allows to lift t-structures in $\DAMod$ to the category $\DD(\AMod^I)$ of coherent diagrams of shape $I$, where $I$ is a small category. We provide a short explanation for the process adjusted for our application here. The reader is referred to \cite{SSV} for the unexplained terminology of coherent diagrams, and to \cite[Example 2.4]{SSV} in particular for the situation of the canonical derivator of the derived category of a Grothendieck category. Let $(\Ucal,\Vcal)$ be a t-structure in $\DAMod$, and let $I$ be a small category. By \cite[Theorem A]{SSV}, there is a t-structure $(\Ucal_I,\Vcal_I)$ in $\DD(\AMod^I)$, where
$$\Ucal_I  = \{\mathscr{U} \in \DD(\AMod^I) \mid \mathscr{U}_i \in \Ucal \text{ for all } i \in I\},$$
$$\Vcal_I  = \{\mathscr{V} \in \DD(\AMod^I) \mid \mathscr{V}_i \in \Vcal \text{ for all } i \in I\},$$
where $\mathscr{X}_i$ is the $i$-th component of the coherent diagram $\mathscr{X} \in \DD(\AMod^I)$. Therefore, for any $\mathscr{X} \in \DD(\AMod^I)$ there is an approximation  triangle
$$\mathscr{U} \rightarrow \mathscr{X} \rightarrow \mathscr{V} \rightarrow \mathscr{U}[1],$$
where $\mathscr{U} \in \Ucal_I$ and $\mathscr{V} \in \Vcal_I$. By \cite[Corollary 4.19]{Groth}, taking the $i$-th coordinate yields a triangle
$$\mathscr{U}_i \rightarrow \mathscr{X}_i \rightarrow \mathscr{V}_i \rightarrow \mathscr{U}_i[1]$$
in $\DAMod$ for any $i \in I$. Since $\mathscr{U}_i \in \Ucal$ and $\mathscr{V}_i \in \Vcal$, this triangle is necessarily isomorphic to the approximation triangle of $\mathscr{X}_i$ with respect to the t-structure $(\Ucal,\Vcal)$ in $\DAMod$. Thus, $\mathscr{U}_i \cong \tau_{\Ucal}(\mathscr{X}_i)$. Let $\alpha$ be an arrow in $I$. Since $\tau_{\Ucal}: \DAMod \rightarrow \Ucal$ is the right adjoint to the inclusion of $\Ucal$, it follows by simple diagram chasing that $\mathscr{U}(\alpha) = \tau_{\Ucal}(\mathscr{X}(\alpha))$. Therefore, the coherent diagram $\mathscr{U}$ is given by applying the functor $\tau_{\Ucal}$ onto the coherent diagram $\mathscr{X}$. The analogous statement for $\mathscr{V}$ follows by a dual argument.
	\end{remark}
	\begin{theorem}\label{hereditaryutc}
		Let $A$ be a left hereditary ring. Then any homotopically smashing t-structure in $\DAMod$ is compactly generated. In particular, every pure-injective cosilting object is of cofinite type.
	\end{theorem}
	\begin{proof}
			Let $(\Ucal,\Vcal)$ be a homotopically smashing t-structure in $\DAMod$. We claim that for any $M \in \Ucal_n$, we can write $M = \varinjlim_{i \in I} M_i$ for a directed system $(M_i \mid i \in I)$ consisting of finitely presented modules from $\Ucal_n$. This is enough for the compact generation of $(\Ucal,\Vcal)$ --- indeed, by the left heredity of $A$, any stalk of a finitely presented left $A$-module is a compact object of $\DAMod$, and since aisles are closed under directed homotopy colimits (\cite[Proposition 4.2]{SSV}), we have that $(\Ucal,\Vcal)$ is compactly generated.

		To prove the claim, we first use \cite[Lemma 5.2]{KSt} to write $M = \varinjlim_{i \in I} F_i$, where $F_i$ is a finitely presented module such that $F_i \in {}^{\perp_1}\Vcal_{n-1}$. For each $i \in I$, consider the approximation triangle of the stalk complex $F_i[-n]$ with respect to the t-structure $(\Ucal,\Vcal)$:
			$$\tau_{\Ucal}(F_i[-n]) \rightarrow F_i[-n] \rightarrow \tau_{\Vcal}(F_i[-n]) \rightarrow \tau_{\Ucal}(F_i[-n])[1].$$
		Passing to cohomology, we obtain a long exact sequence of form
		$$\cdots \rightarrow 0 \rightarrow H^{n-1}\tau_{\Vcal}(F_i[-n]) \rightarrow H^n\tau_{\Ucal}(F_i[-n]) \rightarrow $$ $$\rightarrow F_i \rightarrow H^n\tau_{\Vcal}(F_i[-n]) \rightarrow H^{n+1}\tau_{\Ucal}(F_i[-n]) \rightarrow 0 \rightarrow \cdots$$
			Recall that we have $\tau_{\Vcal}(F_i[-n]) \cong \bigoplus_{k \in \Z}H^k(\tau_{\Vcal}(F_i[-n]))[-k]$, and also $\Hom{\DAMod}{F_i[-n]}{H^{n-1}\tau_{\Vcal}(F_i[-n])[-n+1]} \cong \Ext{1}{\Aop}{F_i}{H^{n-1}(\tau_{\Vcal}(F_i[-n]))} = 0$, as $H^{n-1}(\tau_{\Vcal}(F_i[-n])) \in \Vcal_{n-1}$. Then Lemma~\ref{approxsummand} applies and shows that $H^{n-1}(\tau_{\Vcal}(F_i[-n])) = 0$. As a consequence, $H^n(\tau_{\Ucal}(F_i[-n]))$ is isomorphic to a submodule of $F_i$.  Now we use \cite[Theorems 2.1.4 and 5.1.6]{Cohn} to see that every submodule of a finitely presented module over a left hereditary ring has a direct decomposition in finitely presented modules. Hence $H^n(\tau_{\Ucal}(F_i[-n]))$ is isomorphic to a direct sum of finitely presented modules, and these have to belong to $\Ucal_n$.

			Now we consider $\mathscr{F} = (F_i[-n] \mid i \in I)$ as an object in $\DD(\lmod{A}^I)$, the category of all coherent diagrams in $\DAMod$ of shape $I$. Denote by $\Ucal_I$ the subcategory of $\DD(\lmod{A}^I)$ consisting of those coherent diagrams such that all their coordinates belong to $\Ucal$, and define $\Vcal_I$ similarly. By \cite[Theorem A]{SSV}, the pair $(\Ucal_I,\Vcal_I)$ forms a t-structure in the triangulated category $\DD(\lmod{A}^I)$. Consider the approximation triangle of $\mathscr{F}$ with respect to the t-structure $(\Ucal_I,\Vcal_I)$:
		$$\mathscr{U} \rightarrow \mathscr{F} \rightarrow \mathscr{V} \rightarrow \mathscr{U}[1].$$
		By \cite{Groth}, the directed homotopy colimit functor is exact, and thus passing to directed homotopy colimits yields a triangle in $\DAMod$:
			\begin{equation}\label{hocolim}\hocolim_{i \in I}\mathscr{U} \rightarrow M[-n] \rightarrow \hocolim_{i \in I}\mathscr{V} \rightarrow \hocolim_{i \in I}\mathscr{U}[1].\end{equation}
					Since $(\Ucal,\Vcal)$ is homotopically smashing, both $\Ucal$ and $\Vcal$ are closed under directed homotopy colimits. Therefore, $\hocolim_{i \in I}\mathscr{U} \in \Ucal$ and $\hocolim_{i \in I}\mathscr{V} \in \Vcal$. Then (\ref{hocolim}) is an approximation triangle of $M[-n]$ with respect to $(\Ucal,\Vcal)$. But since $M \in \Ucal_n$, and thus $M[-n] \in \Ucal$, we have an isomorphism $M[-n] \cong \hocolim_{i \in I}\mathscr{U}$. Passing to the $n$-th cohomology, we obtain 
		$$M \cong H^n(\hocolim_{i \in I}\mathscr{U}) \cong \varinjlim_{i \in I} H^n(\tau_{\Ucal}(F_i[-n])).$$
		But as we have shown above, for each $i \in I$ the module $H^n(\tau_{\Ucal}(F_i[-n])$ is isomorphic to a direct sum of finitely presented modules, all of which belong to $\Ucal_n$. Therefore, we have a presentation of $M$ as a direct limit of modules from $\Ucal_n \cap \lfmod{A}$, as desired.
	\end{proof}


\section{Cosilting modules and ring epimorphisms} 
 Inspired by the classification results for commutative noetherian rings in Section~\ref{commnoether}, we  proceed to investigate possible parametrizations of  cosilting objects over further classes of rings. Instead of chains of subsets of the prime spectrum, we will use chains of ring epimorphisms.  

In this section, we start by investigating  the case of a single ring epimorphism. After some preliminaries in subsection~\ref{epi}, we discuss a construction of cosilting modules from  ring epimorphisms  in subsection~\ref{minimalcosiltingmod}.  Over rings of weak global dimension at most one, or over commutative noetherian rings,  this leads us to a bijection between homological ring epimorphisms and  certain cosilting modules   (Corollaries~\ref{weak} and~\ref{flatcomm}). Such cosilting modules will be termed ``minimal'', as their construction  is dual to the construction of minimal silting modules over hereditary rings  in \cite{AMV2}. In fact, over a hereditary ring minimal silting and cosilting modules will correspond to each other under  silting-cosilting duality (Corollary~\ref{hered}).
 
\subsection{Reminder on ring epimorphisms}\label{epi}
Let us first  recall some notions and basic results.

\begin{definition}
(1) A {ring homomorphism} $\lambda:A\longrightarrow B$  is a 
{\em ring epimorphism} if it is an epimorphism in the category of rings with unit, or equivalently, if the functor given by restriction of scalars $\lambda_\ast:\ModB\longrightarrow \ModA$ is fully faithful.
Further, $\lambda$ is a 
{\em homological ring epimorphism} if in addition $\Tor{i}{A}{B}{B}=0$ for all $i>0$, or equivalently,  the functor given by restriction of scalars  $\lambda_\ast:\DModB\rightarrow \DModA$ is a full embedding.

(2)
Two ring epimorphisms $\lambda_1:A\longrightarrow B_1$ and  $\lambda_2:A\longrightarrow B_2$ are said to be  {\em equivalent} if there is an isomorphism of rings $\mu: B_1\longrightarrow B_2$ such that $\lambda_2=\mu\circ \lambda_1$. We then say that $\lambda_1$ and $\lambda_2$ lie in the same {\em epiclass} of $A$.

(3)
A full subcategory $\Xcal$ of $\ModA$ is called {\em bireflective} if the inclusion functor $\Xcal\longrightarrow\ModA$ admits both a left and right adjoint or, equivalently, if it is closed under products, coproducts, kernels and cokernels. 
\end{definition}

\begin{theorem}\label{epicl}
 \cite{GdP,GL,BD,Sch} The assignment which takes  a ring epimorphism  $\lambda:A\rightarrow B$ to  the essential image $\Xcal_B$ of $\lambda_\ast$ defines a bijection between
\begin{itemize}
\item epiclasses of ring epimorphisms $A\rightarrow B$,
\item bireflective subcategories of $\ModA$,
\end{itemize}
which restricts to a bijection between  
\begin{itemize}
\item epiclasses of ring epimorphisms $A\rightarrow B$ with $\Tor{1}{A}{B}{B}=0$,
\item bireflective subcategories closed under extensions in $\ModA$.
\end{itemize}
\end{theorem}

The partial order  on bireflective subcategories   given by inclusion corresponds under the bijection in Theorem~\ref{epicl} to a   partial order on the epiclasses of $A$ defined by setting   $$\lambda_1\le  \lambda_2$$  whenever there is a commutative diagram of ring homomorphisms
$$\xymatrix{A\ar[rr]^{\lambda_1}\ar[dr]_{\lambda_2} & & B_1\\ & B_2\ar[ur]_{\mu} & }$$
Since bireflective subcategories are determined by closure properties,
 the poset induced by $\le$ is a complete lattice, and  the ring epimorphisms $A\to B$ with $\Tor{1}{A}{B}{B}=0$ form a sublattice in it by Theorem~\ref{epicl}.
 
Notice that over a  ring of weak global dimension at most one, every ring epimorphism $A\to B$ with $\Tor{1}{A}{B}{B}=0$ is already a homological ring epimorphism and thus has the properties listed in the next theorem. The same holds true when $A$ is commutative noetherian, since $B$ is then a flat $A$-module by \cite[Proposition 4.5]{AMSTV}. For details on recollements, we refer to \cite{NS}.  
 \begin{theorem}\label{recoll} 
 \cite{NS} Let $A$ be a ring and let  $\lambda:A\to B$ be a homological ring epimorphism. The functor $\lambda_*: \DModB\to \DModA$  given by restriction of scalars, together with the adjoint functors  $\lambda^*=-\lten{A} B$ and $\lambda^!={\mathbf R} \Hom{A}{B}{-}$, induces a stable TTF triple in $\DModA$
$$(\mathcal L=\Ker\lambda^\ast,\; \Bcal=\im\lambda_\ast,\; \mathcal K=\Ker\lambda^!)$$ and  a 
recollement
 $$\xymatrix@!=6pc{\DModB
\ar[r]|{\lambda_*=\lambda_!} &\DModA \ar@<+2.5ex>[l]|{\lambda^!}
\ar@<-2.5ex>[l]|{\lambda^*} \ar[r]|{j^!=j^*} & 
{\:\mathcal L}\ar@<+2.5ex>[l]|{j_*} \ar@<-2.5ex>[l]|{j_!} }
$$ with  $j^*=-\lten{A} L[-1]$ where $L$ is  the cone  of $\lambda$. 
In particular, for every  $X$ in $\DModA $ there is an approximation triangle with respect to the stable t-structure $(\mathcal L,\Bcal)$
\[
\xymatrix@R=0.5pc{
j_! j^! (X)=X\lten{A} L[-1]\ar[r]& X\ar[r]& \lambda_\ast \lambda^\ast(X)=X\lten{A} B\ar[r]& j_!j^!(X)[1]}
\]
and an approximation triangle with respect to the stable t-structure $(\Bcal, \mathcal K)$
\[
\xymatrix@C=0.9pc{
\lambda_\ast \lambda^!(X)= {\mathbf R} \Hom{A}{B}{X}\ar[r]& X\ar[r]& j_\ast j^\ast (X)={\mathbf R} \Hom{A}{L[-1]}{X}\ar[r]& \lambda_\ast \lambda^!(X)[1]}
\]
where the maps are given by adjunctions.
\end{theorem}

For a full subcategory  $\Xcal$ of $\ModA$ we consider the full subcategory of $\DModA$ $$\Dcal_\Xcal(A)=\{X\in\DModA\mid H^n(X)\in\Xcal\text{ for all } n\in\mathbb Z\}.$$ 

\begin{theorem} \label{KS}
Let $A$ be a ring, let $\lambda: A \rightarrow B$ be a homological ring epimorphism and let $\Xcal$ be the corresponding bireflective subcategory of $\ModA$.
\begin{enumerate}
\item[(i)] \cite[Lemma 4.6]{AKL1} The subcategory $\Bcal=\im\lambda_\ast$ equals $\Dcal_\Xcal(A)$.
	\item[(ii)] \cite[Lemma 4.2]{AKL1} If $B_A$ has projective dimension at most one, the subcategory $\Kcal = \Ker \lambda^! = \Ker \rhom{A}{B}{-}$  equals $\Dcal_{\Ccal}(A)$ where $\Ccal = \ModA\,\cap\,\Kcal = B^{\perp_{0,1}}.$
	\item[(iii)] \cite[Theorem 6.1(a)]{BP}\label{T:BP} If ${}_A B$ has weak dimension at most one, the subcategory $\Lcal = \Ker \lambda^\ast = \Ker (- \otimes_A^\LL B)$ equals  $\Dcal_{\Ecal}(A)$ where $\Ecal = \ModA\,\cap\,\Lcal
	= \{M \in \ModA \mid M \otimes_A B = 0 = \Tor{1}{A}{M}{B}\}$.
\item[(iv)] \cite[Propositions 2.6 and 3.1]{KSt}
 If $A$ is  hereditary, the stable TTF triple $(\Lcal,\Bcal,\Kcal)$  is given by $$\Lcal=\Dcal_{^{\perp_{0,1}}\Xcal}(A),\; \Bcal=\Dcal_\Xcal(A),\;\Kcal=\Dcal_{\Xcal^{\perp_{0,1}}}(A).$$ 
\end{enumerate}
\end{theorem}

\begin{theorem}{\cite[Theorem~4.1]{Sch}}\label{def:universallocalisation}
Let $A$ be a ring and $\Sigma$ be a class of morphisms between
finitely generated projective right $A$-modules. Then
there is a ring $A_\Sigma$ and a ring homomorphism 
$f: A\longrightarrow A_\Sigma$ such that
\begin{enumerate}
\item $f$ is \emph{$\Sigma$-inverting,} i.e.~if
$\sigma$ belongs to  $\Sigma$, then
$\sigma\otimes_A A_\Sigma$ is an isomorphism of right
$A_\Sigma$-modules, and \item $f$ is \emph{universal
$\Sigma$-inverting}, i.e.~for any $\Sigma$-inverting morphism $f': A\longrightarrow B$
there exists a unique ring homomorphism ${g}:
A_\Sigma\longrightarrow B$ such that $g\circ f=f'$.
\end{enumerate}
The homomorphism
$f\colon A\longrightarrow A_\Sigma$ is a ring epimorphism
with  $\Tor{1}{A}{A_\Sigma}{{A_\Sigma}}=0$, called
the \textbf{universal localization} of $A$ at
$\Sigma$.
\end{theorem}

Recall that a ring $A$ is said to be \emph{right semihereditary} if all its finitely generated right ideals are projective. It is well known that every finitely generated submodule of a projective right $A$-module is then projective, and therefore
$\rfmod{A}$ consists of modules of projective dimension at most one.

\begin{theorem} \label{univlochom}
Let $A$ be a right semihereditary ring. There is a bijection between
\begin{enumerate}
\item  wide subcategories (i.e.~abelian subcategories closed under extensions)  of $\rfmod{A}$,
\item epiclasses of universal localizations of $A$,
\end{enumerate}
which assigns to a wide subcategory 
$\Mcal\subseteq\rfmod{A}$ 
the epiclass of the universal localization $\lambda_\Mcal$ at the projective resolutions of the modules in $\Mcal$.
If $A$ is right hereditary, then there is also a  bijection with
\begin{enumerate}
\item[(3)] bireflective extension closed subcategories of $\ModA$,
\end{enumerate}
which assigns to a wide subcategory 
$\Mcal\subseteq\rfmod{A}$ the perpendicular category $\Mcal^{\perp_{0,1}}$.
 Conversely, given  a
 bireflective subcategory $\Xcal$, the associated
 wide subcategory is $\Mcal={}^{\perp_{0,1}}\Xcal\cap\rfmod{A}$.\end{theorem} 
\begin{proof} 
The result goes back to \cite[Theorem 2.3]{Scho2},\cite[Theorem 6.1]{KSt} for the case when $A$ is right hereditary. For the semihereditary case one proceeds similarly. Indeed, first of all, 
one easily checks that the proof in \cite[Lemma 2.1]{Scho2} is still valid. Given a universal localization $\lambda:A\to B$,  we can thus assume that $\lambda$ is the universal localization at a set $\Sigma$ of injective morphisms between
finitely generated projective right $A$-modules. We can then consider the bireflective subcategory $\Xcal$ of $\ModA$ associated to $\lambda$ together with its left perpendicular subcategory   $\Mcal={}^{\perp_{0,1}}\Xcal\cap\rfmod{A}$, which  consists of all finitely presented modules $M$ whose projective resolutions are inverted by $B$, that is, $M\otimes_AB=\Tor{1}{A}{M}{{B}}=0$, cf.~\cite[Theorem 5.2]{Scho1}. Since $A$ is right coherent and all finitely presented right $A$-modules have projective dimension at most one, we infer from~\cite[Lemma 5.3]{Scho1} that $\Mcal$ is a wide subcategory of $\rfmod{A}$. By construction,
$\lambda$ lies in the same epiclass as $\lambda_\Mcal$. This shows that the assignment $\Mcal\mapsto\lambda_\Mcal$ in the statement is surjective. The injectivity is shown as in \cite[Theorem 2.3]{Scho2}.
\end{proof}


\subsection{Minimal cosilting modules}\label{minimalcosiltingmod}
 Ring epimorphisms with nice homological properties can be used to  construct silting modules \cite{AMV2,abundance}. We now investigate the dual construction, developing  work from \cite{B} on the cotilting case. This is going to shed some light on the connection between ring epimorphisms and  cosilting modules. In fact, it will turn out that the class of cosilting modules obtained from the dual construction is in general larger  than the class of silting modules arising from ring epimorphisms, cf.~Examples~\ref{exampleminimal}(4) and (5).
\begin{definition}  Let $M$ be a right $A$-module, and let
$\Ccal$ be a class of left $A$-modules. We say that $M$ is a \emph{$\Ccal$-Mittag-Leffler module} if the canonical map
\[\rho \colon  M\otimes_A \prod _{i\in I}C_i\to \prod _{i\in
I}(M\otimes_AC_i)\] is injective for any family $\{C_i\}_{i\in
I}$ of modules in $\Ccal$.
\end{definition}

We refer the reader to \cite[\S 5, \S 6, \S 13]{GT} for the basics on the notions of cotorsion pairs, (pre)envelopes and (pre)covers, and tilting cotorsion pairs in module categories.

\begin{lemma} \label{ML} Let $T$ be a silting right $A$-module, and let $C=T^+$. Then every module in $\Add T$ is $\Cogen C$-Mittag-Leffler.
\end{lemma}
\begin{proof}
We know from \cite{AMV1} that $T$ is a tilting module over $\overline{A}=A/I$ where $I=\{a\in A\mid Ta=0\}$ is the annihilator of $T$ in $A$. We consider the tilting cotorsion pair $(\Acal,\Gen T_{\overline{A}})$ induced by $T$  in $\rmod{\overline{A}}$ and the dual cotilting class $\Cogen _{\overline{A}}C$. By \cite[Corollary 9.8]{relml}, 
every module in $\Acal$ is $\Cogen _{\overline{A}}C$-Mittag-Leffler. Now, any module $M\in\Add T$ is an $\overline{A}$-module contained in $\Acal$. 
Moreover, since $\rmod{\overline{A}}$ is closed under products and submodules, and $C= \Hom{k}{T}{W}$ is an $\overline{A}$-module, every module $X$ in $\Cogen _AC$ is contained in $\Cogen _{\overline{A}}C$ and satisfies  $M\otimes_A X\cong M\otimes_{\overline A} X$.
Hence the Mittag-Leffler property over $\overline{A}$ entails the Mittag-Leffler property over ${A}$.
\end{proof}

\begin{proposition}\label{precosilting}
Let  $T$ be a silting right $A$-module and let 
\begin{equation}\label{approxseq1}
A\stackrel{f}{\longrightarrow} T_0\to T_1\to 0
\end{equation}
	be an exact sequence such that $f$ is a $\Gen{T}$-preenvelope and $T_1$ lies in $\Add T$. 
Set $C=T^+$ and $C_i=(T_i)^+$ for $i=0,1$ and consider the exact sequence
\begin{equation}\label{coapproxseq}
0\to C_1\to C_0\stackrel{f^+}{\longrightarrow}  A^+.
\end{equation}
Then $f^+$ is a $\Cogen C$-precover, $C_1$ lies in $\Prod C$, and the subcategory $\Ycal=\Cogen C\cap{}^{\perp_0}C_1$ is bireflective.  
\end{proposition}
\begin{proof}
We know from Proposition~\ref{dually} that $C$ is a cosilting module with cosilting class $\Cogen C=(\Gen T)^\vee$. 
By Hom-$\otimes$-adjunction, for any left $A$-module $X$ there is a commutative diagram linking the maps $\Hom{A}{X}{f^+}$, $(f\otimes_A X)^+$ and $\Hom{A}{f}{X^+}$. Thus $X\in \Cogen C$ if and only if $X^+\in \Gen T$, which in turn means that $\Hom{A}{f}{X^+}$, or equivalently, $\Hom{A}{X}{f^+}$ is surjective. 

Dually to \cite[Proposition 3.3]{AMV2}, we see that $\Ycal$ is closed under coproducts, kernels and cokernels. By Hom-$\otimes$-adjunction, ${}^{\perp_0}C_1$ consists of the left $A$-modules $Y$ for which $T_1\otimes Y=0$. Now assume that $(Y_i)_{i\in I}$ is a family of left $A$-modules  in $\Ycal$ and consider $Y=\prod_{i\in I}Y_i$. Of course $Y$ belongs to $\Cogen C$. Moreover, 
since  $T_1$ is  $\Cogen C$-Mittag-Leffler by Lemma~\ref{ML}, we have an injective map 
$\rho \colon  T_1\otimes_A Y\to \prod _{i\in
I}(T_1\otimes_AY_i)=0$, showing that $Y$ also belongs to ${}^{\perp_0}C_1$, and therefore to  $\Ycal$.
\end{proof}

\begin{example}\label{Kronecker}
Let $A$ be the Kronecker algebra, i.e., the path algebra of the quiver $\xy\xymatrixcolsep{2pc}\xymatrix{ \bullet \ar@<0.5ex>[r]  \ar@<-0.5ex>[r] & \bullet } \endxy$ over a field $k$. 
Every homological ring epimorphism $\lambda: A\to B$ 
induces a silting module of the form $T=B\oplus\Coker\lambda$. The silting modules arising in this way are termed ``minimal'' (cf.~\cite{AMV2} and Definition~\ref{defmin}).
There is just one   
 silting module  over $A$ (up to equivalence) which is not minimal: the Lukas tilting module $L$ whose
tilting class $\Gen L$ is given by the right $A$-modules without  indecomposable preprojective summands.
 
Consider now the cotilting left $A$-module $W=L^+$.  Up to equivalence, $W$ is the direct sum of the generic module $G$ and all Pr\"ufer modules $S_\infty$, where $S$ runs through a set of representatives of the simple regular modules (for more details, we refer to Chapter 6).  If we apply Proposition~\ref{precosilting} on an exact sequence $0\to A\stackrel{f}{\longrightarrow} L_0\to L_1\to 0$  such that $f$ is a $\Gen{L}$-preenvelope and $L_1$ lies in $\Add L$,
 we obtain an exact sequence $0\to C_1\to C_0\stackrel{f^+}{\longrightarrow}  A^+\to 0$ as in (\ref{coapproxseq}) where  $\Ycal=\Cogen W\cap {}^{\perp_0}C_1=0$. Indeed, this follows from 
 $\Gen L\cap L_1^{\perp_0}= 0$, which is shown in \cite[Example 5.10]{AMV2}.
 
On the other hand, $A^+$ also has a $\Cogen W$-cover $g$ with an exact sequence $0\to W_1\to W_0\stackrel{g}{\longrightarrow}  A^+\to 0$, where $W_1\in\Add G$ and $W_0$ is a direct sum of Pr\"ufer modules by \cite[Theorem 7.1]{RR}. 
Here $\Cogen W\cap {}^{\perp_0}W_1={}^{\perp_{0,1}}W_1={}^{\perp_{0,1}}G$ contains all Pr\"ufer modules, and it is not closed under direct products, because for every simple regular module $S$ the countable product $S_\infty\,^{\mathbb N}$ contains a copy of the generic module as a direct summand by \cite[Proposition 3]{R}. 
Notice that $\Hom{A}{W_0}{W_1}=0$, while $C_0$ and $C_1$ must contain a Pr\"ufer module as a common direct summand.
\end{example}


In view of the discussion above, we introduce the following terminology.

\begin{definition}\label{cominimal} A cosilting left $A$-module $C$  is said to be {\it minimal}
if there is an exact sequence \begin{equation}0\to C_1\to C_0\stackrel{g}{\longrightarrow}  A^+\end{equation}
such that  $g$ is a $\Cogen C$-precover, $C_1$ lies in $\Prod C$, and
\begin{enumerate}  \item[(i)] the subcategory $\Cogen C\cap{}^{\perp_0}C_1$ is bireflective, that is, it is closed under direct products, \item[(ii)] $\Hom{A}{C_0}{C_1}=0$. \end{enumerate}\end{definition}

The following conditions are needed to construct minimal cosilting modules from ring epimorphisms.

 \begin{definition} With the notation from Definition~\ref{def p silting}, we say that
\begin{enumerate}
\item	
a projective presentation $ P\stackrel{\sigma}{\longrightarrow} P'\to T\to 0$ of an $A$-module $T$ is  a \emph{presilting presentation} if 
$\Hom{\DModA}{\sigma}{{\sigma}^{(I)}[1]}=0$ for all sets $I$, or equivalently, $\Gen T\subseteq\mathcal{D}_{\sigma}$;
\item
an injective copresentation $0\to C\to  I\stackrel{\omega}{\longrightarrow} I'$ of an $A$-module $C$ is  a \emph{precosilting copresentation} if 
$\Hom{\DModA}{{\omega}^{I}}{\omega[1]}=0$ for all sets $I$, or equivalently, $\Cogen C\subseteq\Ccal_{\omega}$.
\end{enumerate}
\end{definition}

\begin{remark}\label{min}
  A module $C$ has a precosilting copresentation if and only if its minimal injective copresentation is precosilting, which
 amounts to the condition 
$\Cogen C\subset {}^{\perp_1} C$ by \cite[Lemma 3.3]{abundance}. 
\end{remark}

\begin{example}\label{copreco}
(1) Every (co)silting module has a pre(co)silting (co)presentation.

\smallskip

 (2) If $\lambda:A\to B$ is  a ring epimorphism such that  $B$ has a presilting presentation $\sigma$, then  $T=B\oplus \Coker\lambda$ is a silting right $A$-module with silting class
 $\Gen T=\Gen B\subseteq\Dcal_\sigma$ by  \cite[Proposition 1.3]{AMV4}. 
 
Furthermore,   $\Cogen B^+\subseteq\Ccal_{\sigma^+}$.
Indeed, $\Cogen B^+ = \Cogen T^+$ is a cosilting class with dual definable subcategory $\Gen T$  according to  Proposition~\ref{dually}.
 Hence  $X \in \Cogen B^+$ entails  $X^+\in\Dcal_\sigma$, that is,  the morphisms   $\Hom{A}{\sigma}{X^+}$ and $\Hom{\Aop}{X}{\sigma^+}$ are  surjective,  and $X\in\Ccal_{\sigma^+}$.
 
 We conclude that if $\lambda:A\to B$ is a ring epimorphism such that $B$ has a presilting presentation, then $B^+$ has a precosilting copresentation.
 Example~\ref{exampleminimal}(4) below will show that the converse is not true.

\smallskip
 
 (3)  If $\lambda:A\to B$ is a ring epimorphism  such that the left $A$-module $B^+$ has a precosilting copresentation, then $\Tor{1}{A}{B}{B}=0$. Indeed, the assumption entails that $_AB\in\Cogen B^+\subset {}^{\perp_1} B^+$, hence  $\Tor{1}{A}{B}{B}^+\cong\Ext{1}{A}{B}{B^+}=0$.
 
 \smallskip
 
 (4) A ring epimorphism  $\lambda:A\to B$  such that the weak dimension of $B_A$ is at most one is homological if and only if the left $A$-module $B^+$ has a precosilting copresentation. The if-part follows from (3). For a proof of the only-if-part, we  show that $\Cogen B^+\subset {}^{\perp_1} B^+$.  To this end, observe first that $B^+$ is an injective left $B$-module, hence $\Ext{1}{A}{X}{B^+}\cong\Ext{1}{B}{X}{B^+}$ vanishes for all $X\in\Prod B^+$. Now the formula $\Ext{1}{A}{X}{B^+}\cong \Tor{1}{A}{B}{X}^+$   and the assumption on the weak dimension of $B_A$ ensure that  $\Ext{1}{A}{X}{B^+}$ vanishes even for $X\in\Cogen B^+$. 
 
\smallskip

(5) A ring epimorphism   $\lambda:A\to B$  such that the projective dimension of $B_A$ is at most one  is  homological if and only if $B$ has a presilting presentation. Indeed, if $\lambda$ is homological, then  $\Ext{1}{A}{B}{X}\cong\Ext{1}{B}{B}{X}$ vanishes for all $X\in\Add B$, hence also for $X\in\Gen B$. This shows that $\Gen B\subseteq B^{\perp_1}$, hence any projective resolution of $B_A$ is a presilting presentation. The converse implication follows from (2) and (3).
\end{example}

Given a  bireflective subcategory $\Xcal$ of $\ModA$ and an $R$-module $M$, we call the unit morphism $\eta: M \rightarrow M_\Xcal$ with respect to the left adjoint to the inclusion $\Xcal$ into $\ModA$ the $\Xcal$\emph{-reflection} of $M$; the $\Xcal$\emph{-coreflection}  $\epsilon: M^\Xcal \rightarrow M$ is defined dually. Note that if $\Xcal$ corresponds to a ring epimorphism $\lambda: A \rightarrow B$ via Theorem~\ref{epicl}, then the $\Xcal$-reflection of any $R$-module $M$ is equivalent to the natural map $\eta: M \rightarrow B \otimes_A M$. One can now generalize and refine \cite[Theorem 3.3]{B}  as follows.

\begin{proposition} \label{construct} 
Let  $\lambda:A\longrightarrow B$ be a  ring epimorphism such that the left $A$-module
$B^+$ has a precosilting copresentation. Denote by $\Xcal$ the associated bireflective subcategory of $\AMod$ and set $\Ccal=\Cogen \Xcal$. Then

(1)
 $B^+\oplus \Ker\lambda^+$ is a minimal cosilting  left $A$-module with cosilting class $\Ccal$.
 
 (2) The classes $\Xcal$  and $\Ccal$ consist precisely of the left $A$-modules $M$ whose $\Xcal$-reflection 
 $\eta: M \rightarrow B\otimes_A M$ is bijective, respectively  injective.

(3) $(\Ker(B\otimes_A-),\Ccal)$ is a torsion pair.

 (4) For every module $M\in\ModA$ there is an exact sequence $0\to M'\to M\stackrel{\eta}{\rightarrow} B\otimes_AM\to M''\to 0$ where $M'$ and $M''$ belong to $\Ker(B\otimes_A-)$.
 
 (5) If $0\to X\to C\to Z\to 0$ is a short exact sequence with $X\in \Xcal$ and $C\in\Ccal$, then $Z\in\Ccal$.
 
 (6) If the weak dimension of $B_A$ is at most one, then $\Ccal\subseteq\Ker \Tor{1}{A}{B}{-}$.
\end{proposition}

\begin{proof}
(1) By Remark~\ref{min} we know that 
$\Cogen B^+\subset {}^{\perp_1} B^+$.
By arguments as in the proof of Proposition~\ref{precosilting}, we see that a left $A$-module $X$  belongs to $\lambda_\ast({\lmod{B}})$ if and only if  $\Hom{A}{X}{\lambda^+}$ is bijective. 
In particular, $\lambda^+:B^+\to A^+$ is a $\Prod B^+$-cover and a $\Cogen  B^+$-cover.
By \cite[Proposition 3.5]{abundance} it   follows that $B^+\oplus \Ker\lambda^+$ is a  cosilting left $A$-module.

Moreover, the category $\Ycal=\Cogen B^+\cap{}^{\perp_0}\Ker\lambda^+=\lambda_\ast({\lmod{B}})$ is bireflective, and  since $B^+$ is a left $B$-module, the map $\lambda\otimes_A B^+$ is bijective, which proves that  $\Coker\lambda\otimes_AB^+=0$ and  $\Hom{A}{B^+}{\Ker\lambda^+}\cong(\Coker\lambda\otimes_AB^+)^+=0$.
Finally, 		since $\Xcal$ is closed under products,
		$\Ccal = \{M \in \AMod \mid \text{$M$ embeds into a module from $\Xcal$}\} = \Cogen(B^+).$ 
		
		(2)		The statement for $\Xcal$ is clear. For the second statement, note that the injectivity of $\eta:M \rightarrow B \otimes_A M$ implies $M$ embeds in the module $B \otimes_A M $ from $ \Xcal$ and therefore lies in $\Ccal$. Conversely, if $M \in \Ccal$, there is a monomorphism $M \rightarrow N$ with $N \in \Xcal$, and the $\Xcal$-reflection $\eta: M \rightarrow B \otimes_A M$ must be injective.
		
	(3)	The cosilting class $\Ccal$ is a torsion-free class. A left $A$-module $M$ belongs to the torsion class ${}^{\perp_0}\Ccal= {}^{\perp_0}\Xcal$ if and only if $\Hom{A}{M}{\lambda_\ast(N)}\cong\Hom{B}{B\otimes_AM}{N}$ vanishes for any $B$-module $N$, which means  that  $B\otimes_AM=0$.

 (4) Consider the canonical exact sequence $0\to M'\to M\to \overline{M}\to 0$ with $M'\in\Ker(B\otimes_A-)$ and $\overline{M}\in\Ccal$ induced by the torsion pair $(\Ker(B\otimes_A-),\Ccal)$. Then $B\otimes_A M\cong B\otimes_A\overline{M}$. Moreover, the $\Xcal$-reflection of the torsion-free module $\overline{M}$ gives rise to a short exact sequence $0\to \overline{M}\to B\otimes_AM\to M''\to 0$ with $B\otimes _A M''= 0$, hence  $M''$ belongs to $\Ker(B\otimes_A-)$. The claim is now obtained by splicing the two short exact sequences.

(5)  Consider the commutative diagram
		$$
		\begin{CD}
				0 @>>> X @>>> C @>>> Z @>>> 0 \\
				& & \eta_X @VVV\eta_C  @VVV\eta_Z  @VVV @| \\
				  & &  B \otimes_A X@>>> B \otimes_A C @>>> B\otimes_A Z @>>> 0
		\end{CD}
		$$
		with exact rows. As $X \in \Xcal$ and $C \in \Ccal$, the  map $\eta_X$ is an isomorphism, and the  map $\eta_C$ is a monomorphism. Then the Four Lemma shows that  $\eta_Z:Z \rightarrow B \otimes_A Z$ is a monomorphism, and thus $Z\in \Ccal$.

(6) By Example~\ref{copreco}(3), the  ring epimorphism $\lambda:A\to B$ is  homological. Hence $\Tor{1}{A}{B}{N}\cong\Tor{1}{B}{B}{N}$ vanishes for all $B$-modules $N$, and since $\Tor{1}{A}{B}{-}$ is left exact, also for all modules in $\Ccal$.
\end{proof}

 \begin{theorem}\label{copreco-mincosilting}
 The map assigning to a ring epimorphism $\lambda:A\to B$  the  class $\Cogen B^+$ yields a bijection between
 \begin{enumerate}
\item[(i)] epiclasses of  ring epimorphisms $\lambda:A\to B$ such that $B^+$ has a precosilting copresentation,
\item[(ii)] equivalence classes of minimal cosilting  modules.
\end{enumerate}

\end{theorem}
\begin{proof}
Proposition~\ref{construct}(1) yields a map (i)$\rightarrow$(ii). To prove the injectivity, suppose  $\lambda_i:A\to B_i$, $i=1,2$, are two ring epimorphisms as in (i) which are mapped to the same cosilting class. Then, as seen in the proof of Proposition~\ref{construct}(1), the map  $\lambda_i^+:B_i^+\to A^+$ is a $\Cogen B_i^+$-cover for $i=1,2$. But $\Cogen B_1^+=\Cogen B_2^+$, hence the cover property yields $B_1^+\cong B_2^+$ and $\Ker \lambda_1^+\cong\Ker\lambda_2^+$, and the bireflective subcategories $\lambda_\ast({\lmod{B_i}})= \Cogen B_i^+\cap{}^{\perp_0}\Ker\lambda_i^+$ coincide for $i=1,2$, showing that $\lambda_1$ and $\lambda_2$ are in the same epiclass.

	Now we show the surjectivity. Take a minimal cosilting module $C$ with an exact sequence \begin{equation}0\to C_1\to C_0\stackrel{g}{\longrightarrow}  A^+\end{equation} such that $g$ is a $\Cogen C$-precover, $C_1$ lies in $\Prod C$,  $\Ycal=\Cogen C\cap{}^{\perp_{0}}C_1$ is a bireflective subcategory of $\lmod{A}$, and $\Hom{A}{C_0}{C_1}=0$. Then  there is a  ring epimorphism $\lambda:A\to B$ such that $\Ycal=\lambda_*(\lmod{B})= \{M\in \AMod\mid \lambda\otimes_AM\text{ is bijective }\}= \{M\in \AMod\mid \Hom{A}{M}{\lambda^+} \text{ is bijective }\}$.
Notice that $\lambda^+:B^+\to A^+$ is then a $\Ycal$-coreflection. On the other hand, the condition $\Hom{A}{C_0}{C_1}=0$ implies that $C_0\in\Ycal$, and therefore also $g:C_0\to A^+$ is a $\Ycal$-coreflection. But then 
  $B^+$  is isomorphic to $C_0$, and in particular
  $B^+$ has a precosilting  copresentation, cf.~\cite[page 9]{abundance}. Now it follows from Proposition~\ref{construct}(1) that $B^+\oplus \Ker\lambda^+$  is a cosilting module which is clearly equivalent to $C$. 
So,  the equivalence class of $C$ lies in the image of our assignment.
\end{proof}

\begin{remark} Example~\ref{Kronecker} shows that the  conditions  (i) and (ii)  in the definition of a minimal cosilting module (Definition~\ref{cominimal}) are independent and depend on the choice of the $\Cogen C$-precover. On the other hand, once we know that $C$ is minimal,  we can always guarantee that the $\Cogen C$-cover of $A^+$ (which is precisely  $\lambda^+:B^+\to A^+$ for the associated ring epimorphism $\lambda:A\to B$) satisfies (i) and (ii).  Hence we can rephrase the definition as follows: $C$ is a minimal cosilting module if the $\Cogen C$-cover of $A^+$ satisfies (i) and (ii).\end{remark}

\begin{corollary}\label{weak} Let  $A$ be a ring of weak global dimension at most one. The map assigning to a ring epimorphism $\lambda:A\to B$  the  class $\Cogen B^+$ yields a bijection between \begin{enumerate}
\item[(i)] epiclasses of homological ring epimorphisms,
\item[(ii)] equivalence classes of minimal cosilting modules.
\end{enumerate}
\end{corollary}

\begin{corollary}\label{flatcomm} Let $A$ be a commutative noetherian ring. The map assigning to a ring epimorphism $\lambda:A\to B$  the  class $\Cogen B^+$ yields a bijection between \begin{enumerate}
\item[(i)] epiclasses of homological (that is, flat) ring epimorphisms,
\item[(ii)] equivalence classes of minimal cosilting modules.
\end{enumerate}
\end{corollary}
\begin{proof}
Recall from  \cite[Proposition 4.5]{AMSTV} that every  ring epimorphism $A\to B$  starting in a commutative noetherian ring $A$ and satisfying $\Tor{1}{A}{B}{B}=0$ is flat, i.e.~$B$ is a flat $A$-module. Combining this with  Example~\ref{copreco} (3) and (4) we infer  that a ring epimorphism $A\to B$ is homological if and only if it is flat, if and only if $B^+$  has a precosilting copresentation. 
\end{proof}
\begin{definition}\label{defmin}\cite{AMV2} A silting module $T$ over a hereditary ring is said to be {\it minimal}
if there is an exact sequence \begin{equation}\label{approxseq3}
A\stackrel{f}{\longrightarrow} T_0\to T_1\to 0
\end{equation}
 such that $f$ is an $\Gen{T}$-envelope and $T_1$ lies in $\Add T$. \end{definition}

 \begin{corollary}\label{hered} Let $A$ be a hereditary ring. Then there are bijections between
 \begin{enumerate}
\item[(i)] epiclasses of homological ring epimorphisms;
\item[(ii)] equivalence classes of minimal silting modules;
\item[(iii)] equivalence classes of minimal cosilting modules.
\end{enumerate}
\end{corollary}
\begin{proof}
It is shown in \cite{AMV2} that the  minimal silting (right) modules over a hereditary ring $A$ correspond bijectively to the epiclasses of homological ring epimorphisms $\lambda:A\to B$. The  bijection associates to $\lambda$  the silting module $B\oplus \Coker\lambda$.
The bijection (i)$\leftrightarrow$(iii) is Corollary~\ref{weak}. Notice that this correspondence is the composition of the bijection (i)$\to$(ii)  with the  map from Corollary~\ref{duality} which associates the equivalence class of a silting (right) module with the equivalence class of its dual cosilting module. 
\end{proof}

\begin{example}\label{exampleminimal} (1)  Every finite dimensional (co)silting module over a finite dimensional hereditary algebra $A$ is minimal. Indeed, $A$ has a $\Gen T$-envelope for every finite dimensional silting module $T$, and every finite dimensional cosilting module is of cofinite type (e.g.~by \cite[Corollary 3.8]{AH} or by Theorem~\ref{hereditaryutc}), hence equivalent to $T^+$ for a finite dimensional silting module $T$.


\smallskip

(2)  Let $A$ be a commutative noetherian ring. It is shown in \cite{AH} that the cosilting classes in $\AMod$ are precisely the torsion-free classes in hereditary torsion pairs, and they are therefore parametrized by subsets  $V\subset \Spec A$ closed under specialization. The subset $V$ is computed as $V=\Supp\Tcal$ where $\Tcal={}^{\perp_0}\Ccal$ is the torsion class associated to the cosilting class $\Ccal$. 

Recall further from Corollary~\ref{flatcomm} that every minimal cosilting class corresponds to  
a flat ring epimorphism $A\to B$, thus the hereditary torsion pair is of the form $(\Tcal=\Ker(B\otimes_A-), \Ccal=\Cogen B^+)$. The subsets $V$ corresponding to flat ring epimorphisms are determined in \cite[Theorem 4.9]{AMSTV}; they are characterized by the property that their complement is coherent in the sense of \cite{K}. We conclude that the minimal cosilting classes in $\AMod$ are parametrized by the specialization-closed subsets of $\Spec A$ having a coherent complement.

\smallskip

(3) If $A$ is a commutative noetherian ring of Krull dimension at most one, the map assigning to a ring epimorphism $\lambda:A\to B$  the  class $\Cogen B^+$ yields a bijection between 
\begin{enumerate}
\item[(i)] epiclasses of homological ring epimorphisms,
\item[(ii)] equivalence classes of  cosilting modules.
\end{enumerate}
The claim is a special case of Corollary~\ref{flatcomm}. It follows from (2) by noting that every subset of $\Spec A$ is coherent when $A$ has Krull dimension at most one, see \cite[Corollary 4.3]{K}.

\smallskip

(4) Let $A$ be a commutative noetherian local domain.  Then  the embedding $\lambda:A\hookrightarrow Q$  into  the quotient field $Q$ is a flat ring epimorphism which corresponds to  the   subset $V=\Spec A\setminus \{0\}$ and to the cotilting torsion pair $(\Tcal=\Ker( Q\otimes_A-), \Ccal=\Cogen Q^+)$ given by the torsion and torsionfree $A$-modules, respectively.  The dual tilting class $\Dcal$ consists of all divisible modules. 
Assume that $Q$ is not countably generated over $A$. 
 Combining a result of Kaplansky~\cite{Ka} with \cite[Theorem 1.1]{AHT1},  we infer that the $A$-module $Q$ has projective dimension at least two, the tilting module generating $\Dcal$ has not the shape $ Q\oplus \Coker\lambda$, and $A$ does not admit a $\Dcal$-envelope, in contrast with the bijection (i)$\leftrightarrow$(ii) in Corollary~\ref{hered} for the hereditary case. 
 Moreover, it follows from Example~\ref{copreco}(2)
 that $Q$ cannot have a presilting presentation.
 
\smallskip

(5) The bijection (i)$\leftrightarrow$(ii) in Corollary~\ref{hered} cannot be extended to rings of weak global dimension at most one. For example, if $A$ is a valuation domain (that is, a commutative local ring of weak global dimension at most one), then the silting modules up to equivalence correspond to flat ring epimorphisms, which coincide with the classical localizations of $A$, this is a consequence of \cite[Theorem 4.7]{AH}. On the other hand, if $A$ is not strongly discrete, meaning that $A$ admits a non-trivial idempotent ideal, then there are homological ring epimorphisms over $A$ which are not flat. They  correspond to minimal cosilting modules which are not of cofinite type, see subsection~\ref{commweak} for a more general statement. For a simple example of a valuation domain which is not strongly discrete, see e.g. \cite[Example 5.24]{BS}.  

\end{example}

\section{TTF triples and ring epimorphisms}

We now want to exploit the construction of cosilting modules arising from ring epimorphisms studied in the previous section. We turn to chains of ring epimorphisms and  use them to construct TTF triples and cosilting objects in the derived category. 

In subsection~\ref{theconstruction}, we provide a construction of a t-structure with definable coaisle from an increasing chain $\ldots\, \lambda_n \leq \lambda_{n+1}\,  \ldots$ of ring epimorphisms $\lambda_n:A\to B_n$ such that all left $A$-modules $B_n^+$ have a precosilting copresentation, or in other words, from a chain of nested cosilting classes $\ldots\, \Cogen B_n^+\subseteq\Cogen B_{n+1}^+\,\ldots$ We show that our construction is a natural extension of the construction of compactly generated t-structures from filtrations by supports discussed in  subsection~\ref{commnoether} for   the commutative noetherian case. Then we determine the conditions ensuring that our t-structure will be induced by a cosilting object. In subsection~\ref{overhered} we specialize to rings of weak global dimension at most one. Here, 
as in the commutative noetherian case,  every t-structure with a definable coaisle encodes a sequence of nested cosilting classes. We can then characterize the t-structures obtained from ring epimorphisms by our construction as those where all  cosilting classes involved are minimal. Recall further that over a hereditary ring our construction yields  compactly generated, cosuspended  TTF triples. In subsection~\ref{SS:chains},   we turn to the dual construction  and determine the  corresponding suspended TTF triples  associated under the bijection $\Psi$ from Theorem~\ref{T:TTFduality}. 

\subsection{Constructing t-structures from chains of epimorphisms}\label{theconstruction}
First of all, we show how to construct a t-structure with definable coaisle.
\begin{proposition}\label{hs}
Let $A$ be a ring, and let 
$$\cdots \subseteq \Vcal_{n-1} \subseteq \Vcal_n \subseteq \Vcal_{n+1} \subseteq \cdots$$
be a (not necessarily strictly) increasing $\Z$-indexed sequence of definable classes closed under extensions in $\AMod$, satisfying the following condition: Whenever $f: V_n \rightarrow V_{n+1}$ is a map with $V_n \in \Vcal_n$ and $V_{n+1} \in \Vcal_{n+1}$ for some $n \in \mathbb{Z}$, then $\Ker(f) \in \Vcal_n$ and $\Coker(f) \in \Vcal_{n+1}$. Then there is a t-structure $(\Ucal,\Vcal)$ in $\DAMod$ with definable coaisle
$$\Vcal = \{X \in \DAMod \mid H^n(X) \in \Vcal_n\text{ for all } n \in \mathbb{Z}\}.$$
\end{proposition}
\begin{proof}
	Invoking \cite[Proposition 3.11]{Kr1}, \cite[Proposition 1.4]{Nee} and \cite[Theorem 4.7]{LV}, it is enough to show that $\Vcal$  is closed under extensions, cosuspensions, products, pure subobjects, and pure quotients. Closure under cosuspensions is clear because $\Vcal_n \subseteq \Vcal_{n+1}$ for all $n \in \Z$. Since the cohomology functor $H^n: \DAMod \rightarrow \AMod$ sends products to products, and 
	pure-exact triangles to pure-exact sequences by \cite[Lemma 2.4]{GP}, the last three closure properties follow from the definability of $\Vcal_n$. 

	Finally, we need to show that $\Vcal$ is closed under extensions. Let $X \rightarrow Y \rightarrow Z \rightarrow X[1]$
be a triangle in $\DAMod$ with $X,Z \in \Vcal$ and consider the induced exact sequence on cohomology
$$H^{n-1}(Z) \xrightarrow{f} H^n(X) \rightarrow H^n(Y) \rightarrow H^n(Z) \xrightarrow{g} H^{n+1}(X).$$
By the hypothesis, the cokernel of the map $f$ belongs to $\Vcal_n$, and the kernel of $g$ belongs to $\Vcal_n$. As $H^n(Y)$ is an extension of these two, and $\Vcal_n$ is closed under extensions, we conclude that $H^n(Y) \in \Vcal_n$ for all $n \in \Z$. Therefore, $Y \in \Vcal$, as desired.
\end{proof}
\begin{remark}\label{BH}
    \begin{itemize}
        \item In \cite{SVR}, coaisles of t-structures in the case of hereditary categories were studied using a notion of \emph{reflective co-narrow sequences} of subcategories. These sequences satisfy essentially the same closure condition to the one considered in Proposition~\ref{hs}. In our situation, the reflectivity of the subcategories in the sequence follows automatically from the assumption of definab.
        \item It is proved in \cite[Proposition 3.7]{BH} that if $A$ is a ring of weak global dimension at most one, then every definable coaisle in $\DAMod$ arises as in Proposition~\ref{hs}.
    \end{itemize}
\end{remark}

Recall from subsection~\ref{epi}  that epimorphisms starting in a ring $A$ form a lattice, where the partial order is induced by inclusion of the corresponding bireflective subcategories. In this lattice,  we now fix  a (not necessarily strictly) increasing chain 
		\begin{equation}\label{chain}\cdots \leq \lambda_{n-1} \leq \lambda_n \leq \lambda_{n+1} \leq \cdots\end{equation}
		of ring epimorphisms $\lambda_n:A\to B_n$ and we assume that the left $A$-modules $B_n^+$ have a precosilting copresentation. Then $\Tor{1}{A}{B_n}{B_n} = 0$, cf.~Example~\ref{copreco}(3). Therefore, the bireflective subcategories  corresponding to the $\lambda_n$ are all extension closed by Theorem~\ref{epicl}.

For every $n\in\mathbb Z$ we also fix the induced ring epimorphism   $\mu_n: B_{n+1} \rightarrow B_n$  given by the diagram
	\begin{equation}\label{mu}\xymatrix{A\ar[rr]^{\lambda_n}\ar[dr]_{\lambda_{n+1}} & & B_n\\ & B_{n+1}\ar[ur]_{\mu_n} & }.\end{equation}

The following observation will be needed later.
\begin{lemma}\label{cones} 
If all $\lambda_n$ in the chain (\ref{chain}) are  homological ring epimorphisms, then also all $\mu_n$ are homological, and in $\DModA$ we have $$\Cone(\mu_n)\cong \Cone(\lambda_n)\otimes_A^\LL  B_{n+1}.$$ Moreover, there is a triangle
	$$\Cone(\lambda_{n+1})\to \Cone(\lambda_n)\to \Cone(\mu_n)\to \Cone(\lambda_{n+1})[1].$$\end{lemma}	
	\begin{proof} That $\mu_n$ is homological follows easily form the fact that $\lambda_{n+1}$ being homological yields a natural isomorphism $B_n \otimes_A^\LL B_n \cong B_n \otimes_{B_{n+1}}^\LL B_n$, see \cite[Theorem 4.4]{GL}. For the second statement, consider the diagram obtained by applying the functor $-\otimes_A^\LL B_{n+1}$ on  (\ref{mu}) and use the natural isomorphisms $B_{n+1} \otimes_A^\LL  B_{n+1}\cong B_{n+1}$ and $B_{n} \otimes_A^\LL  B_{n+1}\cong B_{n}$. The third statement follows from the octahedral axiom. \end{proof}

\begin{proposition}\label{construction} {\rm\bf (The construction)}
		Denote by $\Xcal_n$
		the  extension-closed bireflective subcategories of $\AMod$ corresponding to the chain (\ref{chain}), and set  $$\Ccal_n = \Cogen(\Xcal_n) \text{ and } \Vcal_n = \Ccal_n \cap \Xcal_{n+1}$$ for all $n \in \Z$. Then there is a t-structure $(\Ucal,\Vcal)$ in $\DAMod$ with definable coaisle
		$$\Vcal = \{X \in \DAMod \mid H^n(X) \in \Vcal_n  \text{ for all } n \in \Z\}.$$ 		
\end{proposition}
\begin{proof}
		We  check the conditions of Proposition~\ref{hs} for the sequence $(\Vcal_n \mid n \in \Z)$. Clearly $\Ccal_n \subseteq \Ccal_{n+1}$, and therefore we have $\Vcal_n \subseteq \Vcal_{n+1}$ for all $n \in \Z$. The classes $\Ccal_n$ are minimal cosilting classes by Theorem~\ref{copreco-mincosilting}. In particular, $\Xcal_n$ and $\Ccal_n$ are definable and extension closed, and so is $\Vcal_n$ for any $n \in \Z$. Therefore, we are left with showing that $\Ker(f) \in \Vcal_n$ and $\Coker(f) \in \Vcal_{n+1}$ for any homomorphism $f: V_n \rightarrow V_{n+1}$ with $V_i \in \Vcal_i$ for $i=n,n+1$. Denote $K = \Ker(f)$, $C = \Coker(f)$, $I=\im{f}$, and consider the exact sequences
		$$0 \rightarrow K \rightarrow V_n \rightarrow I \rightarrow 0,$$
		and
		$$0 \rightarrow I \rightarrow V_{n+1} \rightarrow C \rightarrow 0.$$
		As $\Ccal_n$ is closed under submodules, $K \in \Ccal_n$. Moreover, since $V_n$ lies in $\Xcal_{n+1}$ and $V_{n+1}$ embeds in a module from $\Xcal_{n+1}$, the module $I$ is the image of a map in $\Xcal_{n+1}$ and therefore lies  in $\Xcal_{n+1}$. Applying Proposition~\ref{construct}(5) to the second exact sequence, we infer	
		 that $C \in \Ccal_{n+1}$. Then also $C \in \Ccal_{n+2}$, and therefore the natural map $\eta_C:C \rightarrow B_{n+2} \otimes_A C$ is a monomorphism. As $C$ is an epimorphic image of $V_{n+1} \in \Xcal_{n+2}$, the  map $\eta_C$ is also an epimorphism, and thus finally $C \in \Xcal_{n+2}$, establishing $C \in \Vcal_{n+1}$.
\end{proof}
 Now we restrict to the special case of homological ring epimorphisms $\lambda: A \rightarrow B_n$ such that the right $A$-modules $B_n$ have weak dimension at most one.  

\begin{proposition}\label{dimone-coaisle}
	Assume that  all $\lambda_n$ in the chain (\ref{chain}) are homological ring epimorphisms such that the right $A$-modules $B_n$ have weak dimension at most one. Then the definable coaisle $\Vcal$ of Proposition~\ref{construction} can be expressed as follows:
$$\Vcal=\{X \in \DAMod \mid \Cone(\lambda_n) \otimes_A^\LL X \in \DD^{\geq n} \text{ for all } n \in \Z\}.$$
\end{proposition}
\begin{proof} 
	Recall from Example~\ref{copreco}(4) that  the left $A$-module $B_n^+$ has a precosilting copresentation for each $n \in \Z$, and therefore Proposition~\ref{construction} applies. So, setting $\Vcal_n = \Ccal_n \cap \Xcal_{n+1}$, we obtain a definable coaisle $\Vcal = \{X \in \DAMod \mid H^n(X) \in \Vcal_n \text{ for all } n \in \Z\}$  in $\DAMod$. We need to prove that $\Vcal$ equals
	 $\tilde{\Vcal} = \{X \in \DAMod \mid \Cone(\lambda_n) \otimes_A^\LL X \in \DD^{\geq n} \text{ for all } n \in \Z\}$.
	From the long exact sequence of cohomology induced by the natural map $X \rightarrow B_n \otimes_A^\LL X$ we infer that
	$$\Cone(\lambda_n) \otimes_A^\LL X \in \DD^{\geq n} \Leftrightarrow 
	 \text{the map } H^l(X) \rightarrow H^l(B_n \otimes_A^\LL X) $$ $$ \text{ is an isomorphism for all $l<n$ and a monomorphism for $l=n$}.$$
	Since $H^l(B_n \otimes_A^\LL X) \in \Xcal_n$, by Theorem~\ref{KS}(i) we conclude that if $X \in \tilde{\Vcal}$, then $H^n(X) \in \Ccal_n \cap \Xcal_{n+1}$ for each $n\in\mathbb Z$, and thus $X \in \Vcal$. Conversely, if $X \in \Vcal$, consider the soft truncation triangle
	$$\tau^{<n}X \rightarrow X \rightarrow \tau^{\geq n}X \xrightarrow{+}.$$
	Since $X \in \Vcal$, the truncation $\tau^{<n}X $ lies in $\DD_{\Xcal_n}(A)=\DD(B_n\text{-}{\rm Mod})$, and thus $\Cone(\lambda_n) \otimes_A^\LL \tau^{<n}X = 0$, see Theorem~\ref{KS}. Therefore, to show that $X \in \tilde{\Vcal}$, it is enough to show that $\Cone(\lambda_n) \otimes_A^\LL \tau^{\geq n}X \in \DD^{\geq n}$. We truncate further to obtain a triangle
	$$H^n(X)[-n] \rightarrow \tau^{\geq n}X \rightarrow \tau^{>n}X \xrightarrow{+}.$$
	Since the right $A$-module $B_n$ has weak dimension at most one, $\Cone(\lambda_n)$ can be replaced by a complex of right flat $A$-modules concentrated in degrees -1 and 0, and therefore $\Cone(\lambda_n) \otimes_A^\LL \tau^{>n}X \in \DD^{\geq n}$. Also, $\Cone(\lambda_n) \otimes_A^\LL H^n(X)[-n] \cong \Cone(H^n(X) \rightarrow B_n \otimes_A^\LL H^n(X))[-n] \in \DD^{\geq n}$, because $H^n(X) \in \Ccal_n$, and thus $\Tor{A}{1}{B_n}{H^n(X)}=0$ by Proposition~\ref{construct}(6).
\end{proof}

\begin{example}\label{commnoethcoherent}
Let $A$ be a commutative noetherian ring, and let
		$\cdots \leq \lambda_{n-1} \leq \lambda_n \leq \lambda_{n+1} \leq \cdots$
		be an increasing chain of homological ring epimorphisms $\lambda_n:A\to B_n$. Recall from Example~\ref{exampleminimal}(2) that all $B_n$ are flat $A$-modules, and that every $\lambda_n$ corresponds to a hereditary torsion pair $(\Tcal_n=\Ker(B_n\otimes_A-), \Ccal_n=\Cogen B_n^+)$, hence to a minimal cosilting class $\Ccal_n$, and to a   specialization-closed subset $V_n\subset\Spec A$ which has a coherent complement. We obtain a filtration by supports $\Phi:~\mathbb Z\longrightarrow\mathcal{P}(\Spec(A)),\, n\mapsto V_n$ which gives rise to a t-structure $(\Ucal_\Phi,\Vcal_\Phi)$. By \cite[Theorem 3.11]{AJS} $$\Vcal_\Phi=\{X \in \DAMod \mid \mathbf{R}\Gamma_{V_n} X \in \DD^{> n} \text{ for all } n \in \Z\}$$
where $\mathbf{R}\Gamma_{V_n}$ is the right derived functor of the torsion radical $\Gamma_{V_n}$ of the torsion class $\Tcal_n$. In the notation of Theorem~\ref{recoll}, we have  that $\mathbf{R}\Gamma_{V_n}=j_!j^*=j_!(-\otimes_A^\LL\Cone(\lambda_n)[-1]$, see \cite[Remark 3.3]{AMSTV}. Hence we deduce that
	$$\Vcal_\Phi=	\{X \in \DAMod \mid \Cone(\lambda_n) \otimes_A^\LL X \in \DD^{\geq n} \text{ for all } n \in \Z\},$$ that is, the t-structure associated to $\Phi$ coincides with the t-structure constructed in Proposition~\ref{construction}.
	
	In particular, it follows from Example~\ref{exampleminimal}(3) that all compactly generated t-structures over a commutative ring of Krull dimension at most one arise in this way. 
\end{example}

Next, we look for conditions ensuring that the t-structure in Proposition~\ref{construction} is non-degenerate and is thus induced by a cosilting object (cf. \cite[Theorem 4.6]{L}).
	
\begin{proposition}\label{constructioncomplex}
Assume that  all $\lambda_n$ in the chain (\ref{chain}) are homological ring epimorphisms 
	such that the right $A$-modules $B_n$ have weak dimension at most one. 	Denote by $\Xcal_n$  the extension closed bireflective subcategories of $\AMod$, and by
	$\Lcal_n = \Ker(B_n \otimes_A^\LL -)$  the smashing subcategories of $\DAMod$ associated with $\lambda_n$ via Theorems~\ref{epicl} and~\ref{recoll}. 
	Then the t-structure $(\Ucal,\Vcal)$ constructed in Proposition~\ref{construction} is induced by a  cosilting object in $\DAMod$ if and only if the following conditions  hold true.
	\begin{equation}\label{zeros}\bigcap_{n \in \Z}\Xcal_n = 0\text{ and }\bigcap_{n \in \Z}\Lcal_n = 0.\end{equation} In this case, the t-structure $(\Ucal,\Vcal)$ is induced by the (pure-injective) cosilting object
		$$C = \prod_{n \in \Z}\rhom{\Aop}{B_{n+1}}{\Cone(\lambda_n)^+}[-n]\cong\prod_{n \in \Z}\Cone(\mu_n)^+[-n].$$
\end{proposition}
\begin{proof} 
 (1) We first prove that the conditions (\ref{zeros})  are necessary.
		If our t-structure $(\Ucal,\Vcal)$ is  induced by a  cosilting object, then it must be non-degenerate.
 Recall that  $\Xcal_n \subseteq \Ccal_n \cap \Xcal_{n+1} = \Vcal_n$ for all $n \in \Z$. Then we have a chain of inclusions
		$$\cdots \subseteq \Vcal_{n-1} \subseteq \Xcal_n \subseteq \Vcal_n \subseteq \Xcal_{n+1} \subseteq \cdots$$
		It follows that $\bigcap_{n \in \Z}\Vcal[n] = 0$ if and only if $\bigcap_{n \in \Z}\Vcal_n = 0$ if and only if $\bigcap_{n \in \Z}\Xcal_n = 0$.
		
		For the rest, it is enough to show that $\bigcap_{n \in \Z}\Ucal[n] = 0$ implies that $\bigcap_{n \in \Z}\Lcal_n = 0$. Recall again from Theorem~\ref{T:BP} that $\Lcal_n = \Dcal_{\mathcal{E}_n}(A)$, where $\mathcal{E}_n = \Ker(B_n \otimes_A^\LL -) {\,\cap\,} \AMod$. Therefore it is enough to show that $\bigcap_{n \in \Z} \mathcal{E}_n = 0$. We proceed by proving that $\mathcal{E}_n \subseteq \Ucal[n-1]$ for all $n \in \Z$. Let $M \in \mathcal{E}_n$ and consider the approximation triangle with respect to $(\Ucal,\Vcal)$:
$$U \rightarrow M[-n] \rightarrow V \rightarrow U[1].$$
Denote $L_n = \Cone(\lambda_n)$. Since $M[-n] \in \Lcal_n$, we know from Theorem~\ref{recoll} that $M[-n] \cong L_n[-1] \otimes_A^\LL M[-n]$. Applying $L_n[-1] \otimes_A^\LL -$ we thus obtain a triangle
$$L_n[-1] \otimes_A^\LL U \rightarrow M[-n] \rightarrow L_n[-1] \otimes_A^\LL V \rightarrow L_n[-1] \otimes_A^\LL U[1].$$
By Proposition~\ref{dimone-coaisle}
we get that $L_n[-1] \otimes_A^\LL V \in \DD^{>n}$. As $M[-n] \in \DD^{\leq n}$, the map $M[-n] \rightarrow L_n[-1] \otimes_A^\LL V$ from the latter triangle is zero, and thus $M[-n]$ is a direct summand of $L_n[-1] \otimes_A^\LL U$. Since $L_n[-1] \in \DD^{\leq 1}$, $L_n[-1] \otimes_A^\LL U$ belongs to $\Ucal[-1]$ (cf.~\cite[Proposition 2.3]{HCG}). In conclusion, $M$ is a direct summand of $L_n[n-1] \otimes_A^\LL U \in \Ucal[n-1]$.

(2) We now assume that the conditions (\ref{zeros}) hold true, and prove that $C$ is a cosilting object inducing our t-structure.
	It is enough to show that $C \in \Vcal$, that $C$ is a cogenerator of $\DAMod$, and that ${}^{\perp_{>0}}C = \Vcal$ (cf.~\cite[Proposition 4.13]{PV}).

	Set $C_n = \rhom{\Aop}{B_{n+1}}{\Cone(\lambda_n)^+}$ so that $C = \prod_{n \in \Z}C_n[-n]$. By Lemma~\ref{L:dualityformulas} we can rewrite $C_n$ as	
	$$C_n \cong (\Cone(\lambda_n) \otimes_A^\LL B_{n+1})^+ \cong \Cone(\mu_n)^+,$$
	where the second isomorphism follows from Lemma~\ref{cones}. Consider the triangle
	$$B_{n+1}^+[-n-1] \rightarrow C_n[-n] \rightarrow  B_n^+[-n] \rightarrow B_{n+1}^+[-n].$$
	Since $B_n^+ \in \Xcal_n \subseteq \Vcal_n$ and $B_{n+1}^+ \in \Xcal_{n+1} \subseteq \Vcal_{n+1}$, we see that $C_n[-n] \in \Vcal$. Therefore, $C \in \Vcal$.

	Notice that for any object $X \in \DAMod$, we have the following equivalence:
	$$X \in {}^{\perp_\Z}C_n \Leftrightarrow \mu_n \otimes_A^\LL X \text{ is an isomorphism in $\DAMod$}.$$

	For each $n \in \Z$, consider the morphism of triangles
	\begin{equation}\label{E:mortriangle}
	\begin{CD}
	X @>>> B_{n+1} \otimes_A^\LL X @>>> Y_{n+1} @>>> X[1] \\
	@| @VV \mu_n \otimes_A^\LL X V @VV \alpha_n V @| \\
	X @>>> B_n \otimes_A^\LL X @>>> Y_{n} @>>> X[1]
	\end{CD}
	\end{equation}
	induced as in Theorem~\ref{recoll} by the homological ring epimorphisms $\lambda_n$, where $Y_n\in\Lcal_n$ for all $n \in \Z$.
	Assume that $X \in {}^{\perp_\Z}\prod_{n \geq l}C_n$ for some $l \in \Z$. Then for each $n \geq l$ the vertical maps of (\ref{E:mortriangle}) are isomorphisms. As a consequence, $Y_l \cong Y_n \in \Lcal_n$ for each $n \geq l$, and therefore $Y_l \in \bigcap_{n \geq l}\Lcal_n$.  But $\bigcap_{n \geq l}\Lcal_n = 0$ using that $\Lcal_n \supseteq \Lcal_{n+1}$ for all $n \in \Z$. Thus we conclude that $X \cong B_l \otimes_A^\LL X \in \DD(B_l\text{-}{\rm Mod})= \Dcal_{\Xcal_l}(A)$.

	Now assume that $X \in {}^{\perp_\Z}C$. Then, by the previous computation, the cohomologies of $X$ belong to $\bigcap_{l \in \Z}\Xcal_l = 0$, and we conclude that $X = 0$. Therefore, $C$ is a cogenerator in $\DAMod$.

	Finally, let us prove that ${}^{\perp_{>0}}C = \Vcal$. Using Lemma~\ref{L:dualityformulas}, we compute:
	$$X \in {}^{\perp_{>0}}C \Leftrightarrow \rhom{\Aop}{X}{C} \in \DD^{\leq 0} \Leftrightarrow 
	\Cone(\mu_n)\otimes_A^\LL X\in \DD^{\geq n} \text{ for all } n\in\mathbb Z$$
	$$\Leftrightarrow \text{ for all } n \in \Z: H^l(\mu_n \otimes_A^\LL X) \text{ is an isomorphism for all $l<n$}$$ $$\text{and $H^n(\mu_n \otimes_A^\LL X)$ is a monomorphism}.$$
	
	Given $X \in {}^{\perp_{>0}}C$,  consider again the morphism of triangles (\ref{E:mortriangle}) together with the induced map on long exact sequences of cohomology. For any $l \in \Z$, using the Four Lemma twice, we see that  $H^l(\alpha_n)$ is an isomorphism for all $l<n$ and a monomorphism for $l = n$. In particular, it follows that $H^k(Y_l)\cong H^k(Y_n)$ when $k<l$. By Theorem~\ref{T:BP}, the smashing subcategory $\Lcal_n$ is determined on cohomology, and thus the soft truncation $\tau^{<l}(Y_l) \cong \tau^{<l}(Y_n)$ belongs to $\Lcal_n$ for all $l \leq n$. By our hypothesis, this implies $\tau^{<l}(Y_l) = 0$, and so $H^k(X) \rightarrow H^k(B_l \otimes_A^\LL X) \in \Xcal_l$ is an isomorphism for all $k<l$. Furthermore, the map $H^k(\mu_k \otimes_A^\LL X): H^k(X) \cong H^k(B_{k+1} \otimes_A^\LL X) \rightarrow H^k(B_k \otimes_A^\LL X) \in \Xcal_k$ is a monomorphism. Together, this proves that $H^k(X) \in \Ccal_k \cap \Xcal_{k+1}$  for all $k \in \Z$, and therefore $X\in \Vcal$. 

	Conversely, let $X \in \Vcal=\{X \in \DAMod \mid  \Cone(\lambda_n) \otimes_A^\LL X \in \DD^{\geq n} \text{ for all } n \in \Z\}$. Then the triangle from Lemma~\ref{cones}
	shows that $\Cone(\mu_n)$ is an extension of $\Cone(\lambda_n)$ and $\Cone(\lambda_{n+1})[1]$, and therefore $\Cone(\mu_n) \otimes_A^\LL X$ belongs to $\DD^{\geq n}$, showing that $X \in {}^{\perp_{>0}}C$.
\end{proof}

\subsection{Minimal cosuspended TTF triples} \label{overhered} The goal of this subsection is to determine the t-structures in $\DAMod$ that arise from chains of homological ring epimorphisms in the case when the weak global dimension of the ring $A$ is at most one.  

We say that a subcategory $\Ccal$ of $\DAMod$ is \emph{determined on cohomology} if the following equivalence holds for any object $X \in \DAMod$:
	$$X \in \Ccal \Leftrightarrow H^n(X)[-n] \in \Ccal \text{ for all } n \in \Z.$$
	
For example, the definable coaisle obtained in Proposition~\ref{construction} is determined on cohomology, by construction. Moreover, all aisles and coaisles of  t-structures over hereditary rings are determined on cohomology, cf.~subsection~\ref{hereditary}. 
For rings of weak global dimension  at most one, it was proved in
 \cite[Theorem 3.4]{BH} that all definable coaisles are determined on cohomology. 
  We are now going to establish the same result for the corresponding aisles. We will also see that in such case the  t-structure determines a sequence of module-theoretic cosilting classes. Before that, we need the following simple but very useful Lemma.
\begin{lemma}\label{L:kunneth}
	Let $A$ be a $k$-algebra of weak global dimension at most one.
	For any complex of right $A$-modules $X$, any complex of left modules $Y$, and any $n \in \Z$ there is a short exact sequence
	$$0 \rightarrow \bigoplus_{p + q = n}H^p(X) \otimes_A H^q(Y) \rightarrow H^n(X \otimes_A^\LL Y) \rightarrow \bigoplus_{p + q = n+1} \Tor{1}{A}{H^p(X)}{H^q(Y)} \rightarrow 0$$
in $\rmod{\, k}$.
\end{lemma}
\begin{proof}
	This follows directly from an application of the K\"{u}nneth formula, see the proof of \cite[Proposition 3.6]{BS}.
\end{proof}
Given a t-structure  $(\Ucal,\Vcal)$, we denote again $\Ucal_n = \{H^n(X) \mid X \in \Ucal\}$ and $\Vcal_n = \{H^n(X) \mid X \in \Vcal\}.$
\begin{theorem}\label{T:aislecohomology}
Let $A$ be a ring of weak global dimension at most one, and let $(\Ucal,\Vcal)$ be a t-structure in $\DAMod$ such that $\Vcal$ is definable. Then both the aisle $\Ucal$ and the coaisle $\Vcal$ are determined on cohomology. Furthermore, the class $\Ccal_n = \Cogen(\Vcal_n)$ is equal to $\Ucal_n^{\perp_0}$ and it is a cosilting class in $\AMod$ for all $n \in \mathbb{Z}$.
\end{theorem}
\begin{proof}
	The coaisle $\Vcal$ is determined on cohomology by \cite[Theorem 3.4]{BH}.

	By {Proposition~\ref{P:cogenPI}}, the t-structure $(\Ucal,\Vcal)$ is cogenerated by the pure-injective objects of $\Vcal$. Let $X \in \Vcal$ be pure-injective. Since $\Vcal$ is definable, we have $X^{++} \in \Vcal$ 
	(see Remark~\ref{R:dualdef}).  
	Furthermore, $X$ is a direct summand of $X^{++}$ by Corollary~\ref{dualsarepureinjective}.
	In conclusion, the t-structure $(\Ucal,\Vcal)$ is cogenerated by $\Vcal^{++}$. 

	Using the derived Hom-$\otimes$ adjunction, we have for any $Y \in \DAMod$:
$$Y \in \Ucal \Leftrightarrow \rhom{\Aop}{Y}{\Vcal^{++}} \in \DD^{>0} \Leftrightarrow (\Vcal^+ \otimes_A^\LL Y) \in \DD^{<0}.$$
Now it follows from Lemma~\ref{L:kunneth} that the condition $(\Vcal^+ \otimes_A^\LL Y) \in \DD^{<0}$ depends only on the cohomology of $Y$. More precisely, for any $X \in \Vcal^+$ and $n \geq 0$ we have that 
$$H^n(X \otimes_A^\LL Y) = 0 \Leftrightarrow$$
$$  \bigoplus_{p + q = n}H^p(X) \otimes_A H^q(Y) = 0 \text{ and } \bigoplus_{p + q = n+1} \Tor{1}{A}{H^p(X)}{H^q(Y)}=0 \Leftrightarrow$$
$$ H^n(X \otimes_A^\LL (\bigoplus_{k \in \Z}H^k(Y)[-k])) = 0,$$
where both equivalences follow from Lemma~\ref{L:kunneth}, using that the objects $Y$ and $\bigoplus_{k \in \Z}H^k(Y)[-k]$ have indistinguishable cohomology. We conclude that $Y \in \Ucal$ if and only if $\bigoplus_{k \in \Z}H^k(Y)[-k] \in \Ucal$. Since $\Ucal$ is closed under direct summands and coproducts, $\Ucal$ is determined on cohomology. In other words, $\Ucal = \{Y \in \DAMod \mid H^n(Y) \in \Ucal_n \text{ for all } n \in \Z\}$.

We claim that $\Ccal_n = \Ucal_n^{\perp_0}$. First, since $\Ucal_n[-n] \subseteq \Ucal$ and $\Vcal_n[-n] \subseteq \Vcal$, we have that $\Vcal_n \subseteq \Ucal_n^{\perp_0}$, and therefore $\Ccal_n = \Cogen(\Vcal_n) \subseteq \Ucal_n^{\perp_0}$. For the converse implication, let $M \in \Ucal_n^{\perp_0}$ and consider the following approximation triangle with respect to the t-structure $(\Ucal,\Vcal)$:
$$U \rightarrow M[-n] \rightarrow V \rightarrow U[1].$$
Passing to cohomology, we obtain an exact sequence 
$$H^n(U) \rightarrow M \rightarrow H^n(V).$$
Since $H^n(U) \in \Ucal_n$, the map $H^n(U) \rightarrow M$ above is zero, and therefore $M$ embeds into $H^n(V)$. Since $H^n(V) \in \Vcal_n$, we conclude that $M \in \Ccal_n$.

In particular, we proved that the subcategory $\Ccal_n = \Ucal_n^{\perp_0}$ is closed under extensions in $\AMod$. Since $\Ccal_n = \Cogen(\Vcal_n)$ is also a definable subcategory of $\AMod$ by \cite[Proposition 3.4.15]{P}, it is a definable torsion-free class, and thus a cosilting class in $\AMod$.
\end{proof}

Recall from Theorem~\ref{T:cosiltingTTF}(1) that  a t-structure $(\Ucal,\Vcal)$ as in Theorem~\ref{T:aislecohomology} gives rise to a cosuspended TTF triple $(\Ucal,\Vcal,\Wcal)$. Moreover, both the aisle $\Ucal$ and the coaisle $\Vcal$ are determined by the cohomological projections $\Ucal_n, \Vcal_n \subseteq \AMod$, and  the classes $\Vcal_n$ form a chain 
$$\cdots \subseteq \Vcal_{n-1} \subseteq \Vcal_n \subseteq \Vcal_{n+1} \subseteq \cdots$$ 
satisfying the conditions of Proposition~\ref{hs}. Fix the notation $\Ccal_n = \Cogen(\Vcal_n)$ for all $n \in \Z$. 
\begin{lemma}\label{minimalsequence} In the situation of  Theorem~\ref{T:aislecohomology},
		assume that there is $l \in \Z$ such that the cosilting class $\Ccal_n$ is minimal for all $n > l$. For $n > l$ denote by $\lambda_n:A \rightarrow B_n$ the homological ring epimorphism  and by $\Xcal_n$  the extension closed bireflective subcategory of $\AMod$ associated with $\Ccal_n$ via Corollary~\ref{weak} and Theorem~\ref{epicl}.
		Then:
		\begin{enumerate}
				\item[(i)] $\Vcal_n = \Ccal_n \cap \Xcal_{n+1}$ for any $n \geq l$.
				\item[(ii)] $\Xcal_n \subseteq \Xcal_{n+1}$ for any $n > l$.
		\end{enumerate}
\end{lemma}
\begin{proof} 
Fix $n\ge l$, and recall from Proposition~\ref{construct}(2),(3) that $\Ccal_{n+1}=\Cogen \Xcal_{n+1}$ is a torsion-free class 
  consisting of the left $A$-modules $M$ whose $\Xcal_{n+1}$-reflection 
 $M \rightarrow B_{n+1}\otimes_A M$ is  injective.

		$(i)$ Pick $M \in \Vcal_n$. Obviously, $M$ lies in $\Ccal_n = \Cogen(\Vcal_n)\subseteq \Ccal_{n+1}$,
		hence the map $M \rightarrow B_{n+1} \otimes_A M$ is injective.
		  We claim that it is even bijective, which will yield $M \in \Xcal_{n+1}$. 
		  Notice that $B_{n+1} \otimes_A M$ lies in $\Xcal_{n+1}\subseteq \Ccal_{n+1} =
		  \Cogen(\Vcal_{n+1})$. Since 
		   $\Vcal_{n+1}$ is closed under products, there is a monomorphism $\iota: B_{n+1} \otimes_A M \rightarrow N$ for some $N \in \Vcal_{n+1}$, and we have a commutative diagram with exact rows
		$$
			\begin{CD}
					0 @>>> M @>>> B_{n+1} \otimes_A M @>>> M'' @>>> 0 \\
					& & @| @VV\iota V @VVV \\
					0 @>>> M @>>> N @>>> Y @>>> 0
			\end{CD}
		$$
		Because $M \in \Vcal_n$ and $N \in \Vcal_{n+1}$,  the cokernel $Y$ belongs to $\Vcal_{n+1} \subseteq \Ccal_{n+1}$. On the other hand, we know from Proposition~\ref{construct}(4) that  $M''$ belongs to the torsion  class
$\Ker(B_{n+1}\otimes_A -)={}^{\perp_0}\Ccal_{n+1}$.
		Since the rightmost vertical map $M'' \rightarrow Y$ is necessarily injective, the only possibility is that $M'' = 0$.
		We showed that $\Vcal_n \subseteq \Ccal_n \cap \Xcal_{n+1}$.

		For the converse inclusion,  assume $M \in \Ccal_n \cap \Xcal_{n+1}$ and let us show that $M \in \Vcal_n$. Consider the approximation triangle of $M[-n]$ with respect to $(\Ucal,\Vcal)$,
		$$U \rightarrow M[-n] \rightarrow V \rightarrow U[1],$$
		yielding an exact sequence on cohomology:
		$$H^n(U) \rightarrow M \rightarrow H^n(V) \rightarrow H^{n+1}(U) \rightarrow 0.$$
		Since $M \in \Ccal_n$, and $\Ccal_n = \Ucal_n^{\perp_0}$ by Theorem~\ref{T:aislecohomology}, the leftmost map $H^n(U) \rightarrow M$ is zero. Therefore, we actually have a short exact sequence of form
		$$0 \rightarrow M \rightarrow H^n(V) \rightarrow H^{n+1}(U) \rightarrow 0.$$
		We know that $H^n(V) \in \Vcal_n \subseteq \Ccal_n \subseteq \Ccal_{n+1}$, and that $M \in \Xcal_{n+1}$. It follows from  Proposition~\ref{construct}(5)
		that $H^{n+1}(U) \in \Ccal_{n+1}$. But $H^{n+1}(U) \in \Ucal_{n+1}$, and $\Ccal_{n+1} = \Ucal_{n+1}^{\perp_0}$ by Theorem~\ref{T:aislecohomology} again, resulting in $H^{n+1}(U) = 0$. Therefore, $M \cong H^n(V) \in \Vcal_n$, as desired.

		$(ii)$  Assume now $n>l$ and pick a module $M \in \Xcal_n$. From part $(i)$ we know that $\Vcal_n = \Ccal_n \cap \Xcal_{n+1}$, and therefore $\Ccal_n = \Cogen(\Vcal_n) = \Cogen(\Ccal_n \cap \Xcal_{n+1})$. As $M \in \Xcal_n \subseteq \Ccal_n$, there is a monomorphism $\iota: M \rightarrow Y$ to a module $Y \in \Ccal_n \cap \Xcal_{n+1}$. Denote $X = \Coker(\iota)$ and consider the following commutative diagram, where the vertical maps are the natural morphisms:
		$$
		\minCDarrowwidth20pt\begin{CD}
			0 @>>> M @>>> Y @>>> X @>>> 0 \\
			@VVV @VVV @VV\cong V @VVV \\
			\Tor{1}{\A}{B_{n+1}}{X} @>>> B_{n+1} \otimes_A M @>>> B_{n+1} \otimes_A Y @>>> B_{n+1} \otimes_A X @>>> 0 
		\end{CD}
		$$
	Since $M \in \Xcal_n$ and $Y \in \Ccal_n$, we infer from Proposition~\ref{construct}(5),(6) 	 that $X\in\Ccal_n\subseteq\Ccal_{n+1}\subseteq\Ker \Tor{1}{A}{B_{n+1}}{-}$, so  $\Tor{1}{A}{B_{n+1}}{X} = 0$. The fact that $X \in \Ccal_{n+1}$ also implies that the rightmost vertical map of the commutative diagram is injective. Then the Snake Lemma shows that the $\Xcal_{n+1}$-reflection $M \rightarrow B_{n+1} \otimes_A M$ is surjective. Since it is clearly also injective, we can finally conclude that $M \in \Xcal_{n+1}$.
\end{proof}
\begin{definition}\label{defminTTF}
Let $A$ be a ring of weak global dimension at most one.
A cosuspended TTF triple $(\Ucal,\Vcal, \Wcal)$ such that $\Vcal$ is a definable subcategory of $\DAMod$  is \emph{minimal} if $\Ccal_n=\Cogen\Vcal_n$ is a minimal cosilting class for all $n\in\mathbb Z$.
A pure-injective cosilting object $C$ is  \emph{minimal} if so is the corresponding TTF triple $(\Ucal,\Vcal, \Wcal)$ with $\Vcal={}^{\perp_{>0}}C$.

Furthermore, we say that two chains of homological epimorphisms   over a ring $A$ 
 are \emph{equivalent} if they give rise to the same chain of bireflective subcategories in $\AMod$.
\end{definition}
\begin{theorem}\label{minimalTTF} 
 If $A$ is a ring of weak global dimension at most one, there is a bijection between 
\begin{enumerate}
\item[(i)] equivalence classes of chains $\cdots\lambda_n\le \lambda_{n+1}\cdots$ of homological ring epimorphisms;
\item[(ii)] minimal cosuspended TTF triples in $\DAMod$
\end{enumerate}
which restricts to a bijection between 
\begin{enumerate}
	\item[(i')] equivalence classes of chains $\cdots\lambda_n\le \lambda_{n+1}\cdots$ of homological ring epimorphisms satisfying condition (\ref{zeros});
\item[(ii')] equivalence classes of minimal cosilting objects in $\DAMod$.
\end{enumerate}
\end{theorem}
\begin{proof}
	Let us start by showing that the construction of Proposition~\ref{construction} induces a well-defined map (i)$\to$(ii). First, by Theorem~\ref{T:cosiltingTTF}(1) the constructed t-structure $(\Ucal,\Vcal)$ extends to a cosuspended TTF triple. Let $\Xcal_n$ be the bireflective subcategory corresponding to $\lambda_n$ via Theorem~\ref{epicl}. Recall that 
$ \Vcal_n = \Cogen(\Xcal_n) \cap \Xcal_{n+1}$. Since 	
	$(i)$ ensures $\Xcal_n \subseteq \Xcal_{n+1}$ for each $n \in \Z$,  we obtain that $\Cogen (\Vcal_n)  = \Cogen(\Xcal_n)$ is the minimal cosilting class corresponding to $\lambda_n$ via Corollary~\ref{weak}, see also Proposition~\ref{construct}. Therefore, the constructed TTF triple is minimal. Note that the correspondence of Corollary~\ref{weak} also ensures that the assignment (i)$\to$(ii) is injective.
	
	By Theorem~\ref{T:aislecohomology} and Definition~\ref{defminTTF}, any minimal cosuspended TTF triple $(\Ucal,\Vcal,\Wcal)$ determines a sequence of minimal cosilting classes $(\Ccal_n=\Cogen\Vcal_n\mid n\in\mathbb Z)$ and thus by  Corollary~\ref{weak}, also a collection of homological ring epimorphisms $(\lambda_n\mid n\in\mathbb Z)$. By Lemma~\ref{minimalsequence}(ii), the induced bireflective subcategories $(\Xcal_n \mid n \in \Z)$ form an increasing chain, and therefore $\lambda_n$ can be chosen to form a chain as in $(i)$. Finally, Lemma~\ref{minimalsequence}(i) yields $\Vcal_n = \Cogen(\Xcal_n) \cap \Xcal_{n+1}$, showing that $(\Ucal,\Vcal)$ coincides with the t-structure constructed from the chain via Proposition~\ref{construction}. This  proves the surjectivity of the assignment (i)$\to$(ii).

	The second statement follows from Proposition~\ref{constructioncomplex}.
\end{proof}
\subsection{Constructing suspended TTF triples}\label{SS:chains}
We have seen that every chain of homological epimorphisms starting in a  hereditary ring $A$ gives rise to a homotopically smashing, and therefore compactly generated cosuspended TTF triple in $\DAMod$. We now want to determine the  suspended TTF triple in $\DModA$ which is associated under the bijection $\Psi$ from Theorem~\ref{T:TTFduality}.

To this end, we now switch to right modules. Let $A$ be an arbitrary ring. This time we  fix  a (not necessarily strictly) increasing chain 
		\begin{equation}\label{chain2}\cdots \leq \lambda_{n-1} \leq \lambda_n \leq \lambda_{n+1} \leq \cdots\end{equation}
		of ring epimorphisms $\lambda_n:A\to B_n$ such that the right $A$-modules $B_n$ have a presilting presentation. Then the left $A$-module $B_n^+$ has a precosilting copresentation, and $\Tor{1}{A}{B_n}{B_n} = 0$   by Example~\ref{copreco}(2) and (3), so the corresponding bireflective subcategories $\Xcal'_n$ of $\ModA$ are all extension closed.				 
For every $n\in\mathbb Z$ we also fix the induced ring epimorphism $\mu_n$ given by the diagram (\ref{mu}), and the subcategory $\Dcal_n = \Gen(\Xcal_n')$ of $\ModA$, which is the silting class induced by the silting right $A$-module $B_n \oplus\Coker(\lambda_n)$ from Example~\ref{copreco}(2). In particular, $\Dcal_n$ is a definable subcategory of $\ModA$. We start with a dual version of Proposition~\ref{construction}.

\begin{proposition}\label{constructionaisle}{\rm\bf (The dual construction)}
	In the  situation above, denote by $\Xcal'_n$
		the  extension closed bireflective subcategories of $\ModA$ corresponding to the chain (\ref{chain2}), and set  
	 $$\Dcal_n = \Gen(\Xcal'_n),\: \Vcal'_n = \Dcal_n \cap \Xcal'_{n+1}$$ for all $n \in \Z$. Then there is a t-structure $(\Vcal',\Wcal')$ in $\DModA$ with definable aisle
		$$\Vcal' = \{X \in \DModA \mid H^{-n}(X) \in \Vcal_n'  \text{ for all } n \in \Z\}.$$ 
\end{proposition}
\begin{proof}
	One can prove that $\Vcal'$ is a definable subcategory of $\DModA$ closed under suspension and extension by dualizing the proofs of Propositions~\ref{hs} and \ref{construction}. In fact, the definability of $\Vcal'$ follows by  precisely the same argument as in the proof of Proposition~\ref{hs}. Since $\Xcal_n$ and $\Xcal'_n$, as well as $\Ccal_n$ and $\Dcal_n$, are dual definable subcategories of $\AMod$ and $\ModA$, respectively, it follows that $\Vcal$ and $\Vcal'$ are dual definable subcategories of $\DAMod$ and $\DModA$, respectively. The closure of $\Vcal'$ under suspensions and extensions then follows immediately by applying the duality functor.
	Now \cite[Theorem 4.7 and Proposition 5.10]{LV} show that $\Vcal'$ is indeed an aisle of a t-structure.
\end{proof}

In analogy to our approach on the cosilting side, we now restrict  to the case of homological ring epimorphisms $\lambda_n: A \rightarrow B_n$   such that the right $A$-modules $B_n$ have projective dimension at most one. 

\begin{proposition}\label{dimone-aisle}
Assume that  all $\lambda_n$ in the chain (\ref{chain2}) are homological ring epimorphisms 
such that the right $A$-modules $B_n$ have projective dimension at most one. Then the definable aisle $\Vcal'$ of Proposition~\ref{constructionaisle} can be expressed as follows:
$$\Vcal'=\{X \in \DModA \mid \rhom{A}{\Cone(\lambda_n)}{X} \in \DD^{\leq -n} \text{ for all } n \in \Z\}.$$
\end{proposition}
\begin{proof}
 Recall from Example~\ref{copreco}(5) that the right $A$-module $B_n$  has a presilting presentation, therefore Proposition~\ref{constructionaisle} applies. We adopt the same notation and set $\Vcal'_n = \Dcal_n \cap \Xcal'_{n+1}$ for each $n \in \Z$. Let us start with an object $X \in \DModA$ such that $H^{-n}(X) \in \Vcal'_n$ for all $n \in \Z$. We have to show that $\rhom{A}{\Cone(\lambda_n)}{X} \in \DD^{\leq -n} \text{ for all } n \in \Z$.

	By Theorems~\ref{recoll} and~\ref{KS}, there is a sequence of stable TTF triples $$(\Ker(- \otimes_A^\LL B_n),\Bcal_n = \im(\lambda_n)_*, \mathcal{K}_n)$$ in $\DModA$, where $\Bcal_n = \Dcal_{\Xcal_n'}(A)$ and  $L_n = \Cone(\lambda_n)\in \Ker(- \otimes_A^\LL B_n)$ for all $n \in \Z$.
	Since all cohomologies of $(\tau^{>-n}(X))$ lie in  $\Xcal'_n$, we have $\tau^{>-n}(X) \in \Bcal_n$, and therefore $\rhom{A}{L_n}{\tau^{>-n}(X)} = 0$. On the other hand, since $B_n$ is a right $A$-module of projective dimension at most one, $L_n$ can be replaced by a complex of projective $A$-modules concentrated in degrees $-1$ and $0$. Therefore, $\rhom{A}{L_n}{\tau^{<-n}(X)} \in \DD^{\leq -n}$. 
	
	It remains to show that $\rhom{A}{L_n}{H^{-n}(X)[n]} \in \DD^{\leq -n}$. 
We show more: $\rhom{A}{L_n}{M} \in \DD^{\leq 0}$ for any $M \in \Dcal_n$. In fact, since  $\Xcal_n'\subseteq\Bcal_n$, we have  $\rhom{A}{L_n}{X} =0$ for all $X\in\Xcal_n'$, and this implies the statement, because  $M\in \Gen\Xcal_n'$ and $L_n$ can be replaced by a complex of projectives concentrated in degrees -1 and 0. 

 For the converse inclusion, let $X \in \DModA$ be such that $\rhom{A}{L_n}{X} \in \DD^{\leq -n}$ for all $n \in \Z$. By Theorem~\ref{recoll}, there is for each $m \in \Z$ a triangle of form
$$\rhom{A}{B_m}{X} \rightarrow X \rightarrow \rhom{A}{L_m[-1]}{X} \rightarrow \rhom{A}{B_m}{X}[1],$$
inducing for any $n \in \Z$ an exact sequence
$$ H^{-n} \rhom{A}{L_m}{X} \rightarrow H^{-n} \rhom{A}{B_m}{X} \rightarrow$$ $$\rightarrow  H^{-n}(X) \rightarrow H^{-n+1} \rhom{A}{L_m}{X}.$$
Using the assumption on $X$, we see from the last exact sequence that the map $H^{-n} \rhom{A}{B_m}{X} \rightarrow H^{-n}(X)$ is surjective if $m = n$, and it is an isomorphism if $m > n$. Since $H^{-n} \rhom{A}{B_m}{X} \in \Xcal'_m$ by Theorem~\ref{recoll}, we conclude that $H^{-n}(X) \in \Dcal_n \cap \Xcal'_{n+1}$, showing that $X \in \Vcal'$.
\end{proof}

\begin{proposition}\label{P:siltconstr}
	Assume that  all $\lambda_n$ in the chain (\ref{chain2}) are homological ring epimorphisms such that the right $A$-modules $B_n$ have projective dimension at most one.
Denote by $\Xcal_n'$  the extension closed bireflective subcategories of $\ModA$ and by $\Kcal_n=\Ker {\mathbf R} \Hom{A}{B_n}{-}$  the triangulated subcategories of $\DModA$ associated with $\lambda_n$ via Theorems~\ref{epicl} and~\ref{KS}. Then the t-structure $(\Vcal',\Wcal')$ is induced by a silting 
	object if and only if the conditions 
	\begin{equation}\label{dualzeros}\bigcap_{n \in \Z}\Xcal_n' = 0\text{ and }\bigcap_{n \in \Z} \Kcal_n = 0\end{equation} hold true.
	In this case, the t-structure $(\Vcal',\Wcal')$ is induced by the silting 
	object $$T = \bigoplus_{n \in \Z} \Cone(\mu_n)[n].$$
\end{proposition}
\begin{proof} 	 
(1)	We start by proving that the conditions (\ref{dualzeros}) hold whenever the t-structure $(\Vcal',\Wcal')$ is non-degenerate. Recall that $\Xcal_n'\subseteq \Dcal_n \cap \Xcal'_{n+1}=\Vcal'_n$. Since there is clearly a chain
$$\cdots \subseteq \Xcal_n' \subseteq \Vcal'_n \subseteq \Xcal_{n+1}' \subseteq \cdots,$$
we have that $\bigcap_{n \in \Z}\Xcal_n' = 0$ if and only if $\bigcap_{n \in \Z} \Vcal'_n = 0$ if and only if $\bigcap_{n \in \Z}\Vcal'[n] = 0$. It remains to check that $\bigcap_{n \in \Z} \Kcal_n = 0$ provided that $\bigcap_{n \in \Z}\Wcal'[n] = 0$. Since $\Kcal_n$ is determined on cohomology by Theorem~\ref{T:BP}(ii), it is enough to show that for any stalk complex $M$ such that $M \in \Kcal_n$ we have $M \in \Wcal'[-n+1]$. Consider the truncation triangle of $M[n]$ with respect to the t-structure $(\Vcal',\Wcal')$:
$$V' \rightarrow M[n] \rightarrow W' \rightarrow V'[1].$$
Denote $L_n = \Cone(\lambda_n)$. Since $M[n] \in \Kcal_n$, Theorem~\ref{recoll} yields $M[n] \cong \rhom{A}{L_n[-1]}{M[n]}$, and we have a triangle
$$\rhom{A}{L_n[-1]}{V'} \xrightarrow{g} M[n] \rightarrow \rhom{A}{L_n[-1]}{W'} \rightarrow \rhom{A}{L_n[-1]}{V'}[1].$$
By Proposition~\ref{dimone-aisle}, $\rhom{A}{L_n[-1]}{V'} \in \DD^{<-n}$, and therefore the map $g$ in the latter triangle is zero. It follows that $M[n]$ is a direct summand in $\rhom{A}{L_n[-1]}{W'}$. Because $L_n[-1] \in \DD^{\leq 1}$, the complex $\rhom{A}{L_n[-1]}{W'}$ belongs to $\Wcal'[1]$ (cf.~\cite[Proposition 2.3]{HCG}), and so $M \in \Wcal'[-n+1]$.

(2)	Now we assume the conditions (\ref{dualzeros}) and show that $T = \bigoplus_{n \in \Z} \Cone(\mu_n)[n]$ is a silting object inducing the t-structure $(\Vcal',\Wcal')$. This is done similarly as in Proposition~\ref{constructioncomplex}. Indeed, it is enough to show that $T \in \Vcal'$, that $T$ is a generator in $\DModA$, and that $T^{\perp_{>0}} = \Vcal'$. Since $\Cone(\mu_n)$ is an extension of $B_n$ and $B_{n+1}[1]$ for each $n \in \Z$, we infer that $T \in \Vcal'$. Recall from Theorem~\ref{recoll} that for every $X\in\DModA$ there is a triangle
$$Z_n \rightarrow \rhom{A}{B_n}{X} \rightarrow X \rightarrow Z_n[1],$$
where $Z_n = \rhom{A}{L_n}{X} \in \Kcal_n$, and $\rhom{A}{B_n}{X} \in \Bcal_n= \Dcal_{\Xcal'_n}(A)$. For all $n \in \Z$, the map $\rhom{A}{\mu_n}{X}$ induces a morphism of triangles:
\begin{equation}\label{rhomsquare}
\begin{CD}
Z_n @>>> \rhom{A}{B_n}{X} @>>> X @>>> Z_n[1] \\
@V \zeta_n VV @V \rhom{A}{\mu_n}{X} VV @| @V \zeta_n[1] VV \\
Z_{n+1} @>>> \rhom{A}{B_{n+1}}{X} @>>> X @>>> Z_{n+1} \\
\end{CD}
\end{equation}
Assume first that $X \in T^{\perp_\Z}$. Then $\rhom{A}{\Cone(\mu_n)}{X} = 0$, implying that $\rhom{A}{\mu_n}{X}$ is a quasi-isomorphism for all $n \in \Z$, and thus so is $\zeta_n$ for all $n \in \Z$. It follows that $Z_n \cong Z_{n+1}$ in $\DModA$ for all $n \in \Z$, and therefore $Z_n \in \bigcap_{n \in \Z}\Kcal_n = 0$. Then  $X \cong \rhom{A}{B_n}{X} \in \Bcal_n$ for all $n \in \Z$, and thus $X = 0$. We proved that $T$ is a generator of $\DModA$.

It remains to show that $T^{\perp_{>0}} = \Vcal'$. Since $\Vcal' = \{X \in \DModA \mid \rhom{A}{L_n}{X} \in \DD^{\leq -n} \text{ for all } n \in \Z\}$ by Proposition~\ref{dimone-aisle}, the inclusion $\Vcal' \subseteq T^{\perp_{>0}}$ follows by observing that $\Cone(\mu_n)$ is an extension of $L_n$ by $L_{n+1}[1]$, cf.~Lemma~\ref{cones}. To prove the converse, pick $X \in T^{\perp_{>0}}$. From the long exact sequence of cohomologies we infer that $H^{-l}(\rhom{A}{\mu_n}{X})$ is an isomorphism whenever $l<n$, and it is an epimorphism for $l=n$. By the Four Lemma, the map $H^{-l}(\zeta_n)$ is an isomorphism whenever $l<n$, and it is an epimorphism for $l=n$. Then $\tau^{>-l}(Z_l) \cong \tau^{>-l}(Z_n)$ for all $l<n$. By Theorem~\ref{T:BP}, the subcategory $\Kcal_n$ is determined on cohomology for any $n \in \Z$. Hence $\tau^{>-l}(Z_l) \in \bigcap_{n \in \Z}\Kcal_n = 0$   for each $l\in\Z$. This establishes the isomorphisms $H^{-n}(\rhom{A}{B_{n+1}}{X}) \cong H^{-n}(X)$, showing that $H^{-n}(X) \in \Xcal'_{n+1}$. Furthermore, the map $H^{-n}(\rhom{A}{\mu_{n}}{X}): H^{-n}(\rhom{A}{B_{n}}{X}) \rightarrow H^{-n}(\rhom{A}{B_{n+1}}{X}) \cong H^{-n}(X)$ is an epimorphism, and therefore $H^{-n}(X) \in \Gen(\Xcal_n')  = \Dcal_n$. We proved that $X \in \Vcal'$.
\end{proof}
\begin{proposition}\label{L:YcgTTF} 
	Let $A$ be a right hereditary ring, and let
		$\cdots \leq \lambda_{n-1} \leq \lambda_n \leq \lambda_{n+1} \leq \cdots$
		be an increasing chain of homological ring epimorphisms $\lambda_n:A\to B_n$.  Then the  t-structure $(\Vcal',\Wcal')$  with definable aisle $\Vcal'$ constructed in Proposition~\ref{constructionaisle} can be completed to a compactly generated suspended TTF triple $(\Ucal',\Vcal',\Wcal')$ in $\DModA$, which corresponds  under the map $\Psi$ of Theorem~\ref{T:TTFduality} to the compactly generated TTF triple $(\Ucal,\Vcal,\Wcal)$  in $\DAMod$ given by  Proposition~\ref{construction}.
\end{proposition}
\begin{proof}  
	To prove the existence of a compactly generated suspended TTF triple of the stated shape, it is enough to exhibit  a set $\Scal$ of compact objects of $\DModA$ such that $\Vcal' = \Scal^{\perp_0}$. Let $(\Mcal_n\mid n \in\mathbb Z)$ be the sequence of wide subcategories of $\modA$ corresponding to the homological epimorphisms $\lambda_n: A \rightarrow B_n$ via Theorem~\ref{univlochom}. It is proved in \cite[Corollary 5.15 and Proposition 5.2]{AMV2} that $\Dcal_n = \Ker \Ext{1}{A}{\Mcal_n}{-} \cap \rmod{A/I_n}$ where $I_n = \Ker(\lambda_n)$ is an idempotent two-sided ideal of $A$. Since $A$ is right hereditary, $I_n$ is a projective right $A$-module and thus, due to a theorem of Kaplansky \cite[(2.24)]{Lam}, $I_n$ admits a direct sum decomposition $I_n = \bigoplus_{\alpha \in \varkappa_n} I_n^\alpha$, where $\varkappa_n$ is a cardinal and $I_n^\alpha$ is a finitely generated projective right $A$-module for each $\alpha \in \varkappa_n$, $n \in \Z$. We consider the subset of $\DD^c(\ModA)$
$$\Scal = \bigcup_{n \in \Z} (\Mcal_{n+1} \cup \{I_{n}^\alpha \mid \alpha \in \varkappa_{n}\})[n],$$
and claim that an object $X \in \DModA$ satisfies $\Hom{\DModA}{\Scal}{X} = 0$ if and only if $X \in \Vcal'$. Since $A$ is right hereditary, and both $\Vcal'$ and $\Ker \Hom{\DModA}{\Scal}{-}$ are closed under direct sums and summands, it is enough to check this when $X = M[n]$ is a stalk complex given by some $M \in \ModA$. Since $I_n$ is an idempotent ideal, clearly $\rmod{A/I_n} = I_n^{\perp_0} \subseteq \ModA$. Then we compute, using again the right heredity of $A$, and the fact that $I_n^\alpha$ is projective for each $\alpha \in \varkappa_n$:
$$M[n] \in \Scal^{\perp_0} \Leftrightarrow M \in \Mcal_n^{\perp_1} \cap \Mcal_{n+1}^{\perp_0} \cap I_n^{\perp_0}.$$
Since $\Mcal_{n+1} \subseteq \Mcal_n$, this is further equivalent to $M$ satisfying $M \in \Mcal_{n+1}^{\perp_{0,1}}$ and $M \in \Mcal_n^{\perp_1} \cap I_n^{\perp_0}$. But $\Mcal_{n+1}^{\perp_{0,1}} = \Xcal_{n+1}'$ and $\Mcal_n^{\perp_1} \cap I_n^{\perp_0} = \Dcal_n$. Therefore, 
$$M[n] \in \Scal^{\perp_0} \Leftrightarrow M \in \Dcal_n \cap \Xcal_{n+1}' \Leftrightarrow M[n] \in \Vcal',$$
as desired.

The map $\Psi$ from Theorem~\ref{T:TTFduality} now provides  a compactly generated TTF triple $(\Ucal,\Vcal,\Wcal)$  in $\DAMod$, where $\Vcal'$ and $\Vcal$ are dual definable subcategories. In other words, $\Vcal$ is uniquely determined by the property that a complex $X \in \DAMod$ lies in $\Vcal$ if and only if $X^+$ lies in $\Vcal'$, cf.~Lemma~\ref{L:dualdefinable2} and Remark~\ref{R:dualdef}. But then, by construction, $\Vcal$ has to be the definable coaisle obtained   from the chain  		$\cdots \leq \lambda_n \leq \lambda_{n+1} \leq \cdots$ as in Proposition~\ref{dimone-coaisle}.
	\end{proof}

\section{Classification results}\label{S:examples}
We finish this note by discussing our results  for specific classes of rings, elaborating on the interplay between TTF triples, (co)silting objects, and ring epimorphisms in these special cases.

\subsection{Commutative noetherian rings of Krull dimension at most one}\label{dimone} Let $A$ be a commutative noetherian ring. Recall from Corollary~\ref{flatcomm} that the flatness of a ring epimorphism is determined by the closure properties of the corresponding bireflective category, and consequently the collection of epiclasses of flat ring epimorphisms forms a complete lattice as in the discussion after Theorem~\ref{epicl}. 
\begin{theorem}\label{T:dimone}
	Let $A$ be a commutative noetherian ring of Krull dimension at most one. Then there are bijections between the following collections:
\begin{enumerate}
	\item[(i)] equivalence classes of chains $\cdots \leq \lambda_{n-1} \leq \lambda_n \leq \lambda_{n+1} \leq \cdots$ in the lattice of flat ring epimorphisms over $A$ such that the meet $\bigwedge_{n \in \Z} \lambda_n$ and the join $\bigvee_{n \in \Z}\lambda_n$ equal the trivial homomorphism $\;{\rm 0}_A: A \rightarrow 0$ and the identity $\;{\rm id}_A: A \rightarrow A$, respectively;
	\item[(ii)] equivalence classes of silting objects $T$ of finite type in $\DModA$;	\item[(iii)] equivalence classes of pure-injective cosilting objects $C$ in $\DModA$.\end{enumerate}

	The representatives of the equivalence classes of objects in (ii) and (iii) are obtained from (i) by the constructions of Proposition~\ref{P:siltconstr} and Proposition~\ref{constructioncomplex}, respectively.
\end{theorem}
\begin{proof}
	By Example~\ref{commnoethcoherent}, the Krull dimension of $A$ being at most one implies that the assignment
 $$(\lambda: A \rightarrow B) \mapsto V = \{\p \in \Spec(A) \mid A/\p \otimes_A B = 0\}$$
yields a bijection between the epiclasses of flat ring epimorphisms over $A$ and the specialization closed subsets of $\Spec(A)$. Recall that under this identification, $V = \Supp \Tcal$ where $\Tcal = \Ker(B \otimes_A -)$ is the hereditary torsion class induced by the flat ring epimorphism $\lambda$. It follows that there is an anti-isomorphism of complete lattices between the lattice of epiclasses of flat ring epimorphisms and the set-theoretic lattice of specialization closed subsets of $\Spec(A)$. Therefore, we naturally obtain a bijection between the equivalence classes of chains of flat ring epimorphisms satisfying the conditions of (i), and the filtrations by supports on $\Spec(A)$ satisfying the conditions in Theorem~\ref{commnoeth}(iii). In this way, we have established the bijections (i) $\leftrightarrow (ii) \leftrightarrow (iii)$.
	
	Since the ring epimorphism $\lambda_n: A \rightarrow B_n$ is flat for all $n \in \Z$, Proposition~\ref{constructioncomplex} applies and yields the cosilting object $C$. Finally, since the Krull dimension of $A$ is at most one, for any $n \in \Z$ the projective dimension of the flat module $B_n$ over $A$ is at most one by \cite[Corollaire 3.2.7]{RG}, and therefore the assumptions of Proposition~\ref{P:siltconstr} are satisfied as well, yielding the silting object $T$.
\end{proof}

\begin{example}\label{EX:cosiltgroups} 
	We compute explicitly the silting and cosilting objects of Theorem~\ref{T:dimone} in the case of the ring of integers  $\Z$. Let $\cdots \leq \lambda_{n-1} \leq \lambda_n \leq \lambda_{n+1} \leq \cdots$ be a chain of flat ring epimorphisms satisfying the conditions (i) of Theorem~\ref{T:dimone}. Inspecting the shape of the lattice, this amounts to a choice of an integer $l \in \Z$ which is the smallest with respect to property $B_l \neq 0$, and then a choice of a decreasing sequence
	$$\mathbb{P} \supseteq P_0 \supseteq P_1 \supseteq P_2 \supseteq \cdots$$
	of subsets of the set $\mathbb{P}$ of prime numbers such that $\bigcap_{n \geq 0} P_n = \emptyset$, determined by the property that $B_{l+n}$ is isomorphic to $\Z[P_n^{-1}]$, the localization of $\Z$ at all the prime numbers from $P_n$. For each $n \geq l$, the connecting ring epimorphism $\mu_n: B_{n+1} \rightarrow B_n$ is injective, and
$$\Cone(\mu_n) = \Coker(\mu_n) \cong \bigoplus_{p \in P_{n} \setminus P_{n+1}} \Z_p^{\infty},$$
where $\Z_p^{\infty}$ is the Pr\"{u}fer $p$-group. The remaining non-trivial morphism $\mu_{l-1}$ is of form $\mu_{l-1}: B_{l} \rightarrow 0$, and thus
$$\Cone(\mu_{l-1}) = B_{l}[1] = \Z[P_l^{-1}][1].$$
	Applying the constructions of Proposition~\ref{P:siltconstr} and Proposition~\ref{constructioncomplex}, we obtain that the desired silting object is the split complex
$$T = (\bigoplus_{n \geq 0} \bigoplus_{p \in P_{n} \setminus P_{n+1}} \Z_p^{\infty}[l+n])  \oplus \Z[P_l^{-1}][l]$$
and the cosilting object is the split complex
$$C = (\prod_{n \geq 0} \prod_{p \in P_{n} \setminus P_{n+1}} \mathbb{J}_p[-l-n])  \oplus \Z[P_l^{-1}]^+[-l],$$
where $\mathbb{J}_p = \Hom{\Z}{\Z_p^{\infty}}{\mathbb{Q}/\Z}$ is the group of $p$-adic integers. Finally, note that $T$ is a bounded silting complex if and only if $C$ is a bounded cosilting complex if and only if there is $n \geq 0$ such that $P_n = \emptyset$.
\end{example}

In subsection~\ref{Krone}, we will show that a construction of silting and cosilting objects similar to Example~\ref{EX:cosiltgroups} is available also for the Kronecker algebra over a field. 

\subsection{Commutative rings of weak global dimension at most one}\label{commweak}  In the commutative case, we can refine Theorem~\ref{minimalTTF} and determine which TTF triples  are compactly generated. An essential ingredient is the classification of definable coaisles over valuation domains provided in \cite{BH}. Recall that a valuation domain is an integral domain such that its ideals are totally ordered by inclusion. Also, a commutative ring $A$ is of weak global dimension at most one if and only if its localization $A_{\p}$ at each prime $\p \in \Spec(A)$ is a valuation domain (\cite[Corollary 4.2.6]{Gl}).

\begin{theorem}\label{wgdcomm}
	Let $A$ be a commutative ring of weak global dimension at most one. Then the bijection of Theorem~\ref{minimalTTF} restricts to a bijection between
\begin{enumerate}
\item[(i)] equivalence classes of chains $\cdots\lambda_n\le \lambda_{n+1}\cdots$ of flat ring epimorphisms;
\item[(ii)] minimal cosuspended TTF triples in $\DAMod$ which are compactly generated.
\end{enumerate}
\end{theorem}
\begin{proof}
	Let $(\Ucal,\Vcal,\Wcal)$ be a minimal cosuspended TTF triple in $\DModA$ corresponding to a chain $\cdots\lambda_n\le \lambda_{n+1}\cdots$ via Theorem~\ref{minimalTTF}. For any prime ideal $\p \in \Spec(A)$ and any subcategory $\Ccal$ of $\DModA$ denote $\Ccal_{\p} = \{X \otimes_A A_{\p} \mid X \in \Ccal\}$, viewed as a subcategory of $\DD(\rmod{A_{\p}})$. By \cite[Lemma 8.6]{BH}, the pair $(\Ucal_{\p},\Vcal_{\p})$ forms a t-structure in $\DD(\rmod{A_{\p}})$ and $\Vcal_{\p} \subseteq \Vcal$ for all $\p \in \Spec(A)$. Using the latter inclusion, it follows easily that $\Vcal_{\p} = \Vcal \cap \DD(\rmod{A_{\p}})$, and also that the subcategory $\Vcal_{\p}$ consists precisely of those complexes $X$ such that $H^n(X)$ belongs to $\Vcal_n \cap \rmod{A_{\p}}$ for any $n \in \Z$, where $\Vcal_n$ is the subcategory of $\ModA$ defined in Proposition~\ref{construction}. Then it is straightforward to check that the t-structure $(\Ucal_{\p},\Vcal_{\p})$ is obtained from the chain $\cdots(\lambda_n \otimes_A A_{\p})\le (\lambda_{n+1} \otimes_A A_{\p})\cdots$ of homological epimorphisms over the ring $A_{\p}$ via Proposition~\ref{construction}. Therefore, there is a minimal cosuspended TTF triple $(\Ucal_{\p},\Vcal_{\p},\Wcal_{\p})$ in $\DD(\rmod{A_{\p}})$ again by Theorem~\ref{minimalTTF}.

	It follows from the proof of \cite[Theorem 8.7]{BH} that the t-structure $(\Ucal,\Vcal)$ is compactly generated in $\DModA$ if and only if $(\Ucal_{\p},\Vcal_{\p})$ is compactly generated in $\DD(\rmod{A_{\p}})$ for all $\p \in \Spec(R)$. By \cite[Theorem 9.4]{BH}, the t-structure $(\Ucal_{\p},\Vcal_{\p})$ is compactly generated in $\DD(\rmod{A_{\p}})$ if and only if the chain $\cdots(\lambda_n \otimes_A A_{\p})\le (\lambda_{n+1} \otimes_A A_{\p})\cdots$ consists of flat epimorphisms over the valuation domain $A_{\p}$. In conclusion, the t-structure $(\Ucal,\Vcal)$ is compactly generated if and only if $\lambda_n \otimes_A A_{\p}$ is a flat ring epimorphism for each $n \in \Z$ and $\p \in \Spec(A)$, which in turn is equivalent to $\lambda_n$ being a flat ring epimorphism over $A$ for each $n \in \Z$. This establishes the bijection.
\end{proof}
The situation of Theorem~\ref{wgdcomm} becomes even nicer if the ring is in addition coherent. Coherent rings of weak global dimension at most one are precisely the semihereditary rings, that is, rings such that any finitely generated ideal is projective. Semihereditary commutative rings include valuation domains, and more generally, Pr\"{u}fer domains. 

Recall that a hereditary torsion pair $(\Tcal,\Ccal)$ in $\AMod$ is \emph{of finite type} if $\Ccal$ is closed under direct limits. 

\begin{proposition}\label{weaksemiher} 
If $A$ is a left semihereditary ring, the bijection in Corollary~\ref{weak} assigning to a ring epimorphism $\lambda:A\to B$  the  minimal cosilting class $\Cogen B^+$ 
takes universal localizations of $A$ to minimal cosilting classes of cofinite type, and it 
induces  a bijection between 
\begin{enumerate}
 \item[(i)] epiclasses of flat ring epimorphisms,
\item[(ii)] hereditary torsion pairs of finite type.
\end{enumerate} 
In particular, every hereditary torsion pair of finite type is induced by a minimal cosilting module of cofinite type.
\end{proposition}
\begin{proof}
Let $\lambda:A\to B$ be a universal localization. By Theorem~\ref{univlochom}, the  finitely presented left $A$-modules whose projective resolutions are inverted by $B$ form  a wide subcategory $\Mcal$ of $\lfmod{A}$ such that $\lambda$ is equivalent to  the universal localization $\lambda_\Mcal$ at the projective resolutions of the modules in $\Mcal$.  Consider the torsion pair $(\Tcal_B=\Ker(B\otimes_A -), \,\Ccal_B= \Cogen B^+)$ in $\AMod$. The torsion
class $\Tcal_B$ obviously contains $\mathcal M$ and thus also ${}^{\perp_0}(\Mcal^{\perp_0})$,
which in turn contains $\Gen{ \mathcal{\Mcal}}$, and
 from \cite[Lemma 5.1 and Theorem 5.5]{Scho1} it even follows that 
 $\Tcal_B={}^{\perp_0}(\Mcal^{\perp_0})=\Gen{ \mathcal{\Mcal}}$. We conclude that 
 $(\Tcal_B,\Ccal_B)$ is the torsion pair generated by $\Mcal$, which shows that $\Ccal_B$ is a minimal cosilting class of cofinite type.

Assume now that $\lambda:A\to B$ is a flat ring epimorphism. Then   $(\Tcal_B, \Ccal_B)$ is obviously a hereditary torsion pair of finite type. Conversely,  every hereditary torsion pair of finite type $(\Tcal, \Ccal)$ is associated to a Gabriel topology with a basis of finitely generated ideals (cf.~\cite[Chapter VI, Theorem 5.1, and Chapter XIII, Proposition 1.2]{Ste}), which are also projective by assumption on  $A$. It then follows from \cite[Chapter XI, Propositions 3.3 and 3.4]{Ste} that there is a flat ring epimorphism $\lambda:A\to B$ such that  $\Ccal$ consists of the left $A$-modules $M$ whose reflection 
 $\eta: M \rightarrow B\otimes_A M$ is  injective. By Proposition~\ref{construct}(2), this means that $\Ccal=\Ccal_B$. Finally, $\lambda$ is equivalent to a universal localization  by \cite[Proposition 5.3]{angarc}, hence $\Ccal$ is a minimal cosilting class of cofinite type. 
\end{proof}
\begin{proposition}\label{semihereditary}
If $A$ is a commutative semihereditary ring then any compactly generated cosuspended TTF triple in $\DModA$ is minimal.
\end{proposition}
\begin{proof}
	Let $(\Ucal,\Vcal,\Wcal)$ be a compactly generated cosuspended TTF triple in $\DModA$. Recall from Theorem~\ref{T:aislecohomology} that $\Vcal=\{X \in \DModA \mid H^n(X) \in \Vcal_n \text{ for all } n \in \Z\}$ where $\Vcal_n = \{H^n(X) \mid X \in \Vcal\}$.  Setting again $\Ccal_n = \Cogen(\Vcal_n)$, we obtain an ascending sequence of cosilting classes $\ldots \Ccal_n\subseteq \Ccal_{n+1}\ldots$. By \cite[Proposition 3.10 and its proof]{BH}, for all $n\in\Z$, the subcategory $\Vcal_n$ is closed under taking injective envelopes in $\ModA$, and there is a hereditary torsion pair $(\Tcal_n,\Ccal_n)$ of finite type. 
	By Proposition~\ref{weaksemiher}, this means that all $\Ccal_n$ are minimal cosilting classes. Hence the TTF triple is minimal according to Definition~\ref{defminTTF}.
	\end{proof}
If we confine Proposition~\ref{semihereditary} to the module-theoretic case, we see that over a commutative semihereditary ring the correspondence of Corollary~\ref{weak}  restricts to a bijection between  (equivalence classes of)  flat ring epimorphisms and cosilting modules of cofinite type. Notice moreover that the flat ring epimorphisms coincide with the universal localisations in this case (\cite[Proposition 5.3]{angarc} and  \cite[Theorem 7.8]{BS}).

\begin{remark}
	Let us give further comments on the minimality condition in the local case, that is, over a valuation domain $A$. In this situation, whether a cosuspended TTF triple with a definable middle term is minimal depends on certain invariants of a topological nature. In what follows, we adhere closely to \cite{BH}. For any valuation domain $A$, \cite[Theorem 8.3]{BH} establishes a bijection between the definable coaisles in $\DAMod$ (and thus, also between the cosuspended TTF triples with a definable middle term) and certain sequences of sets of intervals in $\Spec(A)$ called the \emph{admissible sequences}. Furthermore, \cite[Theorem 9.4]{BH} shows that, under this bijection, the minimal cosuspended TTF triples correspond to admissible sequences which are \emph{non-dense}, an additional topological condition. For example, if $\Spec(A)$ is countable then all admissible sequences are non-dense, \cite[Remark 9.5]{BH}, and therefore all cosuspended TTF triples in $\DAMod$ with a definable middle term are minimal. On the other hand, there are valuation domains $A$ with a rich supply of admissible sequences with density, see \cite[Example 5.1]{B} or \cite[Example 7.6]{BH} for examples of cosilting modules and TTF triples which are not minimal.

Finally, let us sketch how the topological information given by an admissible sequence fits together with the chain of ring epimorphisms under the minimality condition. Let $\cdots\lambda_n\le \lambda_{n+1}\cdots$ be a chain of homological ring epimorphisms $\lambda_n: A \rightarrow B_n$ over a valuation domain $A$ corresponding to a minimal cosuspended TTF triple in $\DAMod$. Then the $n$-th term of the corresponding non-dense admissible sequence is a collection of \emph{admissible intervals} in the terminology of \cite{BS}, such that it recovers the Zariski spectrum $\Spec(B_n)$ by taking the disjoint union of the intervals it contains via \cite[Lemma 5.8]{BS} or under the correspondence \cite[Theorem 5.23]{BS}, see \cite[\S 9]{BH} for details.
\end{remark}

 \subsection{Finite dimensional hereditary algebras} Our next result provides a classification of the compact silting objects over a finite dimensional hereditary algebra $A$, establishing  a  bijection with finite chains of finite dimensional homological ring epimorphisms. In particular, when $A$ has finite representation type, this yields a classification of all bounded silting complexes.
 
\begin{theorem}\label{fdher} Let  $A$ be a finite dimensional hereditary algebra over a field $k$. Every compact silting complex $T$ over $A$ arises as in Proposition~\ref{P:siltconstr} from a finite chain of finite dimensional homological ring epimorphisms $0_A\le \lambda_n\le\ldots\le\lambda_m\le {\rm id}_A$. 
\end{theorem}
\begin{proof}
	To prove this,  let $\Vcal'=T^{\perp_{>0}}$ and  $(\Ucal',\Vcal',\Wcal')$ be the compactly generated suspended TTF triple in $\DModA$ induced by $T$. We fix  an integer $n$ and denote $$\Vcal'_n=\{H^{-n}(X)\mid X\in\Vcal'\}, \: \Wcal'_n=\{H^{-n}(X)\mid X\in\Wcal'\},  \text{ and } \Dcal_n=\Gen\Vcal'_n.$$

Since $A$ is a finite dimensional algebra, the co-t-structure $(\Ucal',\Vcal')$ restricts to a bounded co-t-structure in $\K^b(\projA)$, see \cite[Remark 4.7(2)]{AMV1}. So, there is a triangle $$U\to A[n]\stackrel{f}{\to} V\to U[1]$$
	where $U$ and $V$ are compact objects in $\Ucal'$ and $\Vcal'$, respectively. Taking cohomology, we obtain a morphism $$f_n:A\to V_n$$ in $\modA$ which is clearly a $\Vcal'_n$-preenvelope of $A$. By choosing a left minimal version of $f_n$ (see \cite[Theorem 2.4]{ARS}) we can even assume w.l.o.g.~that $f_n$ is a  $\Vcal'_n$-envelope. Moreover, it is easy to see that $f_n$ is also a $\Dcal_n$-envelope of $A$, and that this implies $\Dcal_n=\Gen V_n$.

Dualizing the arguments in the proof of Theorem~\ref{T:aislecohomology}, we see that $\Dcal_n={}^{\perp_0}\Wcal'_n$ is a torsion class. Since $V_n\in\modA$, it follows that $\Dcal_n\cap\modA=\gen\V_n$ is a functorially finite torsion class in $\modA$, which is therefore generated by a    finite dimensional silting (that is, support $\tau$-tilting)  module $T_n$, see \cite[Proposition 1.1 and Theorem 2.7]{AIR}. 
		In particular, it follows from \cite[Lemma 4.6]{AMV4} that $\Dcal_n=\Gen T_n$ is a minimal silting class. 
		
		By \cite[Corollary 5.12]{AMV2} we can choose $T_n$ of the form $T_n=B_n\oplus\Coker\lambda_n$ where  $\lambda_n:A\to B_n$ is a homological ring epimorphism to a finite dimensional algebra $B_n$. Observe that $\Gen B_n=\Gen T_n=\Dcal_n$, and $\lambda_n$ is also a $\Dcal_n$-envelope of $A$. But then, since envelopes are unique up to isomorphism, we can also set $T_n=V_n\oplus\Coker f_n$.
	
	Now we apply the map $\Psi$ of Theorem~\ref{T:TTFduality}. The TTF triple $(\Ucal',\Vcal',\Wcal')$
	 corresponds   
 to a compactly generated TTF triple 
  $(\Ucal,\Vcal,\Wcal)$  in $\DAMod$ which is induced by  the dual cosilting object $C=T^+$.    Denote again $$\Vcal_n=\{H^{n}(X)\mid X\in\Vcal\} \text{ and } \Ccal_n=\Cogen \Vcal_n.$$
  We know from Lemma~\ref{L:dualityformulas}(i) that $\Vcal_n$ and $\Vcal'_n$ are dual definable subcategories.
As in Proposition~\ref{precosilting} we see that $f_n^+: V_n^+\to A^+$ is a $\Vcal_n$-precover of $A^+$, hence also a $\Cogen \Vcal_n$-precover, and we deduce that $\Ccal_n=\Cogen V_n^+=\Cogen T_n^+$.  Using Corollary~\ref{hered}, we conclude that $\Ccal_n$ is a minimal cosilting class for all $n\in\mathbb Z$. Hence  $(\Ucal,\Vcal,\Wcal)$ is a minimal cosuspended TTF triple, and it follows from Theorem~\ref{minimalTTF}  that the  minimal cosilting object $C$ arises from a chain of homological ring epimorphisms with (\ref{zeros}) by the construction in Proposition~\ref{constructioncomplex}. Since $C$ is a compact complex, this chain has to be finite. Moreover, the conditions (\ref{zeros}) imply that the first term has to be the trivial epimorphism $A\to 0$, and  the last term has to be ${\rm id}_A$.
Finally, the silting complex constructed as in Proposition~\ref{P:siltconstr}
from the same chain must be equivalent to $T$ by  Proposition~\ref{L:YcgTTF}. \end{proof}

\subsection{The Kronecker algebra}\label{Krone}
Throughout this subsection, unless stated otherwise, $A$ denotes the Kronecker algebra, i.e., the path algebra of the quiver $\xy\xymatrixcolsep{2pc}\xymatrix{ \bullet \ar@<0.5ex>[r]  \ar@<-0.5ex>[r] & \bullet } \endxy$ over a field $k$. This algebra has infinite representation type, but we can still classify the silting objects of finite type and the pure-injective cosilting objects, as we are going to see below.

 We adopt terminology and notation from \cite{RR}. In particular,
we denote by $\p,\tube,\q$ the classes of  indecomposable preprojective, regular, and preinjective  $A$-modules (right or left, depending on the context). We fix  a complete
irredundant set of quasi-simple (i.e.~simple regular) modules ${\mathbb U}$, and for
each  $S\in\mathbb U$, we denote by $S[m]$  the
module of regular length $m$ on the ray $$S=S[1]\subset
S[2]\subset \dotsb\subset S[m]\subset S[m+1]\subset\dotsb$$ and let $S_\infty=\varinjlim S[m]$ be the corresponding   {\em Pr\"ufer
module}. The   {\em adic module}   $S_{-\infty}$ corresponding to $S\in\mathbb U$ is defined dually as the inverse limit along the coray ending at $S$.
If the field $k$ is algebraically closed,
the elements of ${\mathbb U}$ can be
 identified with points in the projective line $\mathbb{P}^1(k)$.
 
 Observe that the dual of a Pr\"ufer right $A$-module is the corresponding adic left $A$-module. Moreover, viewed in $\DModA$, the Pr\"ufer modules occur as cones of homological ring epimorphisms. Indeed, the universal localization of $A$ at (the projective presentations of the quasi-simple modules in)  a subset $\Ucal$ of ${\mathbb U}$ gives rise to a short exact sequence $0\to A\stackrel{\lambda}{\to} A_\Ucal\to  \bigoplus_{S \in \mathcal{U}}S_{\infty} \to 0$ in $\ModA$. In particular, when $\Ucal=\mathbb U$, we obtain  a short exact sequence $0\to A\stackrel{\lambda}{\to} A_{\mathbb U}\to  \bigoplus_{S \in \mathbb{U}}S_{\infty} \to 0$ where the right $A$-module $A_{\mathbb U}\cong G^2$ is isomorphic to a direct sum of two copies of the generic module $G$, 
and
$T=G \oplus \bigoplus_{S \in \mathbb{U}}S_{\infty}$ is (equivalent to) the tilting module arising from $\lambda$ as in Example~\ref{copreco}(2).

Let us  review the classification of cosilting classes in $\AMod$. 
We already know from  \cite[Corollary 3.8]{AH} that every
cosilting class is of cofinite type, and therefore any cosilting left $A$-module is equivalent to the dual of a right silting $A$-module. Then the dual version of \cite[Example 5.18]{AMV2} gives a complete classification of cosilting classes in $\AMod$, cf.~also Example~\ref{Kronecker}. In particular, the only cosilting class that is not minimal is $\Cogen(W)$, where $W = L^+$ is the dual of the Lukas tilting module in $\ModA$. This cosilting class induces the torsion pair $(\Add(\mathbf{q}),\Cogen(W))$ in $\AMod$. 

\begin{lemma}\label{kronecker}
	Let  $(\Ucal,\Vcal)$ be a compactly generated t-structure in $\DAMod$ and let $(\Ccal_n \mid n \in \Z)$ be the increasing chain of cosilting classes   obtained by setting $\Ccal_n = \Cogen(\Vcal_n)$.
Suppose that $\Ccal_l = \Cogen(W)$ for some $l \in \Z$. Then 		 $\DD^{>l} \subseteq\Vcal \subseteq \DD^{\geq l}$.\end{lemma}
\begin{proof}
We fix	a set of representatives $\mathbf{q} = \{Q_0,Q_1,Q_2,\ldots\}$ of all indecomposable preinjective left $A$-modules. 
	Since $\Cogen(W) = \Ucal_l^{\perp_0}$ by Theorem~\ref{T:aislecohomology}, it follows that $\Ucal_{l+1} \subseteq \Ucal_l \subseteq \Add(\mathbf{q})$. As every module from $\mathbf{q}$ has a local endomorphism ring, $\Add(\mathbf{q})$ consists (up to isomorphism) of direct sums of copies of objects of $\mathbf{q}$. Since both $\Ucal_l$ and $\Ucal_{l+1}$ are closed under direct summands and direct sums, these two subcategories are determined by the objects from the set $\mathbf{q}$ they contain. Recall that $\Vcal_l = \Ucal_l^{\perp_0} \cap \Ucal_{l+1}^{\perp_1} \subseteq \Ucal_{l+1}^{\perp_{0,1}}$. If $\Ucal_{l+1} \neq 0$, then it contains at least one object from $\mathbf{q}$, say $Q_k$. Therefore, $\Vcal_l \subseteq Q_k^{\perp_{0,1}} = \Add(Q_{k+1})$ (see \cite[Example 5.18]{AMV2}). But $\Vcal_l \subseteq \Ccal_l = \Cogen(W)$ contains no non-zero preinjective, which forces $\Vcal_l = 0$, a contradiction with $\Cogen(\Vcal_l) = \Cogen(W) \neq 0$. Therefore, necessarily $\Ucal_{l+1} = 0$. Since $\Vcal_n = \Ucal_n^{\perp_0} \cap \Ucal_{n+1}^{\perp_1} = \AMod$ for any $n > l$, we proved that 		 $\DD^{>l} \subseteq\Vcal$.

		To prove that $\Vcal \subseteq \DD^{\geq l}$, notice that $0\not=\Ucal_l\subset \Add\mathbf q$ must contain  $Q_j$ for some  $j \geq 0$. Then for all $k<l$ we have that $\Vcal_k = \Ucal_k^{\perp_0} \cap \Ucal_{k+1}^{\perp_1} \subseteq  \Ucal_{l}^{\perp_1} \subseteq Q_j^{\perp_1}$. By the Auslander-Reiten formula  $Q_j^{\perp_1}={}^{\perp_0}Q_{j+2}$, since we have an  almost split sequence 
		$$0 \rightarrow Q_{j+2} \rightarrow Q_{j+1}^{2} \rightarrow Q_{j} \rightarrow 0.$$
 On the other hand, the modules in $\Vcal_k \subseteq \Ccal_l=\Cogen(W)$ can't have summands isomorphic to one of $Q_0,Q_1,\ldots,Q_{j+1}$ and are therefore cogenerated by $Q_{j+2}$. This shows that $\Vcal_k=0$ for all $k<l$.
\end{proof}
Recall from  Theorem~\ref{hereditaryutc} that homotopically smashing t-structures in $\DAMod$ are precisely the compactly generated ones, and pure-injective cosilting objects are precisely the ones of cofinite type. We are now ready for the first classification result.
\begin{theorem}\label{kroneckertstr}
	Let  $A$ be the Kronecker algebra over a field $k$. The following is a complete list of homotopically smashing  t-structures in $\DAMod$:
		\begin{enumerate}
			\item[(i)] t-structures obtained from an increasing chain of homological epimorphisms via Proposition~\ref{construction},
			\item[(ii)] shifts of the Happel-Reiten-Smal\o~ t-structure induced by the torsion pair $(\Add(\mathbf{q}),\Cogen(W))$.
		\end{enumerate}
\end{theorem}
\begin{proof}
	 If the corresponding TTF triple is minimal, then we are in case (i) by Theorem~\ref{minimalTTF}. 
Otherwise 
	  there is an integer $l$ such that $\Ccal_l$ is not minimal. Then $\Ccal_l = \Cogen(W)$, as that is the unique non-minimal cosilting class in $\AMod$. So Lemma~\ref{kronecker} applies and $(\Ucal,\Vcal)$ is the Happel-Reiten-Smal\o~ t-structure induced by the torsion pair $(\Add(\mathbf{q}),\Cogen(W))$ shifted to degree $l$.
\end{proof}

According to Propositions~\ref{constructioncomplex} and~\ref{P:siltconstr}, the cosilting t-structures, and the silting t-structures are determined by the chains of homological ring epimorphisms  $(\lambda_n)$ satisfying the condition  (\ref{zeros}), and (\ref{dualzeros}), respectively.  We are now going to see that over the Kronecker algebra these conditions can both be rephrased by saying that the chain $(\lambda_n)$  has meet $A\to 0$ and join ${\rm id}_A$.  

We recall the shape of the lattice of homological ring epimorphisms from   \cite[Example 5.19]{AMV2}
$$\xymatrix@=0.7cm{&&&& \ {\rm id}_A \ \ar@{-}[lllld]\ar@{-}[llld]\ar@{-}[lld]\ar@{-}[rrrrd]\ar@{-}[rrrd]\ar@{-}[rrd]\ar@<-2ex>@{-}[d]^{\ ...}\ar@<-1.6ex>@{-}[d]\ar@<-1.2ex>@{-}[d]\ar@<1.2ex>@{-}[d]\ar@<1.6ex>@{-}[d]\ar@<2ex>@{-}[d]\\ \lambda_0\ar@{-}[rrrrdddd]&\lambda_1\ar@{-}[rrrdddd]&\lambda_2\ar@{-}[rrdddd] & ... & *+[F]{\{\lambda_x| x\in \mathbb U\}}\ar@{--}@/_1pc/[dd]^{...}\ar@{--}[dd]\ar@{--}@/_2pc/[dd]^{...}\ar@{--}@/_3pc/[dd]^{...}\ar@{--}[dd]\ar@{--}@/^3pc/[dd]_{...}\ar@{--}@/^2pc/[dd]_{...}\ar@{--}@/^1pc/[dd]_{...} & ... & \mu_2\ar@{-}[lldddd] & \mu_1\ar@{-}[llldddd] &\mu_0\ar@{-}[lllldddd] \\ \\ &&&& *+[F]{\{\lambda_{\mathbb U\setminus \{x\}}| x\in \mathbb U\}}\ar@<-2ex>@{-}[d]^{\ ...}\ar@<-1.6ex>@{-}[d]\ar@<-1.2ex>@{-}[d]\ar@<1.2ex>@{-}[d]\ar@<1.6ex>@{-}[d]\ar@<2ex>@{-}[d]\\ &&&& *+[F]{\lambda_{\mathbb U}}\ar@{-}[d]\\ &&&&0}$$
where the interval between ${\rm id}_A$ and $\lambda_{\mathbb U}$ represents the dual poset of subsets of $\mathbb U$. The ring epimorphisms with infinite dimensional target are those in frames, that is, those of the form $\lambda_\Ucal$ with $\emptyset\neq\Ucal\subseteq \mathbb U$.  The remaining ring epimorphisms are universal localizations at indecomposable preprojective or preinjective modules; their targets, viewed as $A$-modules, are preprojective or preinjective, and as rings they are all Morita equivalent to $k$.

\begin{proposition}\label{L:kroneckerduality}
	Let $A$ be a hereditary ring. Consider a chain 		$\cdots  \leq \lambda_n \leq \lambda_{n+1} \leq \cdots$
		 of homological ring epimorphisms $\lambda_n:A\to B_n$, and denote by $\Xcal_n$ and $\Xcal_n'$ the corresponding bireflective subcategories of $\AMod$ and $\ModA$, respectively.
		 Moreover, define as above the subcategories
		 $\Lcal_n = \Ker(B_n \otimes_A^\LL -)$ of $\DAMod$ and $\Kcal_n=\Ker {\mathbf R} \Hom{A}{B_n}{-}$ of $\DModA$. 
\begin{enumerate}
\item $\bigcap_{n \in \Z} \Xcal_n' = 0$ if and only if  $\bigcap_{n \in \Z} {\Xcal}_n = 0$, and this means precisely that the meet $\bigwedge_{n \in \Z}\lambda_n$  equals the trivial ring epimorphism $A\to 0$.
\item $\bigcap_{n \in \Z} \Kcal_n = 0$ implies $\bigcap_{n \in \Z} {\Lcal}_n = 0$, which in turn implies that the join $\bigvee_{n \in \Z}\lambda_n$  equals ${\rm id}_A$.
\item If $A$ is the Kronecker algebra,  then $\bigcap_{n \in \Z} {\Lcal}_n = 0$ if and only if $\bigcap_{n \in \Z} \Kcal_n = 0$, and this means precisely that the join $\bigvee_{n \in \Z}\lambda_n$  equals ${\rm id}_A$. 
\end{enumerate}		
\end{proposition}
\begin{proof}
(1) is clear.

(2) The first implication follows from Lemma~\ref{L:nondegengeneral} (or by checking directly using the duality $(-)^+$).
Moreover, by Theorem~\ref{KS}, the condition $\bigcap_{n \in \Z}\Lcal_n = 0$ is equivalent to  $\bigcap_{n \in \Z}{}^{\perp_{0,1}}\Xcal_n = {}^{\perp_{0,1}}(\bigcup_{n \in \Z}\Xcal_n) = 0$, which implies that the join $\bigvee_{n \in \Z}\lambda_n$  equals ${\rm id}_A$. To see the latter, recall from Theorem~\ref{univlochom} that  $\lambda_n$ coincides with the universal localization at the wide subcategory $\Mcal_n={}^{\perp_{0,1}}\Xcal_n\cap\lfmod{A}$ of $\lfmod{A}$, and that $\Xcal_n=\Mcal_n^{\perp_{0,1}}$. Now it is easy to check that the join $\bigvee_{n \in \Z}\lambda_n$   is  the universal localization at $\bigcap_{n \in \Z}\Mcal_n$, 
which is contained in $ {}^{\perp_{0,1}}(\bigcup_{n \in \Z}\Xcal_n)$.

(3) We have to verify that $\bigcap_{n \in \Z} \Kcal_n = 0$ whenever the join $\bigvee_{n \in \Z}\lambda_n$  equals ${\rm id}_A$. 
Observe that $\bigcap_{n \in \Z} \Kcal_n = \bigcap_{n \in \Z}\Xcal_n'\,^{\perp_{\mathbb Z}}= ( \bigcup_{n \in \Z}\Xcal_n')\,^{\perp_{\mathbb Z}}$.

Now we have two cases. In the first case, the chain involves only a finite number of different ring epimorphisms. Then the join coincides with the largest member of the chain, hence  $\bigcup_{n \in \Z}\Xcal_n'=\ModA$, and the claim is proven.

In the second case, there is a chain of strictly decreasing subsets
						$\mathbb U \supseteq \mathcal{U}_0 \supset \mathcal{U}_1 \supset \mathcal{U}_2 \supset \cdots$
						of the representative set $\mathbb U$ of all quasi-simples such that 						  $\lambda_n: A \rightarrow B_n$ is the universal localization at the (projective presentations of the simple regular modules in the) set $\mathcal{U}_n$ for each $n \in \omega$. 
Now one can see, e.g.~from the shape of the lattice, that for any simple regular left $A$-module $S$ there is $n \in \Z$ such that $S \in {\Xcal}_n'$, and moreover, that all $\lambda_n$ lie above the universal localization at $\mathbb U$, which means precisely that $G\in\Xcal_n'$. It follows that $\Loc(\bigcup_{n \in \Z} \Xcal_n')$ contains the right $A$-module $T=G \oplus \bigoplus_{S \in \mathbb{U}}S_{\infty}$, which is a right tilting $A$-module and therefore satisfies $T^{\perp_{\mathbb Z}}=0$. This shows the claim.						
\end{proof}

 As a consequence, we get the following classification of the pure-injective cosilting objects, and dually, of silting objects of finite type. 

\begin{theorem}\label{kronecosilt} Let $A$ be the Kronecker algebra over a field $k$. Every {pure-injective} cosilting object  in $\DAMod$
 arises from a chain of homological ring epimorphisms   with meet $\;0_A: A\to 0$  and join  $\;{\rm id}_A: A\to A$, or
 it is equivalent to a shift of the cotilting module $W$.
 
The following is 
 a complete list of all pure-injective cosilting objects,  up to equivalence:
\begin{itemize}
	\item[(i)] shifts of the non-minimal cotilting module $W$;
	
\item[(ii)] for any finitely dimensional homological epimorphism $\lambda: A \rightarrow B$, and for all integers $l \leq m$, the cosilting object
			$$C = B^+[-l] \oplus \Cone(\lambda)^+[-m]$$
				with induced t-structure arising from the following chain  of  bireflective subcategories of $\AMod$:
				$$\Xcal_n = \begin{cases} 0 & n<l \\ \Xcal_B & l \leq n \leq m \\ \AMod & l>m; \end{cases}$$
				
\item[(iii)] for any $l \in \Z$, and any chain of subsets 						$\mathbb{U} \supseteq \mathcal{U}_0 \supseteq \mathcal{U}_1 \supseteq \mathcal{U}_2 \supseteq \cdots$
						such that $\bigcap_{n \in \omega}\mathcal{U}_n = \emptyset$, the cosilting object
						$$C = B_0^+[-l] \oplus \prod_{n \in \omega}(\prod_{S \in \mathcal{U}_n \setminus \mathcal{U}_{n+1}}S_{-\infty}[-n-l]),$$
where $\lambda_n: A \rightarrow B_n$ denotes the universal localization at the set $\mathcal{U}_n$, and the induced t-structure  arises from the following chain  of  bireflective
 subcategories of $\AMod$:
				$$\Xcal_n = \begin{cases} 0 & n<l \\ \Xcal_{B_{n-l}} & l \leq n. \end{cases}$$
\end{itemize}
\end{theorem}

\begin{theorem}\label{kronesilt}
	Let $A$ be the Kronecker algebra over an algebraically closed field $k$. The assignment $T \mapsto T^+$ induces a bijection between 
	\begin{enumerate}
			\item  equivalence classes of silting objects  of finite type in $\DModA$,
	\item  equivalence classes of pure-injective cosilting objects in $\DAMod$. 
	\end{enumerate}
	Every silting object of finite type in $\DModA$ 
arises from a chain of homological ring epimorphisms   with meet $\;0_A: A\to 0$  and join  $\;{\rm id}_A: A\to A$,  or 
 it	is  equivalent to a shift of the Lukas tilting module $L$.
	
The following is 
 a complete list of all silting objects of finite type,  up to shift and equivalence:
 \begin{itemize}
 \item[(i)]  the  non-minimal tilting module $L$;
\item[(ii)] for any  finite-dimensional homological epimorphism $\lambda:A\to B$,  the silting object
$$B\oplus \Cone(\lambda)[m];$$ 
\item[(iii)]  for any chain of subsets 						$\mathbb{U} \supseteq \mathcal{U}_0 \supseteq \mathcal{U}_1 \supseteq \mathcal{U}_2 \supseteq \cdots$
						such that $\bigcap_{n \in \omega}\mathcal{U}_n = \emptyset$, the silting object $$B_0\oplus \bigoplus\limits_{n\in\omega}(\bigoplus\limits_{S\in\mathcal{U}_n \setminus \mathcal{U}_{n+1}}S_{\infty}[n]),$$ where $\lambda_n: A \rightarrow B_n$ denotes the universal localization at  $\mathcal{U}_n$.
\end{itemize}
\end{theorem}
\begin{proof}
	By the classification of pure-injective cosilting objects of $\DAMod$, any minimal cosilting object $C$
	is induced 
	by a chain of homological epimorphisms with meet $\;0_A: A\to 0$  and join  $\;{Id}: A\to A$. According to Propositions~\ref{L:YcgTTF} and~\ref{L:kroneckerduality}, this chain also gives rise to a silting object  of finite type in $\DModA$, which is the preimage of $C$ under the injective map $\Psi$ in Theorem~\ref{P:silttocosilt}. We infer  that the assignment $T \mapsto T^+$ induces a bijection between silting objects of finite type in $\DModA$ and pure-injective cosilting objects in $\DAMod$. The classification then follows from the classification of cosilting objects in $\DAMod$ and Proposition~\ref{P:siltconstr}.
\end{proof}

{\bf Acknowledgement.} The authors would like to thank the referee for careful reading of the manuscript and for many valuable suggestions. 


\end{document}